\documentclass[a4paper]{amsart}

\usepackage[utf8]{inputenc}
\usepackage{enumitem}
\usepackage{setspace}
\usepackage{mathtools}
\usepackage{tikz-cd}
\usepackage{lipsum}
\usepackage{tikz}
\usepackage{mathrsfs}
\usetikzlibrary{arrows}
\usepackage{amssymb}
\usetikzlibrary{decorations.pathmorphing}
\usepackage{bbold}
\usepackage{amsmath,amsthm}
\usepackage{xspace}
\usepackage{faktor}
\usepackage{calrsfs}
\usepackage{verbatim}
\usepackage{quiver}
\usepackage[pdfa]{hyperref}
\usepackage{amsfonts}
\usepackage{graphicx}
\usepackage{eucal}
\usepackage{amscd}
\usepackage[all,2cell]{xy}
\UseAllTwocells

\newtheorem{teo}{Theorem}[section]
\newtheorem{cor}[teo]{Corollary}
\newtheorem{lem}[teo]{Lemma}
\newtheorem{defi}[teo]{Definition}
\newtheorem{oss}[teo]{Remark}
\newtheorem{prop}[teo]{Proposition}
\newtheorem{example}[teo]{Example}

\newdir{ m>}{{}*!/-3.5pt/@{}*!/-8pt/:(1,-.2)@^{>}*!/-8pt/:(1,+.2)@_{>}}

\def\pullback{
	\ar@{-}[]+R+<6pt,-1pt>;[]+RD+<6pt,-6pt>%
	\ar@{-}[]+D+<1pt,-6pt>;[]+RD+<6pt,-6pt>}

\newcommand{\zeros}{\mathcal{Z}}
\newcommand{\zideal}{\mathcal{N}_{\zeros}}
\newcommand{\op}[1]{#1^{\mathrm{op}}}
\newcommand{\tort}{\mathcal{T}}
\newcommand{\torf}{\mathcal{F}}
\newcommand{\torsione}{(\tort, \torf)}

\newcommand{\facte}{\mathbb{E}}
\newcommand{\factm}{\mathbb{M}}
\newcommand{\frecce}{\Arr(\mathcal{C})}
\newcommand{\mset}{\mathbf{MSet}}
\newcommand{\terminal}{\boldsymbol{1}}
\newcommand{\initial}{\boldsymbol{2}}
\newcommand{\iniziale}{\boldsymbol{0}}
\newcommand{\tbang}[1]{\tau_{#1}}
\newcommand{\ibang}[1]{\iota_{#1}}
\newcommand{\dn}[1]{#1_{\lnot \lnot}}
\newcommand{\boole}{\mathbf{Boole}}
\newcommand{\heyting}{\mathbf{Heyt}}
\newcommand\myfunc[5]{\begin{aligned}
		#1 \colon #2 &\to #3\\
		#4 &\mapsto #5
\end{aligned}}
\newcommand{\pd}{\mathbf{PD}}
\newcommand{\sset}{\mathbf{sSet}}
\newcommand{\set}{\mathbf{Set}}
\newcommand{\Ab}{\mathbf{Ab}}

\newcommand{\Divm}{\mathbf{Div}_m}
\newcommand{\coslice}{\mathbb{Z}_m/\mathbf{Ab}}
\newcommand{\Zm}{\mathbb{Z}_m}
\newcommand{\Zh}{\mathbb{Z}_h}
\newcommand{\graffe}[1]{\{ #1 \}}
\newcommand{\spigolose}[1]{\langle #1 \rangle}
\newcommand{\suc}[1]{(#1)_{n \in \mathbb{N}}}
\newcommand{\ped}[1]{_{#1}}
\newcommand{\Dm}{\D_m}
\newcommand{\Homcos}{\Hom\ped{\coslice}}
\newcommand{\Homab}{\Hom\ped{\Ab}}
\newcommand{\invf}[1]{f^{-1}(#1)}
\newcommand{\inve}{e^{-1}(0)}
\newcommand{\invm}{m^{-1}(0)}

\newcommand{\mv}{\mathbf{MV}}

\newcommand{\Zum}{\mathbb{Z}[1/m]}

\DeclareMathOperator{\Hom}{Hom}
\DeclareMathOperator{\Arr}{Arr}
\DeclareMathOperator{\Zker}{\zeros-ker}
\DeclareMathOperator{\Zcoker}{\zeros-coker}

\DeclareMathOperator{\Eq}{Eq}
\DeclareMathOperator{\Sub}{Sub}
\DeclareMathOperator{\Fix}{Fix}
\DeclareMathOperator{\D}{D}
\DeclareMathOperator{\ord}{ord}
\DeclareMathOperator{\Rad}{Rad}
\DeclareMathOperator{\Inf}{Inf}

\def\mathsfdef#1{\expandafter\def\csname#1\endcsname{{\sf#1}}}
\mathsfdef{Z}
\mathsfdef{F}
\mathsfdef{T}
\mathsfdef{S}
\mathsfdef{C}
\mathsfdef{i}
\mathsfdef{P}
\mathsfdef{D}
\mathsfdef{j}
\mathsfdef{V}

\begin{document}
	
	\title[Torsion Theories in a non-pointed Context]{Torsion Theories in a non-pointed Context}
	
	\author[A. Cappelletti]{Andrea Cappelletti}
	\address[Andrea Cappelletti]{Dipartimento di Matematica, Universit\`{a} degli Studi di Salerno, Via Giovanni Paolo II 132, 84084 Fisciano (SA), Italy}
	\thanks{}
	\email{acappelletti@unisa.it}
	
	\author[A. Montoli]{Andrea Montoli}
	\address[Andrea Montoli]{Dipartimento di Matematica ``Federigo Enriques'', Universit\`{a} degli Studi di Milano, Via Saldini 50, 20133 Milano, Italy}
	\thanks{} \email{andrea.montoli@unimi.it}
	
	
	
	\begin{abstract}
		We study a non-pointed version of the notion of torsion theory in the framework of categories equipped with a posetal monocoreflective subcategory such that the coreflector inverts monomorphisms. We explore the connections of such torsion theories with factorization systems and categorical Galois structures. We describe several examples of these torsion theories, in the dual of elementary toposes, in varieties of universal algebras used as models for non-classical logic, and in coslices of the category of abelian groups.
	\end{abstract}
	
	\date{\today}
	
	\maketitle
	
	\section{Introduction}
	The study of torsion theories has been initiated in \cite{Dickson} in the context of abelian categories, and it was extended in \cite{JT torsion th} to the setting of pointed categories. A torsion theory in a pointed category $\mathcal{C}$ is a pair $(\tort, \torf)$ of full, replete subcategories of $\mathcal{C}$ such that every morphism $T \to F$, with $T \in \tort$ and $F \in \torf$, factors through $0$ and, moreover, for every object $C \in \mathcal{C}$ there exists a short exact sequence
	\[\begin{tikzcd}
		T & C & F
		\arrow["k", from=1-1, to=1-2]
		\arrow["q", from=1-2, to=1-3]
	\end{tikzcd}\]
	with $T \in \tort$ and $F \in \torf$. Here by short exact sequence we mean that $k$ is a kernel of $q$ and $q$ is a cokernel of $k$.
	
	The classical notion of torsion theory for abelian categories has strong connections with the ones of closure operator, radical, and factorization system. Similar connections still exist in the pointed, non-abelian case. In \cite{BG torsion th} torsion theories are studied in the context of homological \cite{Bbbook} categories, namely pointed, regular, protomodular \cite{Bourn protomod} categories. The name \emph{homological} categories is justified by the fact that, in these categories, non-abelian versions of the classical homological lemmas hold. In \cite{BG torsion th} the authors explore the connections between torsion theories, closure operators and radicals in homological categories. In \cite{Gran cond N} the relation between torsion theories and factorization systems is studied in normal \cite{Zurab normal cats} categories, i.e.\ pointed regular categories in which every regular epimorphism is a cokernel (every homological category is normal). In \cite{Gran torsione} the authors restrict their attention to those torsion theories (in homological categories) in which the reflector of the torsion-free subcategory is protoadditive \cite{everaert2010homology}. A functor between pointed protomodular categories is protoadditive if it preserves split short exact sequences; equivalently, protoadditive functors can be defined as functors preserving the zero object and pullbacks of split epimorphisms along any morphism. This notion is the analogue of the one of additive functor between additive categories. In \cite{Gran torsione} the authors show, among other things, that a torsion theory with protoadditive reflector allows for an easy description of the central and normal extensions w.r.t.\ the Galois structure associated with the reflection. This permits to use the tools of categorical Galois theory (in the sense of \cite{janelidze1990pure}) to describe homology objects relative to such reflection.
	
	More recently, torsion theories have been studied also in non-pointed contexts. In \cite{GrandisJan torsion th} the authors proposed a definition of torsion theory, suitable also for non-pointed categories, based on the notion of ideal of morphisms in the sense of \cite{Ehresmann}. Given an ideal of morphisms $\mathcal{N}$, a torsion theory in their sense is a pair $(\tort, \torf)$ of full, replete subcategories of a category $\mathcal{C}$ such that every morphism $T \to F$, with $T \in \tort$ and $F \in \torf$, belongs to $\mathcal{N}$ and, moreover, for every object $C \in \mathcal{C}$ there exists a short exact sequence
	\[\begin{tikzcd}
		T & C & F
		\arrow["k", from=1-1, to=1-2]
		\arrow["q", from=1-2, to=1-3]
	\end{tikzcd}\]
	with $T \in \tort$ and $F \in \torf$, where now short exact sequence means that $k$ is an $\mathcal{N}$-kernel of $q$ and $q$ is an $\mathcal{N}$-cokernel of $k$. The notions of $\mathcal{N}$-kernel and $\mathcal{N}$-cokernel w.r.t.\ an ideal of morphisms $\mathcal{N}$ are the natural generalizations of the usual notions of kernel and cokernel in pointed categories, see \cite{Ehresmann2}. In \cite{GrandisJan torsion th} the authors required that $\mathcal{N}$-kernels and $\mathcal{N}$-cokernels of identity morphisms exist. In \cite{Facchini1, Facchini2} the authors consider essentially the same notion of non-pointed torsion theory as in \cite{GrandisJan torsion th}, but without the requirement that $\mathcal{N}$-kernels and $\mathcal{N}$-cokernels exist. They call such torsion theories \emph{pretorsion theories}. In their case, the ideal of morphisms is given by those arrows that factor through the subcategory $\zeros = \tort \cap \torf$.
	
	The aim of this paper is to extend to this non-pointed context some of the results of \cite{Gran cond N} and \cite{Gran torsione}, and to provide several new examples of non-pointed torsion theories. Our framework will be the one of a category $\mathcal{C}$ equipped with a posetal monocoreflective subcategory $\zeros$ such that the coreflector inverts monomorphisms. The objects of $\zeros$ will be considered the \emph{trivial} objects, and we will study torsion theories with respect to the ideal of morphisms factoring through $\zeros$. The main general example of our situation will be given by a regular category $\mathcal{C}$ with initial object $0$, equipped with the subcategory $\zeros$ of regular quotients of $0$. This includes, as concrete examples, all varieties of universal algebras with at least one constant, their topological models, all elementary toposes and their duals, and every ideally exact category in the sense of \cite{Janelidze ideally exact}. Categories equipped with such a subcategory $\zeros$ have been considered in \cite{Cappelletti} to prove non-pointed versions of the classical homological lemmas. In our setting, $\mathcal{N}$-cokernels (that we will call $\zeros$-cokernels, since the ideal of morphisms is determined by the class of objects $\zeros$) do not always exist, not even those of the identity arrows. So, our torsion theories will not be, strictly speaking, examples of the situation considered in \cite{GrandisJan torsion th}, but they will be instances of the definition stated in \cite{Facchini2}.
	
	In Section \ref{preliminaries} we describe our setting and recall the non-pointed notion of torsion theory. In Section \ref{torsion th and fact sys} we establish, in the context of regular protomodular categories equipped with a posetal monocoreflective subcategory $\zeros$ such that the coreflector inverts monomorphisms, a correspondence between factorization systems $(\facte, \factm)$ that are stable w.r.t.\ the class of arrows inverted by the coreflector of $\zeros$ and such that every arrow in $\facte$ is a $\zeros$-cokernel, and torsion theories satisfying suitable conditions. In Section \ref{torsion th and Galois} we study which conditions are needed on a torsion theory to have a good behavior of the Galois structure induced by the reflection of the torsion-free subcategory. Section \ref{examples} is devoted to the description of several examples of non-pointed torsion theories, as well as their corresponding factorization systems and Galois structures. Some of the examples live in varieties of universal algebras used as models for non-classical logic, like MV-algebras or Heyting algebras; others are considered in the dual categories of some toposes; the last one lives in coslice categories of the category of abelian groups.
	
	\section{Preliminaries} \label{preliminaries}
	
	We begin by introducing the context in which we will work, already considered in \cite{Cappelletti}.
	
	\begin{prop}[\cite{Cappelletti}, Proposition 2.1]\label{contesto generale}
		Consider a category $\mathcal{C}$ and a full subcategory $\zeros$. The following conditions are equivalent:
		\begin{itemize}
			\item[(a)] $\zeros$ is a posetal monocoreflective subcategory such that the coreflector inverts monomorphisms;
			\item[(b)] \begin{itemize}
				\item[(i)] for every $A \in \mathcal{C}$ there exists a, unique up to isomorphisms, monomorphism $\varepsilon_A \colon Z \rightarrowtail A$, with $Z \in \zeros$;
				\item[(ii)] for every $A \in \mathcal{C}$ and $Z \in \zeros$, $\Hom(Z,A)$ has at most one element;
				\item[(iii)] for every pair of arrows $z \colon Z \rightarrowtail A$ and $z' \colon Z' \to A$, where $z$ is a monomorphism and $Z,Z' \in \zeros$, there exists a unique morphism $\varphi \colon Z' \to Z$ such that $z\varphi=z'$:
				\[\begin{tikzcd}
					{Z'} & A \\
					Z.
					\arrow["{z'}", from=1-1, to=1-2]
					\arrow["{\exists ! \varphi}"', from=1-1, to=2-1]
					\arrow["z"', tail, from=2-1, to=1-2]
				\end{tikzcd}\]
			\end{itemize}
		\end{itemize}
	\end{prop}
	We call such a subcategory $\zeros$ \emph{class of zero objects}, we denote by $\Z$ the coreflector and, given an object $A$ of $\mathcal{C}$, we will say that $\Z(A)$ is the \emph{zero part} of $A$. We denote by $\mathcal{N}_{\zeros}$ the class of arrows of $\mathcal{C}$ factorizing through an object of $\zeros$. $\mathcal{N}_{\zeros}$ is clearly an ideal of morphisms in the sense of \cite{Ehresmann}. \\
	
	As observed in \cite[Remark 2.2]{Cappelletti}, such a class $\zeros$, when it exists, is unique. So, admitting a class of zero objects in our sense is a property of the category $\mathcal{C}$. We will be particularly interested in the case in which $\mathcal{C}$ is a regular category with initial object and $\zeros$ is the class of regular quotients of the initial object. Examples of this situation are all varieties of universal algebras with at least one constant, their topological models, all elementary toposes and their duals, and every ideally exact category in the sense of \cite{Janelidze ideally exact}.
	
	\begin{lem}\label{mono e triviali}
		Let $\mathcal{C}$ be a category with a fixed class $\zeros$ of zero objects. Consider an arrow $f \colon A \rightarrow B$ and a monomorphism $m \colon B \rightarrowtail C$. If $mf \in \zideal$ then $f \in \zideal$.
	\end{lem}
	\begin{proof}
		Recall that $\Z(m)$ is an isomorphism. Thanks to \cite[Remark 2.3]{Cappelletti} we have the following factorization of $mf \in \zideal$:
		\[\begin{tikzcd}
			A & B & C \\
			& {\Z(C).}
			\arrow["f", from=1-1, to=1-2]
			\arrow["\chi"', from=1-1, to=2-2]
			\arrow["m", tail, from=1-2, to=1-3]
			\arrow["{\varepsilon_C}"', tail, from=2-2, to=1-3]
		\end{tikzcd}\]
		By naturality, $\varepsilon_C \Z(m)=m \varepsilon_B$. We observe that $m \varepsilon_B \Z(m)^{-1} \chi = \varepsilon_C \Z(m) \Z(m)^{-1} \chi = \varepsilon_C \chi = mf$ and, since $m$ is a monomorphism, we obtain $f=\varepsilon_B \Z(m)^{-1} \chi \in \zideal$.
	\end{proof}
	Let $\mathcal{C}$ be a regular category equipped with a full subcategory $\zeros$ satisfying the equivalent conditions of Proposition \ref{contesto generale}. Since the class $\mathcal{N}_{\zeros}$ of morphisms that factor through $\zeros$ is an ideal, we can consider kernels and cokernels with respect to it, as in \cite{Ehresmann2}:
	
	\begin{defi}
		Let $\mathcal{C}$ be a category with a fixed class $\zeros$ of zero objects. Let $f \colon A \to B$ be a morphism in $\mathcal{C}$. We say that a morphism $k \colon K \to A$ in $\mathcal{C}$ is
		a \emph{$\zeros$-kernel} of $f$ if the following properties are satisfied:
		\begin{itemize}
			\item[(a)] $fk \in \mathcal{N}_{\zeros}$;
			\item[(b)] whenever $e: E \to A$ is a morphism in $\mathcal{C}$ and $fe \in \mathcal{N}_{\zeros}$, then there exists a unique morphism $\varphi \colon E \to K$ in $\mathcal{C}$ such that $k \varphi = e$.
		\end{itemize}
	\end{defi}
	The definition of a $\zeros$-cokernel is dual. $\zeros$-kernels and $\zeros$-cokernels, when they exist, are unique up to isomorphisms. We will denote the $\zeros$-kernel and the $\zeros$-cokernel of a morphism $f$ by $\Zker(f)$ and $\Zcoker(f)$, respectively.
	
	\begin{defi}[\cite{Facchini1}]
		Let $\mathcal{C}$ be a category with a fixed class $\zeros$ of zero objects. Let $f \colon A \rightarrow B$ and $g \colon B \rightarrow C$ be morphisms in $\mathcal{C}$. We say that 
		\[\begin{tikzcd}
			A & B & C
			\arrow["f", from=1-1, to=1-2]
			\arrow["g", from=1-2, to=1-3]
		\end{tikzcd}\]
		is a \emph{short $\zeros$-exact sequence} (or simply short exact sequence, when there is no ambiguity) in $\mathcal{C}$ if $f$ is a $\zeros$-kernel of $g$ and $g$ is a $\zeros$-cokernel of $f$.
	\end{defi}
	
	Observe that this definition of exact sequence is different from the one we considered in \cite{Cappelletti}, where $g$ is only required to be a regular epimorphism. \\
	
	As shown in \cite{Cappelletti}, in any regular category $\mathcal{C}$ (in fact, it is sufficient for $\mathcal{C}$ to have pullbacks) with a fixed class of zero objects, every morphism $f$ admits a $\zeros$-kernel (obtained as the pullback of the inclusion of the zero part of the codomain of $f$ along $f$):
	
	\begin{prop}[\cite{Cappelletti}, Proposition 4.2]
		Let $\mathcal{C}$ be a regular category with a fixed class $\zeros$ of zero objects. For every arrow $f \colon A \to B$ in $\mathcal{C}$ the $\zeros$-kernel $k \colon K \to A$ of $f$ exists and it is given by the pullback
		\[\begin{tikzcd}
			K & {\Z(B)} \\
			A & B.
			\arrow["\chi", from=1-1, to=1-2]
			\arrow["k"', tail, from=1-1, to=2-1]
			\arrow["\lrcorner"{anchor=center, pos=0}, draw=none, from=1-1, to=2-2]
			\arrow["{\varepsilon_B}", tail, from=1-2, to=2-2]
			\arrow["f"', from=2-1, to=2-2]
		\end{tikzcd}\]
	\end{prop}
	
	Usually, we denote by $K[f]$ the domain of $k$.
	
	\begin{lem}\label{ker e mono}
		Let $\mathcal{C}$ be a category with a fixed class $\zeros$ of zero objects. Consider an arrow $f \colon A \rightarrow B$ and a monomorphism $m \colon B \rightarrowtail C$. Suppose that $\Zker(f)$ exists. Then, $\Zker(mf)$ exists and $\Zker(f) \cong \Zker(mf)$.
	\end{lem}
	
	\begin{proof}
		Since $m$ is a monomorphism, $\Z(m)$ is an isomorphism. Let $k$ denote the $\zeros$-kernel of $f$. Consider the commutative diagram
		\[\begin{tikzcd}
			{K[f]} & {\Z(B)} & {\Z(C)} \\
			A & B & C,
			\arrow["\chi", from=1-1, to=1-2]
			\arrow["k"', from=1-1, to=2-1]
			\arrow["\lrcorner"{anchor=center, pos=0}, draw=none, from=1-1, to=2-2]
			\arrow["{\Z(m)}", from=1-2, to=1-3]
			\arrow["\sim"', from=1-2, to=1-3]
			\arrow["{\varepsilon_B}"', from=1-2, to=2-2]
			\arrow["\lrcorner"{anchor=center, pos=0}, draw=none, from=1-2, to=2-3]
			\arrow["{\varepsilon_C}", from=1-3, to=2-3]
			\arrow["f"', from=2-1, to=2-2]
			\arrow["m"', tail, from=2-2, to=2-3]
		\end{tikzcd}\]
		where the square on the left is a pullback since $k$ is the $\zeros$-kernel of $f$, while the square on the right is a pullback since $\Z(m)$ is an isomorphism and $m$ is a monomorphism. Hence the whole rectangle is a pullback, and so $k=\Zker(mf)$.
	\end{proof}
	
	The situation concerning $\zeros$-cokernels is different: nothing guarantees their existence. For example, in the category $\mathbf{Ring}$ of unitary rings, with the zero class given by the quotients of the ring $\mathbb{Z}$ of integers, the identity on $\mathbb{Z} \times \mathbb{Z}$ does not have a $\zeros$-cokernel (see \cite{Cappelletti} for more details). However, when the $\zeros$-cokernel of a morphism $f$ does exist, it is given by the pushout, along $f$, of an appropriate morphism with the same domain as $f$ and codomain a zero object:
	
	\begin{prop}\label{cokernelpushout}
		Let $\mathcal{C}$ be a category with a fixed class $\zeros$ of zero objects. Suppose the $\zeros$-cokernel $p \colon B \rightarrow P$ of $f \colon A \rightarrow B$ exists. Then, the commutative square
		\[\begin{tikzcd}
			A & B \\
			Z & P,
			\arrow["f", from=1-1, to=1-2]
			\arrow["\chi"', from=1-1, to=2-1]
			\arrow["p", from=1-2, to=2-2]
			\arrow["z"', from=2-1, to=2-2]
		\end{tikzcd}\]
		is a pushout (where $Z \in \zeros$ and the factorization $z \chi$ of $pf$ exists since $pf \in \zideal$).
	\end{prop}
	
	\begin{proof}
		Consider $q \colon B \rightarrow Q$ and $x \colon Z \rightarrow Q$ such that $qf=x \chi$. Since $qf \in \zideal$, there exists a unique $\theta$ such that $\theta p = q$. Moreover, condition (ii) of Proposition \ref{contesto generale} implies $\theta z = x$ i.e.\ the square above is a pushout.
	\end{proof}
	
	Since every arrow with codomain in $\zeros$ is a strong epimorphism (as proved in \cite[Proposition 2.1]{Cappelletti}, it follows that every $\zeros$-cokernel is a strong epimorphism (as a pushout of a strong epimorphism).
	
	\begin{prop}
		Let $\mathcal{C}$ be a category with a fixed class $\zeros$ of zero objects. Consider an arrow $f \colon A \rightarrow B$ and suppose that $k=\Zker(f)$ and $q=\Zcoker(k)$ exist. Then, $k$ is the $\zeros$-kernel of $q$.
	\end{prop}
	
	\begin{proof}
		Clearly $qk \in \zideal$, since $q=\Zcoker(k)$. Let $k' \colon K' \rightarrow A$ be an arrow such that $qk' \in \zideal$. We consider the following commutative diagram
		\[\begin{tikzcd}
			K & A & B \\
			{K'} && Q,
			\arrow["k", from=1-1, to=1-2]
			\arrow["f", from=1-2, to=1-3]
			\arrow["q"', from=1-2, to=2-3]
			\arrow["{k'}"', from=2-1, to=1-2]
			\arrow["\varphi"', from=2-3, to=1-3]
		\end{tikzcd}\]
		where $fk \in \zideal$ implies the existence of $\varphi$. Since $fk'=\varphi q k' \in \zideal$, there exists a unique arrow $\theta \colon K' \rightarrow K$ such that $k \theta = k'$.
	\end{proof}
	
	In other words, this last result tells that every $\zeros$-kernel is the $\zeros$-kernel of its $\zeros$-cokernel (if it exists). In a similar way, it can be proved that every $\zeros$-cokernel is the $\zeros$-cokernel of its $\zeros$-kernel.
	
	\begin{lem}\label{coker id}
		Let $\mathcal{C}$ be a category with a fixed class $\zeros$ of zero objects. Suppose that $q \colon X \rightarrow Q$ is the $\zeros$-cokernel of $id_X$. Then, $Q \in \zeros$.
	\end{lem}
	
	\begin{proof}
		Applying Proposition \ref{cokernelpushout}, we know that the diagram below is a pushout
		\[\begin{tikzcd}
			X & X \\
			{\Z(Q)} & Q.
			\arrow[equal, from=1-1, to=1-2]
			\arrow["\chi"', from=1-1, to=2-1]
			\arrow["q", from=1-2, to=2-2]
			\arrow["{\varepsilon_Q}"', from=2-1, to=2-2]
			\arrow["\lrcorner"{anchor=center, pos=0, rotate=180}, draw=none, from=2-2, to=1-1]
		\end{tikzcd}\]
		Hence $\varepsilon_Q$ is an isomorphism since it is the pushout of $id_X$.
	\end{proof}
	
	\begin{defi}
		Let $\mathcal{C}$ be a category with a fixed class $\zeros$ of zero objects. An object $X \in \mathcal{C}$ is said to have a \emph{maximum quotient in $\zeros$} if there exists an arrow $x \colon X \rightarrow M(X)$, with $M(X) \in \zeros$, such that for every arrow $z \colon X \rightarrow Z$, with $Z \in \zeros$, there exists a (necessarily unique) arrow $\varphi \colon M(X) \rightarrow Z$ such that $\varphi x = z$.
	\end{defi}
	
	If $X$ has a maximum quotient in $\zeros$, then for every $Z \in \zeros$ there is at most one arrow with domain $X$ and codomain $Z$. In fact, consider $z,z' \colon X \rightarrow Z$; then there exist $\varphi, \varphi' \colon M(X) \rightarrow Z$ such that $\varphi x =z$ and $\varphi ' x =z'$; but, thanks to the properties of the subcategory $\zeros$, we have $\varphi= \varphi'$ and so $z=z'$.
	
	\begin{prop} \label{coker id sse massimo quoziente}
		Let $\mathcal{C}$ be a category with a fixed class $\zeros$ of zero objects. For every object $X \in \mathcal{C}$, the $\zeros$-cokernel of $id_X$ exists if and only if $X$ has a maximum quotient in $\zeros$.
	\end{prop}
	
	\begin{proof}
		$(\Longrightarrow)$ Suppose $q=\Zcoker(id_X)$; thanks to Lemma \ref{coker id} we know that $q \colon X \rightarrow Z$, with $Z \in \zeros$. We show that $q$ determines a maximum quotient of $X$ in $\zeros$. Consider an arrow $z \colon X \rightarrow Z'$; since $q=\Zcoker(id_X)$, there exists an arrow $\varphi$ such that $\varphi q =z$.\\
		$(\Longleftarrow)$ Let $x \colon X \rightarrow M(X)$ be a maximum quotient of $X$ in $\zeros$. We show that $x=\Zcoker(id_X)$. In fact, consider an arrow $f \colon X \rightarrow Y$ such that $fid_X=f \in \zideal$. Hence, $f=\varepsilon_Y \chi$ where $\chi \colon X \rightarrow \Z(Y)$ and $\varepsilon_Y \colon \Z(Y) \rightarrowtail Y$. Therefore, thanks to the property of the maximum quotient in $\zeros$, there exists an arrow $\varphi$ such that $\varphi x= \chi$, and so $\varepsilon_Y \varphi x = f$. The uniqueness in guaranteed by the fact that $x$ is an epimorphism.
	\end{proof}
	
	\begin{prop}\label{coker e max quotient}
		Let $\mathcal{C}$ be a category with pushouts and a fixed class $\zeros$ of zero objects. Given an object $X \in \mathcal{C}$, every arrow $f \colon X \rightarrow Y$ admits a $\zeros$-cokernel if and only if $X$ has a maximum quotient in $\zeros$.
	\end{prop}
	
	\begin{proof}
		$(\Longrightarrow)$ It is a direct consequence of Proposition \ref{coker id sse massimo quoziente}.\\
		$(\Longleftarrow)$ We prove that $q$, defined as the pushout of $x$ (where $x \colon X \rightarrow M(X)$ is a maximum quotient of $X$ in $\zeros$) along $f$ is the $\zeros$-cokernel of $f$:
		\[\begin{tikzcd}
			X & Y \\
			{M(X)} & Q.
			\arrow["f", from=1-1, to=1-2]
			\arrow["x"', from=1-1, to=2-1]
			\arrow["q", from=1-2, to=2-2]
			\arrow["m"', from=2-1, to=2-2]
			\arrow["\lrcorner"{anchor=center, pos=0, rotate=180}, draw=none, from=2-2, to=1-1]
		\end{tikzcd}\]
		Consider an arrow $p \colon Y \rightarrow P$ such that $pf \in \zideal$; we know that $pf$ factors as in the commutative diagram below
		\[\begin{tikzcd}
			X & Y & P \\
			& {\Z(P).}
			\arrow["f", from=1-1, to=1-2]
			\arrow["{z_P}"', curve={height=6pt}, from=1-1, to=2-2]
			\arrow["p", from=1-2, to=1-3]
			\arrow["{\varepsilon_P}"', curve={height=6pt}, from=2-2, to=1-3]
		\end{tikzcd}\]
		Since $\Z(P) \in \zeros$, there exists an arrow $\theta \colon M(X) \rightarrow \Z(P)$ such that $\theta x = z_P$. Due to the fact that that $\varepsilon_P \theta x = \varepsilon_P z_P= pf$ and to the universal property of pushouts, there exists an arrow $\varphi \colon Q \rightarrow P$ such that $\varphi q = p$ and $\varphi m = \varepsilon_P \theta$. Finally, $\varphi$ is the unique arrow satisfying $\varphi q=p$ since $q$ is an epimorphism as pushout of the epimorphism $x$.
	\end{proof}
	
	In \cite{Facchini2}, the authors considered a notion of \emph{pretorsion theory} with respect to a fixed class of objects $\zeros$. Specifically, a pair $\torsione$ of full and replete subcategories $\tort$ and $\torf$ of a category $\mathcal{C}$ is a pretorsion theory with respect to the class $\zeros \coloneqq \tort \cap \torf$ if every morphism from an object in $\tort$ to an object in $\torf$ belongs to the ideal $\zideal$ of morphisms that factor through $\zeros$, and for every object $A$ in $\mathcal{C}$ there exists a $\zeros$-exact sequence with a torsion object in $\tort$ as its left endpoint and a torsion-free object in $\torf$ as its right endpoint.
	
	The concept of pretorsion theory can be seen as a generalization of the notion of torsion theory. In fact, when the category $\mathcal{C}$ is pointed, every pretorsion theory such that $\tort \cap \torf$ reduces to the zero object is a torsion theory in the usual sense \cite{BG torsion th, JT torsion th}.
	
	Since in this work we will consider only pretorsion theories with respect to a fixed class $\zeros$, we will refer to such pretorsion theories as $\zeros$-torsion theories (or, if there is no ambiguity, simply torsion theories).
	
	\begin{defi}[\cite{Facchini2}]\label{definizione preesatta}
		Let $\mathcal{C}$ be a category with a fixed class $\zeros$ of zero objects. A \emph{$\zeros$-torsion theory} (or simply torsion theory, if there is no ambiguity) $(\tort, \torf)$ in $\mathcal{C}$ consists of two full, replete subcategories $\tort, \torf$ of $\mathcal{C}$ satisfying the following conditions:
		\begin{itemize}
			\item[(a)] $\zeros = \tort \cap \torf$;
			\item[(b)] $\Hom(T, F) \subseteq \zideal$ for every object $T \in \tort$ , $F \in \torf$;
			\item[(c)] for every object $A$ of $\mathcal{C}$ there is a short $\zeros$-exact sequence
			\[\begin{tikzcd}
				{\T(A)} & A & {\F(A)}
				\arrow["{t_A}", from=1-1, to=1-2]
				\arrow["{\eta_A}", from=1-2, to=1-3]
			\end{tikzcd}\]
			with $\T(A) \in \tort$ and $\F(A) \in \torf$.
		\end{itemize}
	\end{defi}
	
	As proved in \cite[Proposition 3.1]{Facchini2}, such a $\zeros$-exact sequence is unique up to isomorphisms.
	
	In \cite{Facchini2}, the authors prove that, by fixing for each object $A$ a $\zeros$-exact sequence as in Definition \ref{definizione preesatta}, every torsion theory gives rise to a pair of functors:
	\[\begin{tikzcd}[row sep=5pt]
		{\mathcal{C}} && {\torf} && {\mathcal{C}} && {\tort} \\
		A && {\F(A)} && A && {\T(A)} \\
		\\
		\\
		B && {\F(B)} && B && {\T(B),}
		\arrow["\F", from=1-1, to=1-3]
		\arrow["f"', from=2-1, to=5-1]
		\arrow[maps to, from=2-1, to=2-3]
		\arrow["{\F(f)}", from=2-3, to=5-3]
		\arrow[maps to, from=5-1, to=5-3]
		\arrow["\T", from=1-5, to=1-7]
		\arrow["f"', from=2-5, to=5-5]
		\arrow["{\T(f)}", from=2-7, to=5-7]
		\arrow[maps to, from=2-5, to=2-7]
		\arrow[maps to, from=5-5, to=5-7]
	\end{tikzcd}\]
	where $\T(f) \colon \T(A) \rightarrow \T(B)$ is the unique morphism such that $f t_A=t_B \T(f)$ and it exists since $\eta_B f t_A \in \zideal$; in a similar way, $\F(f) \colon \F(A) \rightarrow \F(B)$ is the unique morphism such that $\F(f) \eta_A = \eta_B f$ and it exists since $\eta_B f t_A \in \zideal$:
	\[\begin{tikzcd}
		{\T(A)} & A & {\F(A)} \\
		{\T(B)} & B & {\F(B).}
		\arrow["f", from=1-2, to=2-2]
		\arrow["{\eta_A}", from=1-2, to=1-3]
		\arrow["{\eta_B}"', from=2-2, to=2-3]
		\arrow["{\exists!\F(f)}", from=1-3, to=2-3]
		\arrow["{t_A}", from=1-1, to=1-2]
		\arrow["{t_B}"', from=2-1, to=2-2]
		\arrow["{\exists!\T(f)}"', from=1-1, to=2-1]
	\end{tikzcd}\]
	\begin{prop}[\cite{Facchini2}, Proposition 3.3]
		Let $\mathcal{C}$ be a category with a fixed class $\zeros$ of zero objects, and $(\tort ,\torf)$ a $\zeros$-torsion theory in $\mathcal{C}$. Then:
		\begin{itemize}
			\item[(a)] the functor $\F \colon \mathcal{C} \rightarrow \torf$ is a left inverse left adjoint of the inclusion functor $\i_{\torf} \colon \torf \hookrightarrow \mathcal{C}$ and the unit is given by $\eta$;
			\item[(b)] the functor $\T \colon \mathcal{C} \rightarrow \tort$ is a left inverse right adjoint of the inclusion functor $\i_{\tort} \colon \tort \hookrightarrow \mathcal{C}$ and the counit is given by $t$.
		\end{itemize}
	\end{prop}
	
	As shown in \cite{Facchini2}, the following observations are a direct consequence of the previous proposition:
	\begin{itemize}
		\item[(a)] For every object $Z \in \zeros = \tort \cap \torf$, the chosen short $\zeros$-exact sequence is the sequence $Z \xrightarrow{id_Z} Z \xrightarrow{id_Z} Z$.
		\item[(b)] For every object $T \in \tort$, the chosen short $\zeros$-exact sequence for $T$ is a sequence of the form $T \xrightarrow{id_T} T \xrightarrow{\eta_T} F$ for an object $F \in \torf$.
		\item[(c)] For every $F \in \torf$, the chosen short $\zeros$-exact sequence is a sequence of the form $T \xrightarrow{t_F} F \xrightarrow{id_F} F$ for some $T \in \tort$.
	\end{itemize}
	We observe that in (c), since $T$ is the domain of the $\zeros$-kernel of $id_F$, the square below is a pullback
	\[\begin{tikzcd}
		T & {\Z(F)} \\
		F & F;
		\arrow[from=1-1, to=1-2]
		\arrow[from=1-1, to=2-1]
		\arrow["\lrcorner"{anchor=center, pos=0}, draw=none, from=1-1, to=2-2]
		\arrow["{\varepsilon_F}", from=1-2, to=2-2]
		\arrow[equal, from=2-1, to=2-2]
	\end{tikzcd}\]
	hence $T=\Z(F) \in \zeros$. Moreover, in (b) we are requiring that $\eta_T$ is the $\zeros$-cokernel of $id_T$. So, as proved in Proposition \ref{coker id sse massimo quoziente}, we get that $F \in \zeros$ and $\eta_T$ is a maximum quotient of $T$ in $\zeros$. This simple observation leads to the following:
	
	\begin{cor}\label{dominio torsione}
		Every arrow $f \colon T \rightarrow A$ such that $T \in \tort$ admits a $\zeros$-cokernel.
	\end{cor}
	
	The corollary just stated will be of crucial importance in the following sections. Indeed, as previously observed, in our context the existence of the $\zeros$-cokernel for a general morphism $f$ is not guaranteed. However, in our proofs, we will often work with morphisms that have a torsion domain and, thanks to the results we have just established, we can ensure the existence of the $\zeros$-cokernel for such morphisms.
	
	\section{Torsion Theories and Factorization Systems} \label{torsion th and fact sys}
	
	First of all, to approach this section, we need to recall the notion of a \emph{protomodular} category. A category $\mathcal{C}$ is \emph{protomodular} \cite{Bourn protomod} if
	\begin{itemize}
		\item[(a)] $\mathcal{C}$ has pullbacks of split epimorphisms along any arrow, and
		\item[(b)] given a commutative diagram
		\[\begin{tikzcd}
			D & C \\
			A & B
			\arrow["g", from=1-1, to=1-2]
			\arrow["q"', shift right, from=1-1, to=2-1]
			\arrow["\lrcorner"{anchor=center, pos=0.125}, draw=none, from=1-1, to=2-2]
			\arrow["p"', shift right, from=1-2, to=2-2]
			\arrow["t"', shift right, from=2-1, to=1-1]
			\arrow["f"', from=2-1, to=2-2]
			\arrow["s"', shift right, from=2-2, to=1-2]
		\end{tikzcd}\]
		where $p,q$ are split epimorphisms with respective sections $s,t$ and the downward directed square is a pullback, the pair $(g,s)$ is jointly extremal epimorphic.
	\end{itemize}
	
	This definition, when $\mathcal{C}$ is a pointed category with pullbacks of split epimorphisms along any map, is equivalent to requiring the validity of the split short five lemma.\\
	
	The purpose of this section is to generalize, to the non-pointed case, some of the results obtained in \cite{Gran cond N} and \cite{Gran torsione} concerning the relationship between torsion theories and factorization systems. To begin, let us recall the following:
	
	\begin{defi}
		A \emph{factorization system} for a category $\mathcal{C}$ is a pair of classes of arrows $(\facte, \factm)$ such that:
		\begin{itemize}
			\item[(a)] for every $e \in \facte$ and $m \in \factm$ one has $e \downarrow m$, i.e.\ for every commutative square in $\mathcal{C}$
			\[\begin{tikzcd}
				A & B \\
				C & D
				\arrow["{e \in \facte}", from=1-1, to=1-2]
				\arrow["{m \in \factm}"', from=2-1, to=2-2]
				\arrow["g"', from=1-1, to=2-1]
				\arrow["h", from=1-2, to=2-2]
				\arrow["{d}"', from=1-2, to=2-1]
			\end{tikzcd}\]
			there exists a unique arrow $d \colon B \rightarrow C$ such that $de=g$ and $md=h$;
			\item[(b)] every arrow $f$ in $\mathcal{C}$ factors as $f=me$, where $m \in \factm$ and $e \in \facte$;
			\item[(c)] the classes $\facte$ and $\factm$ (seen as full subcategories of $\frecce$) are replete.
		\end{itemize}
		A factorization system $(\facte, \factm)$ is \emph{stable with respect to a class of arrows} $\mathcal{A}$ if the pullback of every arrow of $\facte$ along an arrow of $\mathcal{A}$ is an arrow of $\facte$, too. If $\mathcal{A}=\frecce$, we say that $(\facte, \factm)$ is \emph{stable}.
	\end{defi}
	
	It is a well-known fact that the conditions (a), (b), and (c) introduced in the definition above are equivalent to requiring that the classes $\facte$ and $\factm$ are closed under composition, contain all isomorphisms, and that every arrow $f$ of $\mathcal{C}$ admits a $(\facte, \factm)$-factorization which is unique up to a unique isomorphism.\\
	
	Let $\mathcal{C}$ be a category with a fixed class $\zeros$ of zero objects. Starting from a torsion theory $\torsione$ in $\mathcal{C}$, we define the classes of arrows
	\[ \facte \coloneqq \{ e \in \frecce \, | \, e \text{ is a $\zeros$-cokernel, } K[e] \in \tort \} \]
	and
	\[ \factm \coloneqq \{ m \in \frecce | K[m] \in \torf \}. \]
	
	\begin{prop}
		For every $e \in \facte$ and $m \in \factm$, one has $e \downarrow m$.
	\end{prop}
	
	\begin{proof}
		Consider the commutative square on the left, where the top horizontal arrow $e$ belongs to $\facte$ and the bottom horizontal arrow $m$ belongs to $\factm$. Then, define $k(e) \coloneqq \Zker(e)$ and $k(m) \coloneqq \Zker (m)$
		\[\begin{tikzcd}
			{K[e]} & A & B \\
			{K[m]} & C & D.
			\arrow["{k(e)}", from=1-1, to=1-2]
			\arrow["{\exists!\varphi}"', from=1-1, to=2-1]
			\arrow["{e\in \facte}", from=1-2, to=1-3]
			\arrow["g"', from=1-2, to=2-2]
			\arrow["h", from=1-3, to=2-3]
			\arrow["{k(m)}"', from=2-1, to=2-2]
			\arrow["{m\in \factm}"', from=2-2, to=2-3]
		\end{tikzcd}\]
		Since $mgk(e)=hek(e) \in \zideal$, there exists a unique $\varphi$ satisfying $gk(e)=k(m) \varphi$. Observe that the domain of $\varphi$ is in $\tort$ and its codomain is in $\torf$, therefore $\varphi \in \zideal$. Now, $gk(e)=k(m)\varphi \in \zideal$ and so there exists a unique $d$ such that $de=g$ (because every $\zeros$-cokernel is the $\zeros$-cokernel of its $\zeros$-kernel). Observing that $e$ is an epimorphism, we get also $md=h$.
	\end{proof}
	
	To guarantee the desired results, we need two additional properties on the torsion theories. The first one is a generalization of a property already known in the pointed case and introduced in \cite{Gran cond N}. The second one is required to ensure the proper behavior of exact sequences in the non-pointed context.
	
	\begin{defi}
		Let $\mathcal{C}$ be a category with a fixed class $\zeros$ of zero objects. A torsion theory $\torsione$ in $\mathcal{C}$ satisfies \emph{Condition (N)} if, for every diagram
		\[\begin{tikzcd}
			{\T(K[f])} & {K[f]} & A & B
			\arrow["f", from=1-3, to=1-4]
			\arrow["k", from=1-2, to=1-3]
			\arrow["{t_{K[f]}}", from=1-1, to=1-2]
		\end{tikzcd}\]
		where $k = \Zker(f)$, then $kt_{K[f]}$ is the $\zeros$-kernel of some arrow.
	\end{defi}
	
	It is known that, in the pointed context, Condition (N) is equivalent to the requirement that every torsion object is characteristic in the sense of \cite{CM char}. Indeed, one of the possible ways of defining characteristic subobjects is the following: a subobject $c \colon C \rightarrowtail B$ is characteristic if, for every normal subobject $n \colon B \rightarrowtail A$, the composite $nc$ is normal. Replacing normal subobjects with $\zeros$-kernels, one can express the condition for a subobject to be characteristic also in our non-pointed setting, allowing to state Condition (N) in terms of characteristic subobjects. The exploration of the notion of characteristic subobject in a non-pointed setting is material for future work.
	
	\begin{defi}
		Let $\mathcal{C}$ be a category with a fixed class $\zeros$ of zero objects. A torsion theory $\torsione$ in $\mathcal{C}$ satisfies \emph{Condition (M)} if, for every short exact sequence
		\[\begin{tikzcd}
			T & A & Q,
			\arrow["k", from=1-1, to=1-2]
			\arrow["q", from=1-2, to=1-3]
		\end{tikzcd}\]
		where $T \in \tort$, $\F(T) \cong \Z(Q)$.
	\end{defi}
	
	\begin{prop}\label{Da torsione a sistema}
		Let $\mathcal{C}$ be a regular protomodular category with a fixed class $\zeros$ of zero objects. Consider a torsion theory $\torsione$ in $\mathcal{C}$ satisfying Conditions (N) and (M). Then, the classes of arrows $(\facte, \factm)$ defined above form a factorization system which is stable with respect to the class of arrows inverted by the functor $\Z$.
	\end{prop}
	
	\begin{proof}
		We shall first prove that the class $\facte$ is replete (in a similar way one can show this is also true for $\factm$). Consider an arrow $e \colon A \rightarrow B$ in $\facte$, an arrow $g \colon C \rightarrow D$, and a pair of isomorphisms $\alpha, \beta$ such that $g \alpha = \beta e$. Let $k \colon K[e] \rightarrow A$ be the $\zeros$-kernel of $e$. We have to show that the sequence $K[e] \xrightarrow{\alpha k} C \xrightarrow{g} D$ is exact (and so $g$ is a $\zeros$-cokernel and $K[g] \in \tort$). Let $x$ be such that $gx \in \zideal$; then $e \alpha^{-1} x =\beta^{-1}gx \in \zideal$. So there exists an arrow $\varphi$ such that $k \varphi = \alpha^{-1}x$, and then $x=\alpha k \varphi$ (the uniqueness of such a $\varphi$ is guaranteed since $\alpha k$ is a monomorphism). Hence, $\alpha k = \Zker(g)$. Now, let $y$ be an arrow such that $y \alpha k \in \zideal$. $e = \Zcoker(k) $, so there exists an arrow $\psi$ satisfying $\psi e = y \alpha$. Then, $\psi \beta^{-1} g = \psi e \alpha^{-1}=y \alpha \alpha^{-1}=y$ (the uniqueness of such a $\varphi$ is guaranteed since $g$ is an epimorphism). Hence, $g=\Zcoker(\alpha k)$.\\
		Let us then proceed by showing the construction of the factorization of any arrow $f \colon A \rightarrow B$. Consider the commutative diagram
		\[\begin{tikzcd}
			{\T(K[f])} & {K[f]} & A && B \\
			& {\Z(Q)} && Q & I,
			\arrow["t", tail, from=1-1, to=1-2]
			\arrow["{z_Q}"', two heads, from=1-1, to=2-2]
			\arrow["k", tail, from=1-2, to=1-3]
			\arrow["f", from=1-3, to=1-5]
			\arrow["e"', two heads, from=1-3, to=2-4]
			\arrow["{\varepsilon_Q}"', tail, from=2-2, to=2-4]
			\arrow["m"', from=2-4, to=1-5]
			\arrow["q"', two heads, from=2-4, to=2-5]
			\arrow["i"', tail, from=2-5, to=1-5]
		\end{tikzcd}\]
		where $t=t_{K[f]}$, $e=\Zcoker(kt)$ (we can guarantee the existence of the $\zeros$-cokernel of $kt$ thanks to Corollary \ref{dominio torsione}), $m$ is defined since $fkt \in \zideal$, and $iq$ is the (regular epi, mono)-factorization of $m$. Moreover, since $ekt \in \zideal$, we have $ekt=\varepsilon_Q z_Q$ (where $\varepsilon_Q$ is a monomorphism). Thanks to Condition (N) we know that $kt=\Zker(e)$. Furthermore, recalling that Condition (M) holds, we get $\F(\T(K[f])) \cong \Z(Q)$ and $\F(\T(K[f])) \cong \Z(\F(K[f]))$. We construct the commutative diagram
		\[\begin{tikzcd}
			{\T(K[f])} & {K[f]} & {\F(K[f])} \\
			{\T(K[f])} & A & Q \\
			{\Z(I)} & I & I,
			\arrow["t", tail, from=1-1, to=1-2]
			\arrow[equal, from=1-1, to=2-1]
			\arrow["\eta", two heads, from=1-2, to=1-3]
			\arrow["k", tail, from=1-2, to=2-2]
			\arrow["{k'}", from=1-3, to=2-3]
			\arrow["kt", tail, from=2-1, to=2-2]
			\arrow["{\Z(q)z_Q}"', two heads, from=2-1, to=3-1]
			\arrow["e", two heads, from=2-2, to=2-3]
			\arrow["qe"', two heads, from=2-2, to=3-2]
			\arrow["q", two heads, from=2-3, to=3-3]
			\arrow["{\varepsilon_I}"', tail, from=3-1, to=3-2]
			\arrow[equal, from=3-2, to=3-3]
		\end{tikzcd}\]
		where $\eta=\eta_{K[f]}$, and $k'$ is induced because $ekt \in \zideal$. We observe that $t= \Zker(\eta)$, $kt=\Zker(e)$, $\varepsilon_I=\Zker(id_I)$, $id_{\T(K[f])}=\Zker(\Z(q)z_Q)$, $k=\Zker(qe)$ (thanks to Lemma \ref{ker e mono}); moreover, $\eta$, $e$, $id_I$, $\Z(q)z_Q$, $qe$, and $q$ are regular epimorphisms. We want to prove that $k'=\Zker(q)$; if this statement holds, it follows that $K[m]=K[q]=\F(K[f]) \in \torf$ and $K[e]=\T(K[f]) \in \tort$, and so $f=me$ is the desired $(\facte, \factm)$-factorization. First of all, we check that $k'$ is a monomorphism. Indeed, observe that in the commutative diagram
		\[\begin{tikzcd}
			{\T(K[f])} & {K[f]} & {\F(K[f])} \\
			{\T(K[f])} & A & Q
			\arrow["t", tail, from=1-1, to=1-2]
			\arrow[equal, from=1-1, to=2-1]
			\arrow["\lrcorner"{anchor=center, pos=0}, draw=none, from=1-1, to=2-2]
			\arrow["\eta", two heads, from=1-2, to=1-3]
			\arrow["k", tail, from=1-2, to=2-2]
			\arrow["{k'}", from=1-3, to=2-3]
			\arrow["kt", tail, from=2-1, to=2-2]
			\arrow["e", two heads, from=2-2, to=2-3]
		\end{tikzcd}\]
		the square on the left is a pullback and $\Z(\F(K[f])) \cong \Z(Q)$; hence, we can apply \cite[Lemma 5.4]{Cappelletti} and conclude that $k'$ is a monomorphism (as we already mentioned, in \cite{Cappelletti} a short sequence $X \xrightarrow{x} Y \xrightarrow{y} Z$ is considered exact if $x=\Zker(y)$ and $y$ is a regular epimorphism). Now, define $\tilde{k} \coloneqq \Zker(q)$ and examine the commutative diagram
		\[\begin{tikzcd}
			&&& {\F(K[f])} \\
			{\T(K[f])} & {K[f]} & {K[q]} \\
			{\T(K[f])} & A & Q \\
			{\Z(I)} & I & I,
			\arrow["n"', from=1-4, to=2-3]
			\arrow["{k'}", curve={height=-12pt}, tail, from=1-4, to=3-3]
			\arrow["t", tail, from=2-1, to=2-2]
			\arrow[equal, from=2-1, to=3-1]
			\arrow["\eta", curve={height=-12pt}, two heads, from=2-2, to=1-4]
			\arrow["g", two heads, from=2-2, to=2-3]
			\arrow["k", tail, from=2-2, to=3-2]
			\arrow["{\tilde{k}}", tail, from=2-3, to=3-3]
			\arrow["kt", tail, from=3-1, to=3-2]
			\arrow["{\Z(q)z_Q}"', two heads, from=3-1, to=4-1]
			\arrow["e", two heads, from=3-2, to=3-3]
			\arrow["qe"', two heads, from=3-2, to=4-2]
			\arrow["q", two heads, from=3-3, to=4-3]
			\arrow["{\varepsilon_I}"', tail, from=4-1, to=4-2]
			\arrow[equal, from=4-2, to=4-3]
		\end{tikzcd}\]
		where $n$ is defined since $qk' \eta = qek \in \zideal$ implies $qk' \in \zideal$ (thanks to \cite[Proposition 2.5]{Cappelletti}), and $g$ exists since $qek \in \zideal$. Let us prove that $g$ is a regular epimorphism: in the commutative diagram
		\[\begin{tikzcd}
			{K[f]} & A & I \\
			{K[q]} & Q & I
			\arrow["k", tail, from=1-1, to=1-2]
			\arrow["g"', two heads, from=1-1, to=2-1]
			\arrow["\lrcorner"{anchor=center, pos=0}, draw=none, from=1-1, to=2-2]
			\arrow["qe", two heads, from=1-2, to=1-3]
			\arrow["e"', two heads, from=1-2, to=2-2]
			\arrow[equal, from=1-3, to=2-3]
			\arrow["{\tilde{k}}"', tail, from=2-1, to=2-2]
			\arrow["q"', two heads, from=2-2, to=2-3]
		\end{tikzcd}\]
		the square on the left is a pullback since $id_I$ is a monomorphism (we are applying, again, \cite[Lemma 5.4]{Cappelletti}); therefore, by regularity, $g$ is a regular epimorphism as pullback of the regular epimorphism $e$. Hence, $n$ is a monomorphism because $k'$ is, and $n$ is a regular epimorphism because $g$ is. Therefore, $n$ is an isomorphism and $f=me$ is the desired $(\facte, \factm)$-factorization.\\
		It remains to prove that, for every arrow $e \in \facte$ and for every arrow $h$ such that $\Z(h)$ is an isomorphism, the pullback of $e$ along $h$ is in $\facte$. Let $\overline{e}$ be defined by the pullback
		\[\begin{tikzcd}
			P & C \\
			A & B.
			\arrow["{\overline{e}}", two heads, from=1-1, to=1-2]
			\arrow[from=1-1, to=2-1]
			\arrow["\lrcorner"{anchor=center, pos=0}, draw=none, from=1-1, to=2-2]
			\arrow["h", from=1-2, to=2-2]
			\arrow["e"', two heads, from=2-1, to=2-2]
		\end{tikzcd}\]
		Applying \cite[Proposition 4.3]{Cappelletti}, we obtain $K[\overline{e}] \cong K[e] \in \tort$. Thanks to Condition (M) we get $\F(K[\overline{e}]) \cong \F(K[e]) \cong \Z(B) \cong \Z(C)$. To conclude, the pullback
		\[\begin{tikzcd}
			{K[\overline{e}]} & P \\
			{\Z(C)} & C
			\arrow["{\overline{k}}",, from=1-1, to=1-2]
			\arrow[two heads, from=1-1, to=2-1]
			\arrow["\lrcorner"{anchor=center, pos=0}, draw=none, from=1-1, to=2-2]
			\arrow["{\overline{e}}", two heads, from=1-2, to=2-2]
			\arrow["{\varepsilon_C}"', from=2-1, to=2-2]
		\end{tikzcd}\]
		is also a pushout, thanks to \cite[Proposition 14]{Bourn protomod}. So, since $M(K[\overline{e}]) \cong \F(K[\overline{e}]) \cong \Z(C)$, $\overline{e}$ is a $\zeros$-cokernel (see the proof of Proposition \ref{coker e max quotient}).
	\end{proof}
	
	We have thus shown that any torsion theory satisfying Conditions (N) and (M) allows for the construction of a stable factorization system with respect to the class of arrows inverted by $\Z$, where the arrows in
	$\facte$ are $\zeros$-cokernels. This naturally leads to the question of whether such a construction admits an inverse. In other words, we want to understand if, starting from a factorization system that satisfies the aforementioned conditions, it is possible to construct a torsion theory.
	First of all, we state the following:
	
	\begin{lem}\label{zero frecce in M}
		Let $\mathcal{C}$ be a category with a fixed class $\zeros$ of zero objects. Consider a factorization system $(\facte, \factm)$ in $\mathcal{C}$ such that every arrow of $\facte$ is a $\zeros$-cokernel. Then, every arrow $Z \rightarrow A$, with $Z \in \mathcal{Z}$, belongs to $\factm$.
	\end{lem}
	
	\begin{proof}
		Let
		\[\begin{tikzcd}
			{Z} && {A} \\
			& I
			\arrow[from=1-1, to=1-3]
			\arrow["e"', from=1-1, to=2-2]
			\arrow["m"', from=2-2, to=1-3]
		\end{tikzcd}\]
		be the $(\facte, \factm)$-factorization of $Z\rightarrow A$. Since $e$ is a $\zeros$-cokernel it follows $e=\Zcoker ( \Zker (e))= \Zcoker (id_{Z})= id_{Z}$. Hence, $e$ is an isomorphism an so $Z\rightarrow A$ belongs to $\factm$.
	\end{proof}
	
	To ensure our desired results, it is necessary to introduce the following:
	
	\begin{defi}
		We say that the class $\zeros$ of zero objects, in a category $\mathcal{C}$, is \emph{terminal} if for every $A \in \mathcal{C}$ there exists an arrow $v_A \colon A \rightarrow \V(A)$ such that:
		\begin{itemize}
			\item[(a)] $ \V(A) \in \zeros$;
			\item[(b)] for every $\chi \colon A \rightarrow Z$, with $Z \in \zeros$, there exists a (necessarily unique) arrow $\varphi \colon Z \rightarrow \V(A)$ such that $\varphi \chi = v_A$.
		\end{itemize}
	\end{defi}
	
	This definition, despite seeming artificial, is actually highly consistent with the underlying idea behind the definition of these non-pointed categories with a class of zero objects
	$\zeros$. In \cite{Cappelletti}, we considered as primary and main example the one given by a protomodular variety of universal algebras in which the theory has at least one constant; in this context, we define $\zeros$ to be the class of objects obtainable as quotients of the initial object. Since in such categories the arrow from the initial object to the terminal one is always a regular epimorphism, it follows that the terminal object belongs to the class
	$\zeros$, making this class trivially terminal.
	
	\begin{oss}
		Let $\zeros$ be a terminal class of zero objects in a category $\mathcal{C}$. Given a pair of arrows $f_1,f_2 \colon A \rightarrow \V(A)$ and observing that the only existing arrow from $\V(A)$ to itself is the identity, the fact that $\zeros$ is a terminal class gives the commutative triangle
		\[\begin{tikzcd}
			A & {\V(A)} \\
			& {\V(A),}
			\arrow["{f_1}", from=1-1, to=1-2]
			\arrow["{f_2}"', from=1-1, to=2-2]
			\arrow[equal, from=1-2, to=2-2]
		\end{tikzcd}\]
		from which we get $f_1=f_2$.
	\end{oss}
	
	\begin{oss}\label{prop V}
		The assignment $A \mapsto \V(A)$ allows to define a functor $\V \colon \op{\mathcal{C}} \rightarrow \zeros$. Given an arrow $f \colon A \rightarrow B$ in $\mathcal{C}$, we define $\V(f) \colon \V(B) \rightarrow \V(A)$ as the unique arrow such that $\V(f)v_Bf=v_A$. Clearly, this assignment is functorial.\\
		In particular, the following facts trivially hold:
		\begin{itemize}
			\item[(a)] for a pair of objects $A,B$ of $\mathcal{C}$ such that there exist an arrow $A \rightarrow B$ and an arrow $B \rightarrow A$, one has $\V(A) \cong \V(B)$;
			\item[(b)] for every object $A$ of $\mathcal{C}$, $\V(\V(A)) \cong \V(A)$.
		\end{itemize}
	\end{oss}
	
	From now on we will work with a fixed terminal class $\zeros$ of zero objects in $\mathcal{C}$.\\
	Starting from a factorization system $(\facte, \factm)$, we define a pair of full subcategories of $\mathcal{C}$ whose objects are $\tort \coloneqq \{ T \in \mathcal{C} \,|\, \exists \, T \rightarrow Z \text{ belonging to } \facte \text{ (where $Z \in \zeros$)}\}$ and $\torf \coloneqq \{ F \in \mathcal{C} \,|\, F \rightarrow \V(F) \text{ belongs to } \factm\}$.
	
	\begin{prop}
		Let $\mathcal{C}$ be a regular category with a fixed terminal class $\zeros$ of zero objects. Consider a factorization system $(\facte, \factm)$ which is stable with respect to the class of arrows inverted by $\Z$, and where the arrows in $\facte$ are $\zeros$-cokernels. Then, the pair of subcategories $\torsione$ defined above forms a torsion theory such that $\tort \cap \torf = \zeros$ and satisfying (N) and (M).
	\end{prop}
	
	\begin{proof}
		First of all, we prove that $\tort \cap \torf = \zeros$. On the one hand, let $Z \in \zeros$; clearly $id_Z \in \facte$, and so $Z \in \tort$. Moreover, $Z \rightarrow \V(Z) \in \factm$ (thanks to Lemma \ref{zero frecce in M}), hence $Z \in \torf$. On the other hand, consider $X \in \tort \cap \torf$ and examine the commutative diagram
		\[\begin{tikzcd}
			X && Z \\
			& {\V(X),}
			\arrow["{e \in \facte}", from=1-1, to=1-3]
			\arrow["{m \in \factm}"', from=1-1, to=2-2]
			\arrow["{z \in \factm}", from=1-3, to=2-2]
		\end{tikzcd}\]
		where $e$ is the arrow that makes $X$ an object of $\tort$, and $m \in \factm$ since $X \in \torf$. Therefore, since by Lemma \ref{zero frecce in M} $z \in \factm$ and since the $(\facte, \factm)$-factorization is unique up to isomorphisms, we get that $e$ is an isomorphism and so $X \in \zeros$.\\
		Let us then prove that every arrow $f \colon T \rightarrow F$ with domain in $\tort$ and codomain in $\torf$ is trivial. Let $e \colon T \rightarrow Z_1$ be an arrow of $\facte$ making $T$ an object of $\tort$, and let $v_F \colon F \rightarrow \V(F)$ be in $\factm$. We know that there exist an arrow $\psi \colon Z_1 \rightarrow \V(T)$ such that $\psi e = v_T$ and an arrow $\theta \colon \V(F) \rightarrow \V(T)$ such that $ \theta v_F f=v_T$ (the existence of both is implied by the property of $v_T$). We observe that $\theta \in \factm$, because it is an arrow between two zero objects. Hence, we can consider the commutative diagram
		\[\begin{tikzcd}
			T & {Z_1} \\
			F & {\V(T),}
			\arrow["e", from=1-1, to=1-2]
			\arrow["f"', from=1-1, to=2-1]
			\arrow["d"', from=1-2, to=2-1]
			\arrow["\psi", from=1-2, to=2-2]
			\arrow["{\theta v_F}"', from=2-1, to=2-2]
		\end{tikzcd}\]
		where $d$ exists since $e \downarrow \theta v_F$ (observe that $\theta v_F \in \factm$ as a composite of arrows belonging to $\factm$). Therefore $f \in \zideal$.\\
		Next, we want to provide, for every object $A$ of $\mathcal{C}$, a short exact sequence $\T(A)\rightarrow A \rightarrow \F(A)$ with $\T(A) \in \tort$ and $\F(A) \in \torf$. To do this, we take the $(\facte, \factm)$-factorization of $v_A \colon A \rightarrow \V(A)$
		\[\begin{tikzcd}
			A && {\V(A)} \\
			& I.
			\arrow["{v_A}", from=1-1, to=1-3]
			\arrow["e"', from=1-1, to=2-2]
			\arrow["m"', from=2-2, to=1-3]
		\end{tikzcd}\]
		Since we have $\V(e) \colon \V(I) \to \V(A)$ and $\V(m) \colon \V(\V(A)) \cong \V(A) \to \V(I)$ (see Remark \ref{prop V}), we get $\V(I) \cong \V(A)$, and this implies $I \in \torf$. We prove that $K[e] \in \tort$. Since $K[e]$ is defined by the pullback
		\[\begin{tikzcd}
			{K[e]} & {\Z(I)} \\
			A & I
			\arrow["{\overline{e}}", from=1-1, to=1-2]
			\arrow["{k}"', from=1-1, to=2-1]
			\arrow["\lrcorner"{anchor=center, pos=0}, draw=none, from=1-1, to=2-2]
			\arrow["{\varepsilon_I}", tail, from=1-2, to=2-2]
			\arrow["e"', from=2-1, to=2-2]
		\end{tikzcd}\]
		and $\facte$ is stable with respect to the class of arrows inverted by $\Z$ ($\varepsilon_I$ is trivially inverted by $\Z$), we get $\overline{e} \in \facte$ and so $K[e] \in \tort$. Therefore, $K[e] \xrightarrow{k} A \xrightarrow{e} I$ is the desired sequence.\\
		As a second-to-last step, we show that the torsion theory $\torsione$ satisfies Condition (N). Consider $\T(K) \xrightarrow{t_K} K \xrightarrow{k} A \xrightarrow{f} B$, where $f$ is an arrow of $\mathcal{C}$ and $k \coloneqq \Zker(f)$. We have to prove that $kt_K$ is the $\zeros$-kernel of some arrow. Take the $(\facte, \factm)$-factorizations of $f$ and $v_K$:
		\[\begin{tikzcd}
			A && B & K && {\V(K)} \\
			& I, &&& {I'.}
			\arrow["f", from=1-1, to=1-3]
			\arrow["e"', from=1-1, to=2-2]
			\arrow["{v_K}", from=1-4, to=1-6]
			\arrow["{e'}"', from=1-4, to=2-5]
			\arrow["m"', from=2-2, to=1-3]
			\arrow["{m'}"', from=2-5, to=1-6]
		\end{tikzcd}\]
		We recall that $k$ is defined by the pullback
		\[\begin{tikzcd}
			K & {\Z(B)} \\
			A & B,
			\arrow["\chi", from=1-1, to=1-2]
			\arrow["k"', tail, from=1-1, to=2-1]
			\arrow["\lrcorner"{anchor=center, pos=0}, draw=none, from=1-1, to=2-2]
			\arrow["{\varepsilon_B}", tail, from=1-2, to=2-2]
			\arrow["f"', from=2-1, to=2-2]
		\end{tikzcd}\]
		and observe that $\V(I')=\V(K)$, using again Remark \ref{prop V}. Consider the commutative diagram
		\[\begin{tikzcd}
			K && {\Z(B)} & {\V(K)} \\
			& {I'',}
			\arrow["\chi", from=1-1, to=1-3]
			\arrow["{e''}"', from=1-1, to=2-2]
			\arrow["\theta", from=1-3, to=1-4]
			\arrow["{m''}"', from=2-2, to=1-3]
			\arrow["{\theta m''}"', curve={height=12pt}, from=2-2, to=1-4]
		\end{tikzcd}\]
		where $m''e''=\chi$ is the $(\facte, \factm)$-factorization of $\chi$ and $\theta$ exists thanks to the property of $\V(K)$. We know that $\theta \in \factm$ (being an arrow between zero objects). Hence, due to the uniqueness of the $(\facte, \factm)$-factorization, we have $I''=I'$, $e''=e'$, and $\theta m'' = m'$. Since $e' \downarrow m$, there exists an arrow $d$ making the following diagram commutative:
		\[\begin{tikzcd}
			K & {I'} \\
			I & B
			\arrow["{e'}", from=1-1, to=1-2]
			\arrow["ek"', from=1-1, to=2-1]
			\arrow["d"', from=1-2, to=2-1]
			\arrow["{\varepsilon_Bm''}", from=1-2, to=2-2]
			\arrow["m"', from=2-1, to=2-2]
		\end{tikzcd}\]
		(observe that $mek=fk=\varepsilon_B \chi = \varepsilon_B m'' e'$). Now, we consider the diagrams
		\[\begin{tikzcd}
			K & {I'} & {\Z(B)} & K & {I'''} & {\Z(B)} \\
			A & I & B & A & I & B,
			\arrow["{e'}", from=1-1, to=1-2]
			\arrow["k"', tail, from=1-1, to=2-1]
			\arrow["\lrcorner"{anchor=center, pos=0}, draw=none, from=1-1, to=2-2]
			\arrow["{m''}", from=1-2, to=1-3]
			\arrow["d", from=1-2, to=2-2]
			\arrow["{\varepsilon_B}", tail, from=1-3, to=2-3]
			\arrow["{e'''}", from=1-4, to=1-5]
			\arrow[tail, from=1-4, to=2-4]
			\arrow["\lrcorner"{anchor=center, pos=0}, draw=none, from=1-4, to=2-5]
			\arrow["{m'''}", from=1-5, to=1-6]
			\arrow["{d'}", from=1-5, to=2-5]
			\arrow["\lrcorner"{anchor=center, pos=0}, draw=none, from=1-5, to=2-6]
			\arrow["{\varepsilon_B}", tail, from=1-6, to=2-6]
			\arrow["e"', from=2-1, to=2-2]
			\arrow["m"', from=2-2, to=2-3]
			\arrow["e"', from=2-4, to=2-5]
			\arrow["m"', from=2-5, to=2-6]
		\end{tikzcd}\]
		where the outer rectangle on the left is a pullback while, in the diagram on the right, both squares are pullbacks. Since $(\facte, \factm)$ is stable with respect to the class of arrows inverted by $\Z$, and $\Z(\varepsilon_B)$ is an isomorphism, it follows that $m''' \in \factm$ and $e''' \in \facte$. Moreover, due to the uniqueness of the $(\facte, \factm)$-factorization of $\chi= \varepsilon_B^*(f)=m'' e' =m''' e'''$, we obtain $I'=I'''$, $m'''=m''$, $e'''=e'$. This also implies $d=d'$ (since $e'$ is an epimorphism and $de'=d'e'$). Therefore, $d$ is a monomorphism as pullback of the monomorphism $\varepsilon_B$. To conclude, applying \cite[Proposition 4.3]{Cappelletti} to the pullback
		\[\begin{tikzcd}
			K & {I'} \\
			A & I,
			\arrow["{e'}", from=1-1, to=1-2]
			\arrow["k"', tail, from=1-1, to=2-1]
			\arrow["\lrcorner"{anchor=center, pos=0}, draw=none, from=1-1, to=2-2]
			\arrow["{d'}", tail, from=1-2, to=2-2]
			\arrow["e"', from=2-1, to=2-2]
		\end{tikzcd}\]
		we get that $\Zker(e)=k \, \Zker(e')=kt_K$ (observe that, by definition, $t_K=\Zker(e')$), and so $kt_K$ is a $\zeros$-kernel.\\
		The last step is to prove the validity of Condition (M). Take a short exact sequence $T \xrightarrow{h} C \xrightarrow{q} Q$ with $T \in \tort$. Let $z \colon T \rightarrow Z$ be a maximum quotient of $T$ in $\mathcal{Z}$ (recall that, as shown at the end of the previous section, every torsion object admits a maximum quotient in $\zeros$). Observe that $z \in \facte$: in fact we know that there exists an arrow $z' \colon T \rightarrow Z'$ belonging to $\facte$. Therefore, $z'=\Zcoker(\Zker(z'))=\Zcoker(id_T)=z$ and so $z$ belongs to $\facte$. Since $h=\Zker(q)$, there exists an arrow $T \rightarrow \Z(Q)$, and this arrow factors through the object $Z$ (because $Z$ is the codomain of the maximum quotient $z$). To sum up, we have the following commutative diagram
		\[\begin{tikzcd}
			T & C & Q \\
			& Z \\
			& {\Z(Q),}
			\arrow["h", from=1-1, to=1-2]
			\arrow["{z \in \facte}", from=1-1, to=2-2]
			\arrow[curve={height=18pt}, from=1-1, to=3-2]
			\arrow["q", from=1-2, to=1-3]
			\arrow["{i \in \factm}", from=2-2, to=1-3]
			\arrow["{\in \factm}", from=2-2, to=3-2]
			\arrow["{\in \factm}"', curve={height=18pt}, from=3-2, to=1-3]
		\end{tikzcd}\]
		where all the arrows with trivial domain belong to $\factm$ thanks to Lemma \ref{zero frecce in M}. Now, we show that $q \in \facte$. Consider the $(\facte, \factm)$-factorization $q=me$. If we prove that $\Zker(q)=\Zker(e)$, since both $q$ and $e$ are $\zeros$-cokernels, we can conclude that $m$ is an isomorphism. For this purpose, consider the commutative diagram
		\[\begin{tikzcd}
			& {I'} && I \\
			T && C && Q \\
			& {T',}
			\arrow["{m'}", from=1-2, to=1-4]
			\arrow["m", from=1-4, to=2-5]
			\arrow["{e'}", from=2-1, to=1-2]
			\arrow["h", from=2-1, to=2-3]
			\arrow["\psi"', shift right, dashed, from=2-1, to=3-2]
			\arrow["e", from=2-3, to=1-4]
			\arrow["q", from=2-3, to=2-5]
			\arrow["\varphi"', shift right, dashed, from=3-2, to=2-1]
			\arrow["{h'}"', from=3-2, to=2-3]
		\end{tikzcd}\]
		where $h' \coloneqq \Zker (e)$ and $m'e'$ is the $(\facte, \factm)$-factorization of $eh$. We want to show that the two dashed arrows that make the corresponding triangles commute exist. Trivially, $\varphi$ is determined by $qh'=meh' \in \zideal$. But also $eh \in \zideal$: we know that $eh=m'e'$, so $qh=mm'e'$; but $qh=iz$, so (again thanks to the uniqueness of the $(\facte, \factm)$-factorization) we obtain $I'=Z$, and hence $eh \in \zideal$. This implies the existence of $\psi$. Clearly, $\varphi$ and $\psi$ are inverse to each other. To conclude, consider the pullback
		\[\begin{tikzcd}
			T & {\Z(Q)} \\
			C & Q;
			\arrow["{z''}", from=1-1, to=1-2]
			\arrow["h"', tail, from=1-1, to=2-1]
			\arrow["\lrcorner"{anchor=center, pos=0}, draw=none, from=1-1, to=2-2]
			\arrow["{\varepsilon_Q}", tail, from=1-2, to=2-2]
			\arrow["{q=e}"', from=2-1, to=2-2]
		\end{tikzcd}\]
		since the class $\facte$ is stable with respect to the class of arrows inverted by $\Z$, and $\varepsilon_Q$ belongs to this class, we get that $z'' \in \facte$. Hence, $z''=\Zcoker(\Zker(z''))=\Zcoker (id_T)=z$. Recalling that the codomain of $z$ is $Z$ and the codomain of $z''$ is $\Z(Q)$, we conclude $Z=\Z(Q)$.
	\end{proof}
	
	To sum up, we have shown that, given a regular protomodular category $\mathcal{C}$ with a fixed terminal class $\zeros$ of zero objects, starting from a torsion theory that satisfies Conditions (N) and (M) it is possible to construct a factorization system $(\facte, \factm)$ which is stable with respect to the class of arrows inverted by $\Z$, and where every arrow in $\facte$ is a $\zeros$-cokernel. Moreover, we have observed how such a factorization system allows for the definition of a torsion theory satisfying Conditions (N) and (M). It makes sense to ask whether these constructions are inverse to each other. The answer is affirmative, but to verify this we need the following:
	
	\begin{lem}\label{torsioni uguali}
		Let $(\tort, \torf)$ and $(\tort', \torf)$ be two torsion theories in a category $\mathcal{C}$ with a fixed class $\zeros$ of zero objects. Then, $\tort= \tort'$.
	\end{lem}
	
	\begin{proof}
		It suffices to observe that the subcategories $\tort$ and $\tort'$ are uniquely determined by the functors $\T$ and $\T'$. So, $\T(A)=K[\eta_A]=\T'(A)$.
	\end{proof}
	
	\begin{prop}
		Let $\mathcal{C}$ be a regular protomodular category with a fixed terminal class $\zeros$ of zero objects. The assignments just described between torsion theories and factorization systems are inverse to each other.
	\end{prop}
	
	\begin{proof}
		On the one hand, let us start from a torsion theory $(\tort, \torf)$, consider the associated factorization system $(\facte, \factm)$ and the induced torsion theory $(\tort', \torf')$. If we consider an object $F$, since $K[F \rightarrow \V(F)]=F$, we have the chain of equivalent conditions: $F \in \torf$ if and only if $F \rightarrow \V(F)$ belongs to $\factm$ if and only if $F \in \torf'$. Therefore, $\torf = \torf '$ and so, by Lemma \ref{torsioni uguali}, $(\tort, \torf)=(\tort', \torf')$. On the other hand, take a factorization system $(\facte, \factm)$, consider the associated torsion theory $(\tort, \torf)$ and the induced factorization system $(\facte', \factm')$. For an arrow $e \colon A \rightarrow B$, observe that $e \in \facte $ if and only if $z \colon K[e] \rightarrow \Z(B)$ belongs to $\facte$. In fact, examine the diagram
		\[\begin{tikzcd}
			{K[e]} & {\Z(B)} \\
			A & B,
			\arrow["z", from=1-1, to=1-2]
			\arrow["k"', from=1-1, to=2-1]
			\arrow["\lrcorner"{anchor=center, pos=0}, draw=none, from=1-1, to=2-2]
			\arrow["{\varepsilon_B}", from=1-2, to=2-2]
			\arrow["e"', from=2-1, to=2-2]
			\arrow["\lrcorner"{anchor=center, pos=0, rotate=180}, draw=none, from=2-2, to=1-1]
		\end{tikzcd}\]
		which is both a pullback (it is constructed to compute the $\zeros$-kernel of $e$) and a pushout ($e$ is a $\zeros$-cokernel). Therefore, if $e \in \facte$, since $\varepsilon_B$ is inverted by $\Z$, we get $z \in \facte$; if $z \in \facte$, since the class $\facte$ is always pushout stable, we get $e \in \facte$. Thus, we have $e \in \facte$ if and only if $z \in \facte$. Moreover, we see that $z \in \facte$ if and only if $K[z]=K[e] \in \tort$, and this is equivalent to requiring $z \in \facte'$ (recall that $z=\Zcoker(id_{K[e]})$). To conclude, since $\facte'$ is pushout stable and pullback stable along arrows inverted by $\Z$, we get $z \in \facte'$ if and only if $e \in \facte'$. It is a general result that every factorization system is uniquely determined by one of its classes, and so $(\facte, \factm)=(\facte', \factm')$.
	\end{proof}
	
	\section{Torsion Theories and Galois Structures} \label{torsion th and Galois}
	Similarly to what we have seen about factorization systems, in this section we are interested in studying the connection between torsion theories and Galois structures in a non-pointed context. The goal is, once again, to generalize some of the results valid for pointed categories obtained in \cite{Gran torsione}.  With this purpose in mind, let us begin by recalling from \cite{janelidze1990pure} the basic definitions related to categorical Galois theory.
	
	\begin{defi}
		Let $\mathcal{C}$ be a category with pullbacks. A class $\mathbb{C}$ of arrows in $\mathcal{C}$ is admissible when:
		\begin{itemize}
			\item[(a)] every isomorphism is in $\mathbb{C}$;
			\item[(b)] $\mathbb{C}$ is closed under composition;
			\item[(c)] in the pullback
			\[\begin{tikzcd}
				\bullet & \bullet \\
				\bullet & \bullet
				\arrow["c", from=1-1, to=1-2]
				\arrow["a"', from=2-1, to=2-2]
				\arrow["d"', from=1-1, to=2-1]
				\arrow["b", from=1-2, to=2-2]
				\arrow["\lrcorner"{anchor=center, pos=0}, draw=none, from=1-1, to=2-2]
			\end{tikzcd}\]
			if $a$ and $b$ are in $\mathbb{C}$ then $c$ and $d$ are in $\mathbb{C}$, too.
		\end{itemize}
	\end{defi}
	
	\begin{defi}
		Let $\mathbb{C}$ be an admissible class of morphisms in a category $\mathcal{C}$. For an object $C$ of $\mathcal{C}$, we write $\mathbb{C}/C$ for the following category:
		\begin{itemize}
			\item[(a)] the objects are the pairs $(X,f)$ where $f \colon X \rightarrow C$ is in $\mathbb{C}$;
			\item[(b)] the arrows $h \colon (X,f) \rightarrow (Y,g)$ are all arrows in $\mathcal{C}$ such that $gh=f$.
		\end{itemize}
	\end{defi}
	
	\begin{defi}
		A \emph{relatively admissible adjunction} consists of an adjunction $\S \dashv \C \colon \mathcal{P} \rightarrow \mathcal{A}$ (where $\mathcal{A}$ and $\mathcal{P}$ have pullbacks) and two admissible classes $\mathbb{A} \subseteq \mathcal{A}$, $\mathbb{P} \subseteq \mathcal{P}$ of arrows, such that
		\begin{itemize}
			\item[(a)] $\S(\mathbb{A}) \subseteq \mathbb{P}$,
			\item[(b)] $\C(\mathbb{P}) \subseteq \mathbb{A}$,
			\item[(c)] for every object $A$ of $\mathcal{A}$ the component $\eta_A$ of the unit of the adjunction $\S \dashv \C$ is in $\mathbb{A}$,
			\item[(d)] for every object $X$ of $\mathcal{P}$ the component $\varepsilon_X$ of the counit of the adjunction $\S \dashv \C$ is in $\mathbb{P}$.
		\end{itemize}
		We denote a relatively admissible adjunction with $(\S,\C,\mathbb{A},\mathbb{P})$.
	\end{defi}
	
	In order to recall the notion of an admissible adjunction, remaining consistent with the notation of the previous definitions, we review the definition of the following functors: $\S_A \colon \mathbb{A}/A \rightarrow \mathbb{P}/\S(A)$ and $\C_A \colon \mathbb{P}/\S(A) \rightarrow \mathbb{A}/A$ where $A$ is an object of $\mathcal{A}$. One has $\S_A(f \colon B \rightarrow A)  \coloneqq \S(f)$ and $\C_A(g \colon P \rightarrow \S(A)) \coloneqq \pi_A$ where $\pi_A$ is the upper morphism in the following pullback diagram:
	\[\begin{tikzcd}
		K & A \\
		{\C(P)} & {\C\S(A)}
		\arrow["{\C(g)}"', from=2-1, to=2-2]
		\arrow["{\eta_A}", from=1-2, to=2-2]
		\arrow["{\pi_A}", from=1-1, to=1-2]
		\arrow["{\pi_{\C(P)}}"', from=1-1, to=2-1]
		\arrow["\lrcorner"{anchor=center, pos=0}, draw=none, from=1-1, to=2-2]
	\end{tikzcd}\]
	($\eta_A$ is the component in $A$ of the unit of the adjunction $\S \dashv \C$). It can be proved that $\S_A$ is left adjoint to $\C_A$.
	
	\begin{defi}
		A relatively admissible adjunction $(\S,\C,\mathbb{A},\mathbb{P})$ is \emph{admissible} when the functor $\C_A$ is full and faithful for every object $A$ of $\mathcal{A}$.
	\end{defi}
	
	Now, we recall the definition of effective descent morphism, as this concept will be of crucial importance for the detailed analysis of admissible Galois structures.
	
	\begin{defi}
		Let $\mathbb{A}$ be an admissible class of arrows in a category $\mathcal{A}$ with pullbacks. An arrow $h \colon A \rightarrow B$ is an \emph{effective descent} morphism relatively to $\mathbb{A}$ if:
		\begin{itemize}
			\item[(a)] $h \in \mathbb{A}$;
			\item[(b)] the change-of-base functor $h^* \colon \mathbb{A}/B \rightarrow \mathbb{A}/A$ is monadic.
		\end{itemize}
	\end{defi}
	
	It has been shown that, in Barr-exact categories, the effective descent morphisms relatively to the class of all morphisms are precisely the regular epimorphisms.\\
	
	The notion of effective descent morphism enables the study of specific classes of arrows with respect to an admissible Galois structure.
	
	\begin{defi}
		Given a relatively admissible adjunction $(\S,\C,\mathbb{A},\mathbb{P})$,
		\begin{itemize}
			\item[(a)] a \emph{trivial extension} is an arrow $f \colon A \rightarrow B$ of $\mathbb{A}$ such that the square
			\[\begin{tikzcd}
				A & {\C\S(A)} \\
				B & {\C\S(B)}
				\arrow["f"', from=1-1, to=2-1]
				\arrow["{\C\S(f)}", from=1-2, to=2-2]
				\arrow["{\eta_A}", from=1-1, to=1-2]
				\arrow["{\eta_B}"', from=2-1, to=2-2]
				\arrow["\lrcorner"{anchor=center, pos=0}, draw=none, from=1-1, to=2-2]
			\end{tikzcd}\]
			is a pullback (where $\eta$ is the unit of the adjunction $\S \dashv \C$);
			\item[(b)] a \emph{normal extension} is an arrow $f$ of $\mathbb{A}$ such that it is an effective descent morphism relatively to $\mathbb{A}$ and its kernel pair projections are trivial extensions;
			\item[(c)] a \emph{central extension} is an arrow $f$ of $\mathbb{A}$ such that there exists an effective descent morphism $g$ relatively to $\mathbb{A}$ such that the pullback $g^*(f)$ of $f$ along $g$ is a trivial extension.
		\end{itemize}
	\end{defi}
	
	We are now ready to state and prove the main results of this section. To ensure these, it will not be required to assume that the class $\zeros$ is terminal. However, just as in the case of the previous section, it will be necessary to introduce a property of torsion theories related to the behavior of the class $\zeros$ with respect to certain exact sequences.
	
	\begin{defi}
		A torsion theory $\torsione$ satisfies \emph{Condition (M')} if, for every short exact sequence
		\[\begin{tikzcd}
			\T(A) & A & \F(A),
			\arrow["t_A", from=1-1, to=1-2]
			\arrow["\eta_A", from=1-2, to=1-3]
		\end{tikzcd}\]
		we have $\F(\T(A))=\Z(\F(A))$.
	\end{defi}
	
	Clearly, Condition (M) implies Condition (M').
	
	\begin{prop}\label{prop 1 torsione riflessione}
		Let $\mathcal{C}$ be a regular category with a fixed class $\zeros$ of zero objects. Consider a full, replete, reflective subcategory $\torf$ of $\mathcal{C}$, with inclusion $\i \colon \torf \hookrightarrow \mathcal{C}$ and reflector $\F \colon \mathcal{C} \rightarrow \torf$, such that every component $\eta_A$ of the unit of the adjunction $\F \dashv \i$ is a $\zeros$-cokernel, $\F(Z)=Z$ for every $Z \in \zeros$, and $\F(f^*(\eta_B))$ is an isomorphism for every arrow $f \colon A \rightarrow \F(B)$ of $\torf$ inverted by $\Z$. Then, $\torf$ is the torsion-free part of a torsion theory $\torsione$ satisfying Condition (M') and such that $\tort \cap \torf = \zeros$.
	\end{prop}
	
	\begin{proof}
		We consider the full, replete subcategory $\tort$ whose objects are $\tort \coloneqq \{T \in \mathcal{C} \,|\, T=K[\eta_X] \text{ for some } X \in \mathcal{C} \}$. First of all, we prove that $\tort \cap \torf = \mathcal{Z}$. On the one hand, since $K[\eta_Z]=K[id_Z]=Z$ for every $Z \in \mathcal{Z}$, we have $\mathcal{Z} \subseteq \tort \cap \torf$. On the other hand, consider an object $S \in \tort \cap \torf$. Since $S \in \tort$, we have the pullback
		\[\begin{tikzcd}
			S & {\Z(\F(X))} \\
			X & {\F(X).}
			\arrow["w", from=1-1, to=1-2]
			\arrow[from=1-1, to=2-1]
			\arrow["\lrcorner"{anchor=center, pos=0}, draw=none, from=1-1, to=2-2]
			\arrow["{\varepsilon_{\F(X)}}", tail, from=1-2, to=2-2]
			\arrow["{\eta_X}"', from=2-1, to=2-2]
		\end{tikzcd}\]
		We observe that $\varepsilon_{\F(X)} \in \Arr(\torf)$ (since its domain and its codomain belong to $\torf$) and so, by assumption, $\F(w)$ is an isomorphism ($w$ is the pullback of $\eta_X$ along an arrow of $\torf$ inverted by $\Z$). But $S \in \torf$, so $\F(w)=w$. This implies $S \in \zeros$.\\
		As second step, we prove that every arrow $f \colon T \rightarrow F$, with $T \in \tort$ and $F \in \torf$, is trivial. Due to naturality of $\eta$, we know that the following diagram is commutative:
		\[\begin{tikzcd}
			T & {\F(T)} \\
			F & F.
			\arrow["{\eta_T}", from=1-1, to=1-2]
			\arrow["f"', from=1-1, to=2-1]
			\arrow["{\F(f)}", from=1-2, to=2-2]
			\arrow["{\eta_F=id_F}"', equal, from=2-1, to=2-2]
		\end{tikzcd}\]
		We show that $\F(T) \in \zeros$: recalling that $T=K[\eta_X]$ for some $X$, we consider the pullback
		\[\begin{tikzcd}
			T & {\Z(\F(X))} \\
			X & {\F(X);}
			\arrow["l", from=1-1, to=1-2]
			\arrow[from=1-1, to=2-1]
			\arrow["\lrcorner"{anchor=center, pos=0}, draw=none, from=1-1, to=2-2]
			\arrow["{\varepsilon_{\F(X)}}", tail, from=1-2, to=2-2]
			\arrow["{\eta_X}"', from=2-1, to=2-2]
		\end{tikzcd}\]
		hence, by assumptions, $\F(l)$ is an isomorphism and then $\F(T) \in \zeros$.\\
		Trivially, for every object $A$ the short exact sequence $K[\eta_A] \xrightarrow{k} A \xrightarrow{\eta_A} \F(A)$ is such that $K[\eta_A] \in \tort$ and $\F(A) \in \torf$.\\
		To conclude, we prove that $\torsione$ satisfies Condition (M'). Consider the pullback
		\[\begin{tikzcd}
			{\T(X)} & {\Z(\F(X))} \\
			X & {\F(X)}
			\arrow["\chi", from=1-1, to=1-2]
			\arrow["{t_X}"', from=1-1, to=2-1]
			\arrow["\lrcorner"{anchor=center, pos=0}, draw=none, from=1-1, to=2-2]
			\arrow["{\varepsilon_{\F(X)}}", tail, from=1-2, to=2-2]
			\arrow["{\eta_X}"', from=2-1, to=2-2]
		\end{tikzcd}\]
		and observe that $\F(\chi)$ is an isomorphism ($\chi$ is the pullback of $\eta_X$ along $\varepsilon_{\F(X)}$); therefore, $\F\T(X) \cong \Z\F(X)$.
	\end{proof}
	
	\begin{prop}\label{prop 2 torsione e riflessione}
		Let $\mathcal{C}$ be a regular protomodular category with a fixed class $\zeros$ of zero objects. Consider a torsion theory $\torsione$ in $\mathcal{C}$ satisfying Condition (M'). Then, every component of the unit $\eta$ of the reflector $\F$ is a $\zeros$-cokernel, $\F(f^*(\eta_B))$ is an isomorphism for every $f\colon A \rightarrow \F(B)$ of $\torf$ inverted by $\Z$, and $\F(Z)=Z$ for every $Z \in \zeros$.
	\end{prop}
	
	\begin{proof}
		The only non-trivial statement we need to prove is that $\F(f^*(\eta_B))$ is an isomorphism if $f$ is an arrow of $\torf$ inverted by $\Z$. So, consider the pullback
		\[\begin{tikzcd}
			{A'} & A \\
			B & {\F(B).}
			\arrow["{f''}", from=1-1, to=1-2]
			\arrow["{f'}"', from=1-1, to=2-1]
			\arrow["\lrcorner"{anchor=center, pos=0}, draw=none, from=1-1, to=2-2]
			\arrow["f", from=1-2, to=2-2]
			\arrow["{\eta_B}"', from=2-1, to=2-2]
		\end{tikzcd}\]
		Since $\Z(f)$ is an isomorphism, it follows that $K[f'']=K[\eta_B] \in \tort$ (thanks to \cite[Proposition 4.3]{Cappelletti}). Hence, $f''$ is a $\zeros$-cokernel (this last fact can be proved in the same way as showing that the class $\facte$ is stable under pullback along arrows inverted by the functor $\Z$, see Proposition \ref{Da torsione a sistema}). Now, consider the commutative diagram
		\[\begin{tikzcd}
			{\T(A')} & {A'} && {\F(A')} \\
			&& A,
			\arrow["{t_{A'}}", from=1-1, to=1-2]
			\arrow["{\eta_{A'}}", from=1-2, to=1-4]
			\arrow["{f''}"', from=1-2, to=2-3]
			\arrow["{f'''}", from=1-4, to=2-3]
		\end{tikzcd}\]
		where $f'''$ is induced by the universal property of $\eta_{A'}$ as $\zeros$-cokernel of $t_{A'}$, since $f''t_{A'} \in \zideal$ (it is an arrow with torsion domain and torsion-free codomain). Hence, due to the uniqueness up to isomorphisms of the short exact sequence $\T(A') \xrightarrow{t_{A'}} A' \xrightarrow{\eta_{A'}} \F(A')$, we have that $f'''$ is an isomorphism. To conclude, we observe that $\F(f'')=\F(f''' \eta_{A'})=f'''$ (observing that $\F(f''')=f'''$, since it is an arrow of $\torf$).
	\end{proof}
	
	\begin{lem}\label{F lemma 1}
		Let $\mathcal{C}$ be a regular category and $\torf \hookrightarrow \mathcal{C}$ a full, replete, (regular epi)-reflective subcategory. Then $\torf$ is closed under subobjects in $\mathcal{C}$.
	\end{lem}
	
	\begin{proof}
		Consider the commutative diagram
		\[\begin{tikzcd}
			S && A \\
			{\F(S)} && {\F(A)} \\
			& I,
			\arrow["s", tail, from=1-1, to=1-3]
			\arrow["{\eta_S}"', two heads, from=1-1, to=2-1]
			\arrow[equal, from=1-3, to=2-3]
			\arrow["{\F(s)}"', from=2-1, to=2-3]
			\arrow["e"', two heads, from=2-1, to=3-2]
			\arrow["m"', tail, from=3-2, to=2-3]
		\end{tikzcd}\]
		where the rectangle commutes since $\eta$ is a natural transformation and $\F(s)=me$ is the (regular epi, mono)-factorization of $\F(s)$. Since $e\eta_S$ is a regular epimorphism and $me\eta_S=s$ we conclude that $e \eta_S$ is an isomorphism (due to the uniqueness of the (regular epi, mono)-factorization). So, $\eta_S$ is both a monomorphism and a regular epimorphism.
	\end{proof}
	
	\begin{lem}\label{F lemma 2}
		Let $\mathcal{C}$ be a protomodular category with pullbacks and a fixed class $\zeros$ of zero objects. Let $\F \colon \mathcal{C} \rightarrow \torf$ be the reflector into a full, replete, (regular epi)-reflective subcategory $\torf$. Suppose, moreover, that the reflector is semi-left exact (i.e.\ $\F(f^*(\eta_X))$ is an isomorphism for every component of the unit of the adjunction $\eta_X \colon X \rightarrow \F(X)$ and for every arrow $f \colon F \rightarrow \F(X)$ of $\torf$). Hence, for every pullback
		\[\begin{tikzcd}
			F & X \\
			{Z_1} & {Z_2,}
			\arrow["\theta", from=1-1, to=1-2]
			\arrow["a"', from=1-1, to=2-1]
			\arrow["\lrcorner"{anchor=center, pos=0}, draw=none, from=1-1, to=2-2]
			\arrow["b", from=1-2, to=2-2]
			\arrow["z"', from=2-1, to=2-2]
		\end{tikzcd}\]
		if $F \in \torf$ and $Z_1,Z_2 \in \zeros$, then $X \in \torf$.
	\end{lem}
	
	\begin{proof}
		Consider the commutative diagram
		\[\begin{tikzcd}
			F & X \\
			{F'} & {\F(X)} \\
			{Z_1} & {Z_2,}
			\arrow["\theta", from=1-1, to=1-2]
			\arrow["f"', from=1-1, to=2-1]
			\arrow["\lrcorner"{anchor=center, pos=0}, draw=none, from=1-1, to=2-2]
			\arrow["a"', curve={height=20pt}, from=1-1, to=3-1]
			\arrow["{\eta_X}", two heads, from=1-2, to=2-2]
			\arrow["{\tilde{z}}", from=2-1, to=2-2]
			\arrow["c"', from=2-1, to=3-1]
			\arrow["\lrcorner"{anchor=center, pos=0}, draw=none, from=2-1, to=3-2]
			\arrow["{\F(b)}", from=2-2, to=3-2]
			\arrow["z"', from=3-1, to=3-2]
		\end{tikzcd}\]
		where the two squares are pullbacks. Clearly $c, \tilde{z} \in \Arr(\torf)$ as pullbacks of arrows of $\torf$. Since the reflection is semi-left exact, $\F(f)$ is an isomorphism. Moreover, since in a protomodular category pullbacks reflect monomorphisms (see e.g.\ \cite[Lemma 3.1.20]{Bbbook}), $\eta_X$ is a monomorphism. Therefore, $\eta_X$ is an isomorphism and $X \in \torf$.
	\end{proof}
	
	The following definition is a generalization to our non-pointed context of the notion of a \emph{protoadditive} functor between pointed protomodular categories (introduced in \cite{everaert2010homology}).
	
	\begin{defi}
		Let $\mathcal{C}$ be a protomodular category with pullbacks and a fixed class $\zeros$ of zero objects. A functor $\F \colon \mathcal{C} \rightarrow \torf$, where $\torf$ is a protomodular category, is said to be \emph{$\mathcal{Z}$-protoadditive} (or simply protoadditive if there is no ambiguity) if $\F$ preserves the pullbacks of split epimorphisms along arrows inverted by $\Z$.
	\end{defi}
	
	\begin{defi}
		Let $\mathcal{C}$ be a protomodular category with pullbacks and a fixed class $\zeros$ of zero objects. We say that a torsion theory $\torsione$ in $\mathcal{C}$ satisfies Condition \emph{(P)} if the reflector $\F$ is protoadditive (in our sense).
	\end{defi}
	
	Unlike the previous one, the condition we are about to introduce is not a generalization of a well-known condition for torsion theories in pointed categories. Instead, it introduces a stability condition for torsion objects in our non-pointed context, which is trivially satisfied in the pointed case but necessary to guarantee the desired results. Before proceeding further, after presenting the following definition, we will take a closer look at how this condition relates to the results established so far.
	
	\begin{defi}
		Let $\mathcal{C}$ be a category with a fixed class $\zeros$ of zero objects, and $\torsione$ a torsion theory in $\mathcal{C}$. Consider the following pullback in $\mathcal{C}$
		\[\begin{tikzcd}
			{T'} & {Z'} \\
			T & Z,
			\arrow["\chi", from=1-1, to=1-2]
			\arrow[from=1-1, to=2-1]
			\arrow["\lrcorner"{anchor=center, pos=0}, draw=none, from=1-1, to=2-2]
			\arrow[from=1-2, to=2-2]
			\arrow["{\eta_T}"', from=2-1, to=2-2]
		\end{tikzcd}\]
		where $T \in \tort$, $Z,Z' \in \zeros$, and $\eta_T$ is the $T$-component of the unit of the reflection induced by $\torsione$. We say that $\torsione$ satisfies \emph{Condition (S)} if, for every pullback diagram of the type above, one has $T' \in \tort$ and $\chi=\eta_{T'}$.
	\end{defi}
	
	Let $\mathcal{C}$ be a regular protomodular category with a fixed class $\zeros$ of zero objects. If we consider a stable (with respect to the class of all arrows) factorization system $(\facte, \factm)$ in $\mathcal{C}$ such that every arrow of $\facte$ is a $\zeros$-cokernel, then the induced torsion theory $\torsione$ (described in the previous section) satisfies Condition (S). In fact, $\eta_T \in \facte$ (recall that $T \in \tort$ if and only if its maximum quotient in $\zeros$, i.e.\ $\eta_T$, exists and belongs to $\facte$). So, since the factorization system is stable, we get $\chi \in \facte$ and this implies $T' \in \tort$ and $\chi=\Zcoker(id_{T'})=\eta_{T'}$. Conversely, if $\torsione$ is a torsion theory in $\mathcal{C}$ that satisfies Conditions (M), (N), and (S), then the induced factorization system $(\facte, \factm)$ (described in the previous section) is stable under pullbacks with respect to the class of all arrows. To see this, consider an arrow $e \in \facte$ (i.e.\ $e \colon A \rightarrow Q$ is a $\zeros$-cokernel such that $T \coloneqq K[e] \in \tort$). Given an arrow $g$ of $\mathcal{C}$, we compute $e'$ as the pullback of $e$ along $g$. Hence we construct the commutative cube
	\[\begin{tikzcd}
		& {K[e']} && {\Z(C)} \\
		P && C \\
		& T && {\Z(Q)} \\
		A && Q,
		\arrow["\chi", from=1-2, to=1-4]
		\arrow["{k'}"', from=1-2, to=2-1]
		\arrow["\theta"'{pos=0.3}, from=1-2, to=3-2]
		\arrow[from=1-4, to=2-3]
		\arrow["{\Z(g)}", from=1-4, to=3-4]
		\arrow["{e'}"{pos=0.7}, from=2-1, to=2-3]
		\arrow["{g'}"', from=2-1, to=4-1]
		\arrow["g"{pos=0.3}, from=2-3, to=4-3]
		\arrow["{\eta_T}"{pos=0.3}, from=3-2, to=3-4]
		\arrow["k"', from=3-2, to=4-1]
		\arrow[from=3-4, to=4-3]
		\arrow["e"', from=4-1, to=4-3]
	\end{tikzcd}\]
	where the front face is a pullback, the top and the bottom faces are obtained, respectively, as the pullback diagrams given by the $\zeros$-kernels of $e'$ and of $e$, and $\theta$ is induced by the universal property of the bottom pullback. Moreover, thanks to Condition (M) we know that the arrow from $T$ to $\Z(Q)$ is a maximum quotient of $T$ in $\zeros$. It is easy to see that the square
	\[\begin{tikzcd}
		{K[e']} & {\Z(C)} \\
		T & {\Z(Q)}
		\arrow["\chi", from=1-1, to=1-2]
		\arrow["\theta"', from=1-1, to=2-1]
		\arrow["{\Z(g)}", from=1-2, to=2-2]
		\arrow["{\eta_T}"', from=2-1, to=2-2]
	\end{tikzcd}\]
	is a pullback. Therefore, thanks to Condition (S), we conclude that $K[e'] \in \tort$ and $\chi=\Zcoker(id_{T'})=\eta_{T'}$. As a consequence we have $\chi \in \facte$. Applying \cite[Proposition 14]{Bourn protomod}, we can state that the pullback diagram
	\[\begin{tikzcd}
		{K[e']} & {\Z(C)} \\
		P & C
		\arrow["\lrcorner"{anchor=center, pos=0}, draw=none, from=1-1, to=2-2]
		\arrow[two heads, "\chi", from=1-1, to=1-2]
		\arrow["{k'}"', from=1-1, to=2-1]
		\arrow[from=1-2, to=2-2]
		\arrow[two heads, "{e'}"', from=2-1, to=2-2]
	\end{tikzcd}\]
	is also a pushout. Since the class $\facte$ is pushout stable, we conclude that $e' \in \facte$. To sum up, we have a one-to-one correspondence between stable (with respect to the class of all arrows) factorization systems $(\facte, \factm)$, where every arrow in $\facte$ is a $\zeros$-cokernel, and torsion theories satisfying Conditions (M), (N), and (S).
	
	In Propositions \ref{prop 1 torsione riflessione} and \ref{prop 2 torsione e riflessione}, we examined the relationship between torsion theories and reflective subcategories in a protomodular category with a fixed class $\zeros$ of zero objects. How can these results be refined by introducing Condition (S)? Consider a full, replete, reflective subcategory $\torf$ of $\mathcal{C}$, with reflector $\F \colon \mathcal{C} \rightarrow \torf$, such that every component $\eta_A$ of the unit of the reflection is a $\zeros$-cokernel, $\F(Z)=Z$ for every $Z \in \zeros$, and the reflector is semi-left exact. As proved in Proposition \ref{prop 1 torsione riflessione}, we know that such a reflection defines a torsion theory $\torsione$. Fix a torsion object $T \in \tort$ such that its maximum quotient in $\zeros$ is $\eta_T \colon T \rightarrow Z$, and an arrow $z \colon Z' \rightarrow Z$. Given the pullback
	\[\begin{tikzcd}
		{T'} & {Z'} \\
		T & Z,
		\arrow["{z^*(\eta_T)}", from=1-1, to=1-2]
		\arrow[from=1-1, to=2-1]
		\arrow["\lrcorner"{anchor=center, pos=0}, draw=none, from=1-1, to=2-2]
		\arrow["z", from=1-2, to=2-2]
		\arrow["{\eta_T}"', from=2-1, to=2-2]
	\end{tikzcd}\]
	since the reflection is semi-left exact we can conclude that $\F(z^*(\eta_T))$ is an isomorphism. Hence, $\F(T')=Z'$ i.e.\ $T' \in \tort$ and $\eta_{T'}=z^*(\eta_T)$; this means that $\torsione$ satisfies (S). Conversely, thanks to Proposition \ref{prop 2 torsione e riflessione}, we know that every torsion theory $\torsione$ satisfying Condition (M') defines a reflection whose reflector is $\F \colon \mathcal{C} \rightarrow \torf$. We will prove that if $\torsione$ satisfies (S), then $\F$ is semi-left exact. To do this, let $B$ be an object of $\mathcal{C}$ and $f \colon A \rightarrow \F(B)$ a morphism in $\torf$. Using the $\zeros$-kernel of $\eta_B$ and the $\zeros$-kernel of $f'' \coloneqq f^*(\eta_B)$, we build the following commutative cube:
	\[\begin{tikzcd}
		& {K[f'']} && {\Z(A)} \\
		{A'} && A \\
		& {\T(B)} && {\Z(\F(B))} \\
		B && {\F(B).}
		\arrow[from=1-2, to=1-4]
		\arrow["k"', from=1-2, to=2-1]
		\arrow[from=1-2, to=3-2]
		\arrow["{\varepsilon_A}"', from=1-4, to=2-3]
		\arrow[from=1-4, to=3-4]
		\arrow["{f''}"{pos=0.7}, from=2-1, to=2-3]
		\arrow["{f'}"', from=2-1, to=4-1]
		\arrow["f"{pos=0.2}, from=2-3, to=4-3]
		\arrow[from=3-2, to=3-4]
		\arrow["{t_B}", from=3-2, to=4-1]
		\arrow[from=3-4, to=4-3]
		\arrow["{\eta_B}"', from=4-1, to=4-3]
	\end{tikzcd}\]
	Since the top, the bottom, and the front faces are pullbacks, we deduce that also the back face is a pullback. Hence, thanks to Condition (S), we can state that $K[f''] \in \tort$ and $K[f''] \rightarrow \Z(A)$ is the $K[f'']$-component of the unit of the reflection. Hence, $K[f''] \rightarrow \Z(A)$ is a maximum quotient of $K[f'']$ in $\zeros$. Thanks to \cite[Proposition 14]{Bourn protomod} the top face of the cube is a pushout and, applying Proposition \ref{coker e max quotient}, we obtain that $f''=\Zcoker(k)$. Therefore, the sequence $K[f''] \xrightarrow{k} A' \xrightarrow{f''} A$ is exact and such that $K[f''] \in \tort$, $A \in \torf$; hence, it is isomorphic to the sequence $\T(A') \rightarrow A' \rightarrow \F(A')$. This suffices to conclude that $\F(f'')$ is an isomorphism (since $\F(\eta_{A'})$ is). Recalling that $f''$ is the pullback of $\eta_B$ along $f$, we conclude that the reflection is semi-left exact.
	To sum up, we get the following:
	
	\begin{prop}
		There is a one-to-one correspondence between torsion theories satisfying Conditions (M') and (S), and full, replete, reflective subcategories $\torf \subseteq \mathcal{C}$ with reflector $\F \colon \mathcal{C} \rightarrow \torf$, such that every component $\eta_A$ of the unit of the reflection is a $\zeros$-cokernel, $\F(Z) = Z$ for every $Z \in \zeros$, and the reflector is semi-left exact.
	\end{prop}
	
	Using the well-known fact that a reflection is semi-left exact if and only if it gives an admissible Galois structure w.r.t. the class of all arrows, we get the following result, which is fundamental to apply the methods of categorical Galois theory to our context:
	
	\begin{cor}
		Let $\mathcal{C}$ be a regular protomodular category with a fixed class $\zeros$ of zero objects. Consider a torsion theory $\torsione$ in $\mathcal{C}$ satisfying Conditions (M') and (S). Then, the data $\Gamma \coloneqq (\F, \i, \Arr(\mathcal{C}), \Arr(\torf))$ determine an admissible Galois structure.
	\end{cor}
	
		
		Adding the protoadditivity of the reflector, we are able to characterize in a simple way the effective descent morphisms that are central and normal extensions of this Galois structure:
		
		\begin{teo}\label{centrali e torsione}
			Let $\mathcal{C}$ be a Barr-exact protomodular category with a fixed class $\zeros$ of zero objects. Consider a torsion theory $\torsione$ in $\mathcal{C}$ satisfying Conditions (S), (M'), and (P). Let $f$ be an effective descent morphism in $\mathcal{C}$ (i.e.\ a regular epimorphism, since $\mathcal{C}$ is Barr-exact) and let $\Gamma$ be the Galois structure associated with the reflector $\F$. Then, the following conditions are equivalent:
			\begin{itemize}
				\item[(a)] $f$ is a normal extension for $\Gamma$;
				\item[(b)] $f$ is a central extension for $\Gamma$;
				\item[(c)] $K[f] \in \torf$.
			\end{itemize}
		\end{teo}
		
		\begin{proof}
			(a)$\implies$(b) Always true.\\
			(b)$\implies$(c) Let $p \colon E \rightarrow B$ be an effective descent morphism such that $p^*(f)$ is a trivial extension. Consider the commutative diagram
			\[\begin{tikzcd}
				& {K[p^*(f)]} && {K[f]} \\
				P && A \\
				& {\Z(E)} && {\Z(B)} \\
				E && B,
				\arrow["\theta", from=1-2, to=1-4]
				\arrow[from=1-2, to=2-1]
				\arrow[from=1-2, to=3-2]
				\arrow[from=1-4, to=2-3]
				\arrow[from=1-4, to=3-4]
				\arrow[from=2-1, to=2-3]
				\arrow["{p^*(f)}"', from=2-1, to=4-1]
				\arrow["f"{pos=0.3}, from=2-3, to=4-3]
				\arrow[from=3-2, to=3-4]
				\arrow[from=3-2, to=4-1]
				\arrow[from=3-4, to=4-3]
				\arrow["p"', from=4-1, to=4-3]
			\end{tikzcd}\]
			where the front face, the left face, and the right face are pullbacks. Hence, the back face is a pullback, too. Moreover, we know that the naturality square
			\[\begin{tikzcd}
				P & {\F(P)} \\
				E & {\F(E)}
				\arrow["{\eta_P}", from=1-1, to=1-2]
				\arrow["{p^*(f)}"', from=1-1, to=2-1]
				\arrow["\lrcorner"{anchor=center, pos=0}, draw=none, from=1-1, to=2-2]
				\arrow["{\F(p^*(f))}", from=1-2, to=2-2]
				\arrow["{\eta_E}"', from=2-1, to=2-2]
			\end{tikzcd}\]
			is a pullback. Therefore, we can construct the commutative cube
			\[\begin{tikzcd}
				& {K[p^*(f)]} && {K[\F(p^*(f))]} \\
				P && {\F(P)} \\
				& {\Z(E)} && {\Z(\F(E))} \\
				E && {\F(E)}
				\arrow["\chi", from=1-2, to=1-4]
				\arrow[from=1-2, to=2-1]
				\arrow[from=1-2, to=3-2]
				\arrow[from=1-4, to=2-3]
				\arrow[from=1-4, to=3-4]
				\arrow["{\eta_P}"{pos=0.7}, from=2-1, to=2-3]
				\arrow["{p^*(f)}"', from=2-1, to=4-1]
				\arrow["{\F(p^*(f))}"{pos=0.3}, from=2-3, to=4-3]
				\arrow[from=3-2, to=3-4]
				\arrow[from=3-2, to=4-1]
				\arrow[from=3-4, to=4-3]
				\arrow["{\eta_E}"', from=4-1, to=4-3]
			\end{tikzcd}\]
			where the front face, the left face, and the right face are pullbacks. Hence, the back face is a pullback, too. Since $K[\F(p^*(f))] \in \torf$ (we can apply Lemma \ref{F lemma 1} because it is a subobject of $\F(P) \in \torf$), we get that $K[p^*(f)] \in \torf$ ($\torf$ is closed under limits). To conclude, we recall that
			\[\begin{tikzcd}
				{K[p^*(f)]} & {K[f]} \\
				{\Z(E)} & {\Z(B)}
				\arrow["\theta", from=1-1, to=1-2]
				\arrow[from=1-1, to=2-1]
				\arrow["\lrcorner"{anchor=center, pos=0}, draw=none, from=1-1, to=2-2]
				\arrow[from=1-2, to=2-2]
				\arrow[from=2-1, to=2-2]
			\end{tikzcd}\]
			is a pullback; and so, applying Lemma \ref{F lemma 2} (the reflection is semi-left exact thanks to Condition (S)), we get $K[f] \in \torf$.\\
			(c)$\implies$(a) Let $(\Eq(f), \pi_1, \pi_2)$ be the kernel pair of $f$, and $\Delta= \langle 1_A, 1_A \rangle$ the canonical morphism. Consider the commutative cube
			\[\begin{tikzcd}
				& {K[\pi_1]} && {K[f]} \\
				{\Eq(f)} && A \\
				& {\Z(A)} && {\Z(B)} \\
				A && B,
				\arrow[from=1-2, to=1-4]
				\arrow[from=1-2, to=2-1]
				\arrow[shift right, from=1-2, to=3-2]
				\arrow[from=1-4, to=2-3]
				\arrow[from=1-4, to=3-4]
				\arrow["{\pi_2}"{pos=0.7}, from=2-1, to=2-3]
				\arrow["{\pi_1}"', shift right, from=2-1, to=4-1]
				\arrow["f"{pos=0.3}, from=2-3, to=4-3]
				\arrow[shift right, from=3-2, to=1-2]
				\arrow[from=3-2, to=3-4]
				\arrow[from=3-2, to=4-1]
				\arrow[from=3-4, to=4-3]
				\arrow["\Delta"', shift right, from=4-1, to=2-1]
				\arrow["f"', from=4-1, to=4-3]
			\end{tikzcd}\]
			where the front face, the left face, and the right face are pullbacks. Hence, the back face is a pullback and then $K[\pi_1] \in \torf$ (by assumption $K[f] \in \torf$ and $\torf$ is closed under limits). We apply the functor $\F$ to the left face of the diagram above, and since $\F$ is $\mathcal{Z}$-protoadditive we obtain the pullback (in $\mathcal{C}$)
			\[\begin{tikzcd}
				{\F(K[\pi_1])=K[\pi_1]} & {\F(\Eq(f))} \\
				{\F(\Z(A))=\Z(A)} & {\F(A).}
				\arrow[from=1-1, to=1-2]
				\arrow[from=1-1, to=2-1]
				\arrow["\lrcorner"{anchor=center, pos=0}, draw=none, from=1-1, to=2-2]
				\arrow[from=1-2, to=2-2]
				\arrow[from=2-1, to=2-2]
			\end{tikzcd}\]
			We construct the diagram
			\[\begin{tikzcd}
				& {K[\pi_1]} && {K[\pi_1]} \\
				{\Eq(f)} && {\F(\Eq(f))} \\
				& {\Z(A)} && {\Z(A)} \\
				A && {\F(A),}
				\arrow[equal, from=1-2, to=1-4]
				\arrow[equal, from=3-2, to=3-4]
				\arrow[from=1-2, to=2-1]
				\arrow[from=1-2, to=3-2]
				\arrow[from=1-4, to=2-3]
				\arrow[from=1-4, to=3-4]
				\arrow["{\eta_{\Eq(f)}}"{pos=0.7}, from=2-1, to=2-3]
				\arrow["{\pi_1}"', shift right, from=2-1, to=4-1]
				\arrow["{\F(\pi_1)}"{pos=0.3}, from=2-3, to=4-3]
				\arrow[from=3-2, to=4-1]
				\arrow[from=3-4, to=4-3]
				\arrow["\Delta"', shift right, from=4-1, to=2-1]
				\arrow["{\eta_A}"', from=4-1, to=4-3]
			\end{tikzcd}\]
			where the left face, the right face, and the back face are pullbacks. By protomodularity, applying \cite[Proposition 3.1.24]{Bbbook}, we obtain that the front face is a pullback, and so $f$ is a normal extension.
		\end{proof}
		
		\section{Examples} \label{examples}
		
		In this section we provide several examples of torsion theories in regular categories with a fixed class $\zeros$ of zero objects.
		
		In the examples we are about to consider, we will often work with categories where the only quotient of the initial object is the terminal object, and hence, in these cases, $\zeros = \graffe{\iniziale, \terminal}$. Another important property that holds in many of the examples we will present is the following: the terminal object is strict (i.e., every arrow $f \colon \terminal \rightarrow A$ is an isomorphism).
		Let us first observe that under these assumptions, for any object $A \neq \terminal$, the morphism $\iniziale \rightarrow A$ is a monomorphism. To see this, consider the (regular epi, mono)-factorization of $\iniziale \rightarrow A$
		\[\begin{tikzcd}
			\iniziale && A \\
			& I.
			\arrow[from=1-1, to=1-3]
			\arrow[two heads, from=1-1, to=2-2]
			\arrow[tail, from=2-2, to=1-3]
		\end{tikzcd}\]
		Since $A$ is not the terminal object, we deduce that $I$ cannot be terminal either. Therefore, since $I$ is a quotient of $\iniziale$, the only possibility is that $I$ is $\iniziale$ itself. Finally, note that any torsion theory in a category satisfying these assumptions also satisfies Condition (M). Indeed, consider an exact sequence $T \xrightarrow{k} A \xrightarrow{q} Q$, where $T \in \tort$, and observe that $\Z(Q) = \terminal$ if and only if $Q = \terminal$ (recall that $\terminal$ is strict). Moreover, since $q = \Zcoker(k)$ and $k = \Zker(q)$, we have that $Q = \terminal$ if and only if the maximum quotient of $T$ in $\zeros$ is $\terminal$. This completes the proof.
		
		\subsection{Semisimple and perfect MV-algebras}
		We briefly recall some notions concerning MV-algebras; more information can be found in \cite{cignoli2013algebraic}. An \emph{MV-algebra} is an algebra $(A, \oplus, \lnot, 0)$ with a binary operation $\oplus$, a unary operation $\lnot$, and a constant $0$ such that $(A, \oplus, 0)$ is a commutative monoid, $\lnot$ is an involution, $\lnot 0$ is the absorbing element of $\oplus$, and for every $x,y \in A$
		\[ \lnot (\lnot x \oplus y) \oplus y=\lnot (\lnot y \oplus x) \oplus x. \]
		The category $\mv$ is the variety of universal algebras determined by this equational theory. Given an MV-algebra $A$, we define the constant $1$ and the binary operations $\odot$, $\ominus$, $\rightarrow$ and $d$ as follows:
		\[ 1 \coloneqq \lnot 0; \]
		\[ x \odot y \coloneqq \lnot (\lnot x \oplus \lnot y); \]
		\[ x \rightarrow y \coloneqq \lnot x \oplus y; \]
		\[ x \ominus y \coloneqq \lnot (\lnot x \oplus y)=\lnot (x \rightarrow y)=x \odot \lnot y; \]
		\[ d(x,y) \coloneqq (x \ominus y) \oplus (y \ominus x) \]
		(this last operation is called \emph{distance} between $x$ and $y$). In every MV-algebra $A$, it is possible to introduce a partial order: we say that $x \leq y$ if and only if $\lnot x \oplus y = 1$. With respect to this order, the MV-algebra $A$ can be shown to be a distributive lattice. In this structure, the least element is $0$ and the greatest element is $1$. In an MV-algebra $A$, an ideal $I$ is a non-empty subset of $A$ such that:
		\begin{itemize}
			\item[-] if $x \in I$ and $y \leq x$, then $y \in I$ (downward closure);
			\item[-] for any $x, y \in I$, the element $x \oplus y \in I$ (closure under addition).
		\end{itemize}
		
		Given an ideal $I$ of $A$, it is possible to introduce a congruence $\theta$ on $A$ induced by $I$: we say that $(x, y) \in \theta$ if and only if $d(x, y) \in I$. Moreover, we denote by $A/I$ the algebra obtained as the quotient of $A$ w.r.t.\ the congruence determined by $I$.
		We define the \emph{radical} of an MV-algebra $A$ as the intersection of its maximal ideals, and we denote it with $\Rad(A)$. An MV-algebra is said to be \emph{semisimple} if its radical is trivial. We will also consider perfect MV-algebras, which are defined as follows: an MV-algebra $A$ is said to be \emph{perfect} if $A = \Rad(A) \cup \lnot \Rad(A)$ (where, given $S \subseteq A$, $\lnot S \coloneqq \graffe{ a \in A \,|\, \lnot a \in S}$). In $\mv$ the initial object is $\initial \coloneqq \graffe{0,1}$ and the terminal one is $\terminal \coloneqq \graffe{0=1}$; one can easily see that the terminal object is strict. So, the class of quotients of the initial object is $\graffe{\terminal, \initial}$; this class will be our class $\zeros$ of zero objects. We will show that, with respect to $\zeros$, the full subcategory of semisimple MV-algebras and the full subcategory of perfect MV-algebras constitute the torsion-free and torsion parts, respectively, of a torsion theory on $\mv$.\\
		
		Let $A$ be an MV-algebra. An element $a$ in $A$ is said to be \emph{infinitely small} or \emph{infinitesimal} if $a \neq 0$ and $na \leq \lnot a$ for every integer $n \geq 0$. The set of all infinitesimals in $A$ will be denoted by $\Inf(A)$.

		\begin{prop}[\cite{cignoli2013algebraic}, Lemma 7.3.3]
			Let $A$ be an MV-algebra. One has
			\[ \Rad(A)=\Inf(A) \cup \{ 0 \}. \]
		\end{prop}
		
		Given an MV-algebra $A$, we define $\S(A) \coloneqq A/\Rad(A)$ and $\P(A) \coloneqq \Rad(A) \cup \lnot \Rad(A)$. Thanks to \cite[Lemma 3.6.6]{cignoli2013algebraic} we obtain that $\S(A)$ is semisimple. Moreover, it is possible to see that $\P(A)$ is a perfect subalgebra of $A$.
		If we denote by $s\mv$ the full subcategory of $\mv$ whose objects are the semisimple MV-algebras, we get a functor $\S$ described by the following assignment:
		\[\begin{tikzcd}[row sep=5pt]
			\mv & s\mv \\
			A & {\S(A)=A/\Rad(A)} \\
			\\
			\\
			B & {\S(B)=B/\Rad(B),}
			\arrow["\S",from=1-1, to=1-2]
			\arrow["f"', from=2-1, to=5-1]
			\arrow[maps to, from=2-1, to=2-2]
			\arrow[maps to, from=5-1, to=5-2]
			\arrow["{\S(f)=\overline{f}}", from=2-2, to=5-2]
			\arrow[from=1-1, to=1-2]
		\end{tikzcd}\]
		where $\overline{f}([a]) \coloneqq [f(a)]$, for every $[a] \in \S(A)$. We have to show that $\overline{f}$ is well defined: $a \in \Rad(A)$ if and only if $na \leq \lnot a$ for every $n \in \mathbb{N}$, thus $nf(a)=f(na) \leq f(\lnot a)=\lnot f(a)$ for every $n \in \mathbb{N}$, and so $f(a) \in \Rad(B)$.\\
		Similarly, if we denote by $p\mv$ the full subcategory of $\mv$ whose objects are the perfect MV-algebras, we get a functor $\P$ described by the following assignment:
		\[\begin{tikzcd}[row sep=5pt]
			\mv & p\mv \\
			A & {\P(A)=\Rad(A) \cup \lnot \Rad(A)} \\
			\\
			\\
			B & {\P(B)=\Rad(B) \cup \lnot \Rad(B),}
			\arrow["f"', from=2-1, to=5-1]
			\arrow[maps to, from=2-1, to=2-2]
			\arrow[maps to, from=5-1, to=5-2]
			\arrow["{\P(f)}", from=2-2, to=5-2]
			\arrow["\P", from=1-1, to=1-2]
		\end{tikzcd}\]
		where $\P(f)(a)=f(a)$ for every $a \in \P(A).$ We have to show that $\P(f)$ is well defined: as we saw before, if $a \in \Rad(A)$ then $f(a) \in \Rad(B)$, and so $f \colon A \rightarrow B$ restricts to $\P(f) \colon \P(A) \rightarrow \P(B)$.
		
		\begin{oss}
			Given an MV-algebra $A$ and an ideal $I$, the set $I \cup \lnot I$ is always a subalgebra of $A$. Now, consider a morphism $f \colon A \to B$ in $\mv$. We know that the $\zeros$-kernel of $f$ is given by the pullback
			\[\begin{tikzcd}
				{K[f]} & Z \\
				A & B,
				\arrow[from=1-1, to=1-2]
				\arrow["k"', from=1-1, to=2-1]
				\arrow["\lrcorner"{anchor=center, pos=0}, draw=none, from=1-1, to=2-2]
				\arrow[from=1-2, to=2-2]
				\arrow["f"', from=2-1, to=2-2]
			\end{tikzcd}\]
			where $Z=\terminal$ if $B=\terminal$, otherwise $Z=\initial$. It follows that $K[f]=I \cup \lnot I$, where $I=\invf{0}$ (it is a well-known fact that the preimage of $0$ under a morphism of MV-algebras is always an ideal).
		\end{oss}
		
		\begin{prop}
			$(p\mv,s\mv)$ is a torsion theory for $\mv$.
		\end{prop}
		
		\begin{proof}
			First of all we show that
			\[ p\mv \cap s\mv= \{ \terminal, \initial \}. \]
			Let $X$ be a perfect and semisimple MV-algebra; then for every $x \in X$ one has $x \in \Rad(X)=\graffe{0}$ or $\lnot x \in \Rad(X)=\graffe{1}$. Next, we prove that, given two MV-algebras $A,B$, with $A$ perfect and $B$ semisimple, $\Hom_{\mv}(A,B)\subseteq \zideal$. Let us suppose $B \neq \terminal$ (this implies $A \neq \terminal$); then a morphism $f \colon A \rightarrow B$ factors in the following way:
			\[\begin{tikzcd}
				A && B \\
				& {\initial,}
				\arrow["f", from=1-1, to=1-3]
				\arrow["{\chi_{\lnot \Rad(A)}}"', from=1-1, to=2-2]
				\arrow[from=2-2, to=1-3]
			\end{tikzcd}\]
			where
			\[ \chi_{\lnot \Rad(A)}(a) \coloneqq \begin{cases}
				0 & a \in \Rad(A) \\
				1 & a \in \lnot \Rad(A).
			\end{cases}
			\]
			In fact, if $a \in A$ then $a \in \Rad(A)$ or $a \in \lnot \Rad(A)$, since $A$ is perfect; so $f(a) \in \Rad(B)=\{ 0 \}$ or $f(a) \in \lnot \Rad(A)=\{ 1 \}$, because $B$ is semisimple.
			If $B= \terminal$ we have
			\[\begin{tikzcd}
				A && {\terminal} \\
				& {\terminal.}
				\arrow["f", from=1-1, to=1-3]
				\arrow["f"', from=1-1, to=2-2]
				\arrow[equals, from=2-2, to=1-3]
			\end{tikzcd}\]
			Now, we show that, for every MV-algebra $B$, we have an exact sequence with $B$ in the middle, a torsion object on the left, and a torsion-free object on the right. If $B= \terminal$ the sequence is given by $B=B=B$. Suppose $B \neq \terminal$; we prove that
			\[\begin{tikzcd}
				{\P(B)} & B & {\S(B)}
				\arrow["{t_B}",from=1-1, to=1-2]
				\arrow["{\eta_B}", from=1-2, to=1-3]
			\end{tikzcd}\]
			is an exact sequence. Observe that, since $\eta_B(b)=0$ for every $b \in \Rad(B)$, $\eta_B t_B \in \zideal$. We show that $t_B$ is a $\zeros$-kernel of $\eta_B$. Consider an arrow $\lambda_B \colon C \rightarrow B$ such that there exists an arrow $x \colon C \rightarrow \initial$ making the diagram below commutative
			\[\begin{tikzcd}
				B & {\S(B)} \\
				C & {\initial.}
				\arrow["{\eta_B}", from=1-1, to=1-2]
				\arrow["{\lambda_B}", from=2-1, to=1-1]
				\arrow[from=2-2, to=1-2]
				\arrow["x"', from=2-1, to=2-2]
			\end{tikzcd}\]
			Notice that if $\eta_B \lambda_B$ factored through $\terminal$, then we would have $\S(B)=\terminal$; since $\S(B)=\terminal$ implies $1 \in \Rad(B)$, it would follow $B=\terminal$. Now, for every $c \in C$, if $x(c)=0$ then $\lambda_B(c) \in \Rad(B)$ and, if $x(c)=1$, then $\lambda_B(c) \in \lnot \Rad(B)$; hence, $\lambda_B$ restricts to $\lambda'_B \colon C \rightarrow \P(B)$ and the following diagram commutes:
			\[\begin{tikzcd}[column sep=5pt]
				{\P(B)} && B && {\S(B)} \\
				& C && {\initial.}
				\arrow[from=2-4, to=1-5]
				\arrow["x"', from=2-2, to=2-4]
				\arrow["{\lambda'_B}", from=2-2, to=1-1]
				\arrow["{\lambda_B}", from=2-2, to=1-3]
				\arrow["{\eta_B}", from=1-3, to=1-5]
				\arrow["{t_B}", from=1-1, to=1-3]
			\end{tikzcd}\]
			Therefore, $t_B$ is a $\zeros$-kernel of $\eta_B$.\\
			Next, we focus on $\eta_B$: we want to show that it is a $\zeros$-cokernel of $t_B$. Fix an arrow $\theta_B \colon B \rightarrow C$ such that $\theta_B t_B \in \zideal$. If $\theta_B t_B$ factors through $\terminal$, we can conclude that $C=\terminal$. Hence, the claim becomes trivial. Then, suppose there exists an arrow $y \colon \P(B) \rightarrow \initial$ making the following diagram commute:
			\[\begin{tikzcd}
				{\P(B)} & B \\
				{\initial} & C.
				\arrow["{\theta_B}", from=1-2, to=2-2]
				\arrow["{t_B}", from=1-1, to=1-2]
				\arrow["y"', from=1-1, to=2-1]
				\arrow[from=2-1, to=2-2]
			\end{tikzcd}\]
			If $b \in \Rad(B)$ then $\theta_B(b) \in \initial \subseteq C$; but, if $\theta_B(b)=1$, we get $1 \in \Rad(C)$ and so $C=\terminal$ (we are excluding this case). Hence, $\theta_B$ defines a unique morphism $\theta'_B \colon \S(B) \rightarrow C$ such that the diagram below is commutative
			\[\begin{tikzcd}[column sep=5pt]
				{\P(B)} && B && {\S(B)} \\
				& {\initial} && C.
				\arrow["{t_B}", from=1-1, to=1-3]
				\arrow["{\eta_B}", from=1-3, to=1-5]
				\arrow[from=2-2, to=2-4]
				\arrow["y"', from=1-1, to=2-2]
				\arrow["{\theta_B}", from=1-3, to=2-4]
				\arrow["{\theta'_B}", from=1-5, to=2-4]
			\end{tikzcd}\]
			Therefore, $\eta_B$ is a $\zeros$-cokernel of $t_B$.
		\end{proof}
		
		\begin{prop}\label{ideali in MV}
			The torsion theory $(p\mv,s\mv)$ in $\mv$ satisfies Conditions (M), (S), and (N).
		\end{prop}
		
		\begin{proof}
			Since $\mv$ is a regular category in which the terminal object is strict, then every torsion theory satisfies Condition (M) (as observed at the beginning of this section).
			We continue by noting that $\S(A) = \terminal$ if and only if $A = \terminal$: clearly, $A = \terminal$ implies $\S(A) = \terminal$; conversely, if $\S(A) = \terminal$, then $1 \in \Rad(A)$. Assume, for contradiction, that $1 \neq 0$; then $1 \in \Inf(A)$, which gives $1 \leq 0$ (and so $1=0$). This immediately implies the validity of (S): in fact, the only non-isomorphic arrow between zero objects is $\initial \to \terminal$, but we have just seen that $\S(A) = \terminal$ if and only if $A = \terminal$. Therefore, the validity of (S) reduces to observing that $\initial \in p\mv$. Let us finally show that Condition (N) holds. We consider a morphism $f$ and the diagram $\P(K[f]) \xrightarrow{t} K[f] \xrightarrow{k} A \xrightarrow{f} B$, where $k = \Zker(f)$ (in our non-pointed sense) and $t$ is the inclusion of the perfect part of $K[f]$. We know that $K[f] = I \cup \lnot I$, with $I = \invf{0}$.
			
			We aim to show that $\Rad(I \cup \lnot I) = I \cap \Rad(A)$. In general, it can be verified that if $B$ is a subalgebra of $A$, then $\Rad(B) = \Rad(A) \cap B$ (this fact essentially follows from the representation $\Rad(B) = \Inf(B) \cup \graffe{0}$). Thus, we have:
			$$\Rad(I \cup \lnot I)=(I \cup \lnot I) \cap \Rad(A)=(I \cap \Rad(A)) \cup (\lnot I \cap \Rad(A)).$$
			
			Next, we observe that if $\lnot I \cap \Rad(A)$ is non-empty, then $I = A$. Indeed, let $a \in \lnot I \cap \Rad(A)$. Then $a \leq \lnot a$ (since $a \in \Rad(A)$), and additionally, $\lnot a \in I$ (because $a \in \lnot I$). Therefore, from $a \leq \lnot a$, it follows that $a \in I$. Consequently, we have $a \oplus \lnot a = 1 \in I$, which implies that $I = A$.
			From this, we conclude that $\Rad(I \cup \lnot I) = I \cap \Rad(A)$ holds whether $\lnot I \cap \Rad(A)$ is empty (in which case it is obvious) or whether $\lnot I \cap \Rad(A)$ is non-empty (since, in this case, $A = I$). Finally, we note that
			$$\P(K[f])=\Rad(K[f]) \cup \lnot \Rad(K[f])=(I \cap \Rad(A)) \cup \lnot (I \cap \Rad(A)),$$
			and therefore, the inclusion of $\P(K[f])$ in $A$ can be seen as the $\zeros$-kernel of the quotient $A \to A/(I \cap \Rad(A))$.
		\end{proof}
		
		\begin{prop}
			The functor $\S \colon \mv \rightarrow s\mv$ is protoadditive.
		\end{prop}
		
		\begin{proof}
			It is clear that $\S(\terminal)=\terminal/\terminal=\terminal$ and $\S(\initial)=\initial/\{0\}=\initial$. Moreover, since $s\mv$ is a full reflective subcategory of $\mv$, we know that $s\mv$ is closed under limits. We consider, in $\mv$, the pullback of a split epimorphism along an arbitrary morphism
			\[\begin{tikzcd}
				{A \times_B C} & C \\
				A & B.
				\arrow["p", shift left=1, from=2-1, to=2-2]
				\arrow["s", shift left=1, from=2-2, to=2-1]
				\arrow["g", from=1-2, to=2-2]
				\arrow["{\pi_C}", from=1-1, to=1-2]
				\arrow["{\pi_A}"', from=1-1, to=2-1]
				\arrow["\lrcorner"{anchor=center, pos=0}, draw=none, from=1-1, to=2-2]
			\end{tikzcd}\]
			We observe that
			\begin{align*}
				\Rad(A \times_B C)&= \{ (a,c) \in A \times C \, | \, a \in \Rad(A), \, c \in \Rad(C), \, p(a)=g(c) \}\\
				&=\Rad(A \times C) \cap (A \times_B C).
			\end{align*}
			We now proceed to compute the pullback of $\S(p)$ along $\S(g)$ in $s\mv$:
			\[\begin{tikzcd}
				{\S(A \times_B C)} \\
				& {\S(A) \times_{\S(B)} \S(C)} & {\S(C)} \\
				& {\S(A)} & {\S(B).}
				\arrow["{\S(p)}", shift left=1, from=3-2, to=3-3]
				\arrow["{\S(s)}", shift left=1, from=3-3, to=3-2]
				\arrow["{\pi_{\S(A)}}"', shift right=1, from=2-2, to=3-2]
				\arrow["{\S(g)}", shift right=1, from=2-3, to=3-3]
				\arrow["{\pi_{\S(C)}}", shift left=1, from=2-2, to=2-3]
				\arrow["\lrcorner"{anchor=center, pos=0}, draw=none, from=2-2, to=3-3]
				\arrow["{\S(\pi_A)}"', curve={height=24pt}, from=1-1, to=3-2]
				\arrow["{\S(\pi_C)}", curve={height=-24pt}, from=1-1, to=2-3]
				\arrow["{\varphi}", from=1-1, to=2-2]
			\end{tikzcd}\]
			Using the fact that $\S(p)\S(\pi_A)=\S(p\pi_A)=\S(g\pi_C)=\S(g)\S(\pi_C)$, the universal property of this pullback defines an arrow $\varphi \colon \S(A \times_B C) \rightarrow \S(A) \times_{\S(B)} \S(C)$ such that $\varphi([a,c])=([a],[c])$. We will prove that $\varphi$ is an isomorphism. It is not difficult to see that $\varphi$ is injective: if $\varphi([a,c])=([0],[0])$, then we have $(a,c) \in A \times_B C$, $a \in \Rad(A)$, and $c \in \Rad(C)$. Thus, $(a,c) \in \Rad(A \times C) \cap (A \times_B C)$, and so $[a,c]=[0,0]$. Now, let us show that $\varphi$ is surjective. Let $([a],[c]) \in \S(A) \times_{\S(B)} \S(C)$. Since $\S(p)([a])=\S(g)([c])$, we have $[p(a)]=[g(c)]$, and so $[sp(a)]=[sg(c)]$. Our goal is to find $a' \in A$ and $c' \in C$ such that $p(a')=g(c')$ and $[a']=[a]$, $[c']=[c]$. First, we observe that $(sp(a) \ominus a, 0) \in A \times_B C$ and $(a \ominus sp(a), 0) \in A \times_B C$. By the definition of $\varphi$, we get $\varphi([sp(a) \ominus a, 0])=([sp(a) \ominus a],[0])$ and $\varphi([a \ominus sp(a), 0])=([a \ominus sp(a)],[0])$. Additionally, we have $(sg(c),c) \in A \times_B C$ and $\varphi([sg(c),c])=([sg(c)],[c])=([sp(a)],[c])$ (since $[sp(a)]=[sg(c)]$). Let us write $x \coloneqq [a \ominus sp(a), 0]$, $y \coloneqq [sp(a) \ominus a, 0]$,  and $z \coloneqq [sg(c),c]$, for simplicity. Recalling that $$(a \ominus sp(a)) \oplus ((a \oplus \lnot sp(a)) \odot sp(a))=a,$$ we get $$\varphi (x \oplus (\lnot y \odot z))=([a],[c]).$$ Thus, $\varphi$ is surjective, and so it is an isomorphism.
		\end{proof}
		
		\begin{oss}
			Let $B$ be a subalgebra of an MV-algebra $A$. We know that $\Rad(B) = B \cap \Rad(A)$. It follows that $B$ is perfect (i.e.\ $B = \P(B)$) if and only if $B=(B \cap \Rad(A)) \cup (B \cap \lnot \Rad(A))$. Therefore, if $B = I \cup \lnot I$, for some ideal $I$ of $A$, using what we observed in the proof of Proposition \ref{ideali in MV}, we have that $B$ is perfect if and only if $I \subseteq \Rad(A)$. Similarly, $B = I \cup \lnot I$ is semisimple if and only if $I \cap \Rad(A) = \{ 0 \}$.
		\end{oss}
		
		We are now ready to describe the factorization system induced by the torsion theory $(p\mv, s\mv)$. We begin by defining
		\[ \facte \coloneqq \{ e \colon A \rightarrow B \in \Arr(\mv) \, | \, e \text{ is surjective, } \inve \subseteq \Rad(A) \} \]
		and
		\[ \factm \coloneqq \{ m \colon A \rightarrow B \in \Arr(\mv) \, | \, \invm \cap \Rad(A)= \{0 \} \}. \]
		Thanks to the previous remark, we know that the arrows of $\facte$ all have a perfect $\zeros$-kernel, while the arrows of $\factm$ have a semisimple $\zeros$-kernel. Now, given an arrow $e \colon A \rightarrow B$ in $\facte$, it is easy to observe that $e$ is the $\zeros$-cokernel of the inclusion of $K[e]$ into $A$: indeed, since $e$ is surjective, we have $B \cong A/\inve$. Thus, $(\facte, \factm)$ is the desired factorization system. Specifically, given an arrow $f \in \Arr(\mv)$, the $(\facte, \factm)$-factorization is given by:
		\[\begin{tikzcd}
			A && B \\
			& {A/(\invf{0}\cap \Rad(A)),}
			\arrow["f", from=1-1, to=1-3]
			\arrow["e"', two heads, from=1-1, to=2-2]
			\arrow["m"', from=2-2, to=1-3]
		\end{tikzcd}\]
		where $e$ is the quotient of $A$ over the ideal $\invf{0} \cap \Rad(A)$, while $m$ is defined by the universal property of the quotient.
		
		Now, we shift our focus to the Galois structure induced by the torsion theory.
		We recall that in the category of MV-algebras the effective descent morphisms are exactly the regular epimorphisms. Hence, a surjective morphism $f$ of MV-algebras is a central extension for the Galois structure $\Gamma_{s\mv}$, induced by the torsion theory $(p\mv, s\mv)$, if and only if $K[f]$ is a semisimple algebra. Thanks to the previous remark, this last statement is equivalent to saying that $\invf{0} \cap \Rad(A) = \{0\}$.
		
		More details on this example can be found in \cite{CappellettiMV}.
		
		\subsection{$\mset$ and fix points}
		It is a well-known fact that, for every monoid $M$, the category $\mset$ of sets with a fixed action of $M$ is an elementary topos. Clearly, $\mset$ is a two-valued topos ($\Sub(\terminal)=\{ \terminal, \emptyset \}$). Given an object $X$ of $\mset$, we consider the set of fix points
		\[\Fix(X) \coloneqq \{ x \in X \, | \, mx=x \text{ for every } m \in M \}. \]
		We observe that, for every arrow $f \colon \terminal=\{ * \} \rightarrow X$ in $\mset$ and every
		$m \in M$, one has $mf(*)=f(m*)=f(*)$. Then, there is a bijection $$\Hom_{\mset}(\terminal, X) \cong \Fix(X).$$ Therefore, the objects of $\mset$ for which there exists a unique arrow $\terminal \rightarrow X$ are precisely the ones with exactly one fix point. In the light of what has just been said, we define two full subcategories of $\mset$ whose objects are $$\torf \coloneqq \{ X \in \mset \, | \, \Fix(X)=X \} \text{ and } \tort \coloneqq \{ X \in \mset \, | \, |\Fix(X)|\leq 1 \}.$$
		From this point on, we will consider as class of zero objects $\zeros \coloneqq \torf \cap \tort = \{ \terminal, \emptyset \}$.
		
		We begin by studying the $\zeros$-kernels and $\zeros$-cokernels. Given an arrow $f \colon A \rightarrow B$ (with $A \neq \emptyset$) in $\mset$, $q= \Zcoker(f)$ is defined by the following pushout diagram
		\[\begin{tikzcd}
			A & B \\
			\terminal & {Q[f].}
			\arrow["f", from=1-1, to=1-2]
			\arrow["q", from=1-2, to=2-2]
			\arrow[from=1-1, to=2-1]
			\arrow[from=2-1, to=2-2]
			\arrow["\lrcorner"{anchor=center, pos=0, rotate=180}, draw=none, from=2-2, to=1-1]
		\end{tikzcd}\]
		If $A = \emptyset$, we get $\Zcoker(\emptyset \rightarrow B)=id_B$. Additionally, since colimits in $\mset$ are computed as in $\mathbf{Set}$, we obtain that $Q[f] \cong B/f(A)$, where, given an object $X$ of $\mset$ and a subset $Y \subseteq X$ closed under the action of $M$, $X/Y$ is the $\mset$ obtained by contracting $Y$ to one point and defining the action of $M$ as the one induced by $X$.
		Finally, if there exists a unique arrow $\terminal \rightarrow B$, the $\zeros$-kernel $k$ of an arrow $f \colon A \rightarrow B$ is given by the following pullback diagram:
		\[\begin{tikzcd}
			{K[f]} & \terminal \\
			A & B.
			\arrow["f"', from=2-1, to=2-2]
			\arrow["k"', from=1-1, to=2-1]
			\arrow[from=1-1, to=1-2]
			\arrow[from=1-2, to=2-2]
			\arrow["\lrcorner"{anchor=center, pos=0}, draw=none, from=1-1, to=2-2]
		\end{tikzcd}\]
		
		\begin{prop}
			$\torsione$ is a torsion theory in $\mset^{op}$.
		\end{prop}
		
		\begin{proof}
			To prove our statement, we work in $\mset$. First of all, we consider $F \in \torf$ and $T \in \tort$. If $\Fix(T) \neq \emptyset$, then every arrow $h \colon F \rightarrow T$ factors as
			\[\begin{tikzcd}
				F && T \\
				& \terminal
				\arrow["h", from=1-1, to=1-3]
				\arrow[from=1-1, to=2-2]
				\arrow[from=2-2, to=1-3]
			\end{tikzcd}\]
			since, for every $x \in F$ and every $m \in M$, $mh(x)=h(mx)=h(x)=y$ (where $y$ is the only fix point of $T$). If $\Fix(T)= \emptyset$, then $F$ must be empty (since every image of a fix element is a fix element), and so $h \colon \emptyset \to T \in \zideal$. Hence $\Hom_{\mset}(F,T) \subseteq \zideal$. Now, given an object $X$ of $\mset$, we define $$\F(X) \coloneqq \Fix(X) \text{, } \T(X) \coloneqq X/\Fix(X),$$ and we have the sequence
			\[\begin{tikzcd}
				{\F(X)} & X & {\T(X)}
				\arrow["{t_X}", from=1-1, to=1-2]
				\arrow["{\eta_X}", from=1-2, to=1-3]
			\end{tikzcd}\]
			where $t_X$ is the inclusion and $\eta_X$ the quotient projection. Clearly $\Fix(\Fix(X))=\Fix(X)$ and $|\Fix(X/\Fix(X))| \leq 1$ (it is exactly $1$ when $\Fix(X) \neq \emptyset$, and the unique fix point is $[x]$ with $x \in \Fix(X)$). We prove that the sequence above is exact. If $\Fix(X) \neq \emptyset$, the following square (which we know to be a pushout, thanks to the description of $\zeros$-cokernels)  is a pullback
			\[\begin{tikzcd}
				{\Fix(X)} & X \\
				\terminal & {X/\Fix(X).}
				\arrow["{\eta_X}", from=1-2, to=2-2]
				\arrow[from=2-1, to=2-2]
				\arrow[from=1-1, to=2-1]
				\arrow["{t_X}", from=1-1, to=1-2]
				\arrow["\lrcorner"{anchor=center, pos=0, rotate=180}, draw=none, from=2-2, to=1-1]
			\end{tikzcd}\]
			To show this, we consider an arrow $g \colon Y \rightarrow X$ such that $\eta_Xg(y)=[x]$, for an arbitrary $x \in \Fix(X)$. Then $g(y) \in \Fix(X)$, and so $g$ restricts to $\Fix(X)$; hence we get $t_X = \Zker(\eta_X)$ and $\eta_X= \Zcoker(t_X)$. If $\Fix(X)= \emptyset$, we have the following sequence 
			\[\begin{tikzcd}
				\emptyset & X & X,
				\arrow["{id_X}", from=1-2, to=1-3]
				\arrow["{\ibang{X}}", from=1-1, to=1-2]
			\end{tikzcd}\]
			where $id_X= \Zcoker(\ibang{X})$ (thanks to the description of $\zeros$-cokernels), and $\ibang{X}= \Zker(id_X)$ (since every arrow $g \in \zideal$ with codomain $X$ must have $\emptyset$ as domain, otherwise $X$ would have at least one fix point).
		\end{proof}
		
		Since the category $\mset^{op}$ satisfies the assumptions introduced at the beginning of this section, we can conclude that $\torsione$ satisfies Condition (M).
		
		\begin{prop}
			The torsion theory $\torsione$ in $\mset^{op}$ satisfies Conditions (S).
		\end{prop}
		
		\begin{proof}
			Again we work, dually, in $\mset$. We recall that both limits and colimits in $\mset$ are computed as in $\mathbf{Set}$. Moreover, in general, for every pair $A,B$ of objects of $\mset$ one has $\Fix(A+B)=\Fix(A)+\Fix(B)$. Therefore, in order to show that Condition (S) holds, the only non-trivial case we need to study is the one involving the pushout
			\[\begin{tikzcd}
				\emptyset & \terminal \\
				T & {T+\terminal,}
				\arrow[from=1-1, to=1-2]
				\arrow[from=1-1, to=2-1]
				\arrow[from=1-2, to=2-2]
				\arrow[from=2-1, to=2-2]
				\arrow["\lrcorner"{anchor=center, pos=0, rotate=180}, draw=none, from=2-2, to=1-1]
			\end{tikzcd}\]
			where $\Fix(T)=\emptyset$. The validity of (S) follows from $\Fix(T + \terminal) = \Fix(T) + \terminal = \terminal$.
		\end{proof}
		
		\begin{prop}
			The torsion theory $\torsione$ in $\mset^{op}$ satisfies Condition (N).
		\end{prop}
		
		\begin{proof}
			We work, once more, in $\mset$. Given an arrow $f \colon A \rightarrow B$ we consider
			\[\begin{tikzcd}
				A & B & {Q[f]=B/f(A)} & {\T(Q[f])=(B/f(A))/\Fix(B/f(A))}
				\arrow["f", from=1-1, to=1-2]
				\arrow["q", from=1-2, to=1-3]
				\arrow["{\eta_{Q[f]}}", from=1-3, to=1-4]
			\end{tikzcd}\]
			and we show that $\eta_{Q[f]}q$ is a $\zeros$-cokernel. If $A= \emptyset$, the assertion is trivial. If $A \neq \emptyset$, we have that $$(B/f(A))/\Fix(B/f(A)) \cong B/(f(A) \cup \Fix(B))$$ where the isomorphism is given by $\overline{\varphi} \colon (B/f(A))/\Fix(B/f(A)) \rightarrow B/(f(A) \cup \Fix(B))$ induced by $\varphi \colon B /f(A) \rightarrow B/(f(A) \cup \Fix(B))$, where $\varphi([b]) \coloneqq (b)$ ($[b]$ denotes the class of $b$ with respect to the quotient on $f(A)$, while $(b)$ denotes the class of $b$ with respect to the quotient on $f(A) \cup  \Fix(B)$). Clearly $\varphi$ is well defined and surjective. We observe that $\varphi([b])=\varphi([c])$ implies $b,c \in f(A) \cup \Fix(B)$: if $b,c \in f(A)$, then $[b]=[c]$; if $b \in f(A)$ and $c \in \Fix(B)$, then $[b] \in \Fix(B/f(A))$ (since every element coming from $A$ via $f$ is a fix point of $B/f(A)$) and $[c] \in \Fix(B/f(A))$; if $b,c \in \Fix(B)$, then $[b], [c] \in \Fix(B/f(A))$. Therefore $\overline{\varphi}$ is also injective and $\T(Q[f])$ is the $\zeros$-cokernel of $f(A) \cup \Fix(B) \hookrightarrow B$.
		\end{proof}
		
		We are ready to describe the stable factorization system induced by $\torsione$ on $\mset^{op}$. In order to simplify the argument, we will work again in $\mset$. We get that
		\[ \facte= \{ e \colon A \rightarrow B \in \Arr(\mset) \, | \, e \text{ is a $\zeros$-kernel, } B/e(A) \in \tort \} \]
		and
		\[ \factm = \{ m \colon A \rightarrow B \in \Arr(\mset) \, | \, B/m(A) \in \torf \}. \]
		Hence, if we consider an arrow $f \colon A \rightarrow B$, the $(\factm, \facte)$-factorization of $f$ is given by
		\[\begin{tikzcd}
			A && B \\
			& {f(A) \cup \Fix(B),}
			\arrow["f", from=1-1, to=1-3]
			\arrow["m"', from=1-1, to=2-2]
			\arrow[tail, "e"', from=2-2, to=1-3]
		\end{tikzcd}\]
		where the morphism $m\in \factm$ is obtained by restricting the codomain of $f$, while $e\in \facte$ is the inclusion map. In order to see that $e$ is a $\zeros$-cokernel, consider the pullback (assume $A \neq \terminal$, otherwise the statement is trivial)
		\[\begin{tikzcd}
			{f(A) \cup \Fix(B)} & \terminal \\
			B & {B/(f(A) \cup \Fix(B)),}
			\arrow[from=1-1, to=1-2]
			\arrow["r", from=1-2, to=2-2]
			\arrow["e"', from=1-1, to=2-1]
			\arrow[""{name=0, anchor=center, inner sep=0}, "q"', from=2-1, to=2-2]
			\arrow["\lrcorner"{anchor=center, pos=0}, draw=none, from=1-1, to=0]
		\end{tikzcd}\]
		where $r$ is defined and unique since $|\Fix(B/(f(A) \cup \Fix(B)))|=1$.
		
		\begin{prop}
			The reflector $\F$ induced by the torsion theory $\torsione$ in $\mset^{op}$ is protoadditive.
		\end{prop}
		
		\begin{proof}
			Again, we work in $\mset$. Clearly $\F(\emptyset)=\emptyset$ and $\F(\terminal)=\terminal$. We prove that $\F$ preserves the pushouts of split monomorphisms. Consider the following pushout:
			\[\begin{tikzcd}
				X & Y \\
				Z & P,
				\arrow["{i_Z}"', from=2-1, to=2-2]
				\arrow["{i_Y}", from=1-2, to=2-2]
				\arrow["s", shift left=1, from=1-1, to=1-2]
				\arrow["p", shift left=1, from=1-2, to=1-1]
				\arrow["f"', from=1-1, to=2-1]
				\arrow["\lrcorner"{anchor=center, pos=0, rotate=180}, shift right=1, draw=none, from=2-2, to=1-1]
			\end{tikzcd}\]
			where $ps=id_X$. We define on $Y$ the following relation: $y \sim w$ if and only if $y=w$ or there exist $x,t \in X$ such that $y=s(x)$, $w=s(t)$, and $f(x)=f(t)$. Let us prove that $\sim$ is an equivalence relation preserved by the action of $M$. It is clear that $\sim$ is both reflexive and symmetric. We deal with the transitivity: suppose $y \sim w$ and $w \sim r$; clearly if $y=w$ or $w=r$ the property holds, so suppose $y=s(x)$, $w=s(t)$ with $f(x)=f(t)$ and $w=s(\overline{t})$, $r=s(q)$ with $f(\overline{t})=f(q)$; since $s$ is injective we deduce $t= \overline{t}$, and so $y \sim r$. Finally, if $y \sim w$ and $y \neq w$ then $y=s(x)$, $w=s(t)$ with $f(x)=f(t)$ and, therefore, $my=s(mx)$, $mw=s(mt)$ with $f(mx)=f(mt)$, i.e.\ $my \sim mw$. We prove that, in the category $\mathbf{Set}$, the following square is a pushout:
			\[\begin{tikzcd}
				X & Y \\
				Z & {(Z \setminus f(X))+(Y/\sim),}
				\arrow["s", from=1-1, to=1-2]
				\arrow["f"', from=1-1, to=2-1]
				\arrow["{i_2}", from=1-2, to=2-2]
				\arrow["{i_1}"', from=2-1, to=2-2]
			\end{tikzcd}\]
			where $i_1(z)=z$ if $z \in Z \setminus f(X)$, $i_1(z)=[s(x)]$ if $z=f(x)$ (if $f(x_1)=f(x_2)$ then, by definition, $s(x_1) \sim s(x_2)$), and $i_2(y)=[y]$. The square is commutative: $i_1f(x)=[s(x)]=i_2(s(x))$. Now, let us consider a pair of arrows $a \colon Z \rightarrow A$ and $b \colon Y \rightarrow A$ such that $af=bs$; we define $\varphi \colon (Z \setminus f(X))+(Y/\sim) \rightarrow A$ by putting
			$$\varphi(e)=
			\begin{cases}
				a(z) & \text{if } e=z \in Z \setminus f(X) \\
				b(y) & \text{if } e=[y] \in Y/\sim.
			\end{cases}$$
			$\varphi$ is well defined: if $[y]=[w]$ and $y \neq w$, then $y=s(x)$, $w=s(t)$ and $f(x)=f(t)$, hence $b(y)=bs(x)=af(x)=af(t)=b(w)$. The uniqueness is guaranteed because $i_1,i_2$ are jointly epimorphic ($i_1(Z) \supseteq Z \setminus f(X)$ and $i_2(Y)=Y/\sim$). With the action of $M$ induced by $i_1$ and $i_2$ we deduce that the square above is a pushout also in $\mset$. Now, we observe that $\Fix((Z \setminus f(X))+(Y/\sim))=\Fix(Z \setminus f(X))+ \Fix(Y/\sim)$. So, we study the set $\Fix(Y/ \sim)$. We consider an element $[y] \in \Fix(Y/\sim)$ such that $y \notin \Fix(Y)$; then, for every $m \in M$ such that $my \neq y$, since $my \sim y$, there exists an element $x \in X$ such that $y=s(x)$ (and, since $s$ is injective, this element $x$ does not depend on $m$) and an element $x' \in X$ such that $my=s(x')$ with $f(x)=f(x')$. But $s$ is injective, hence, from $s(x')=my=ms(x)=s(mx)$, we deduce $x'=mx$ and then $f(x)=mf(x)$. Moreover, for every element $m \in M$ such that $my=y$, we get $ms(x)=s(x)$ and, applying $p$, we obtain $mx=x$ and so $mf(x)=x$. In other words, if $y \notin \Fix(Y)$ there exists a unique $x \in X$ such that $y=s(x)$ and $f(x) \in \Fix(Z)$. We are ready to prove that the commutative diagram below is a pushout (i.e.\ $\F$ is protoadditive)
			\[\begin{tikzcd}
				{\Fix(X)} & {\Fix(Y)} \\
				{\Fix(Z)} & {\Fix(Z \setminus f(X))+ \Fix(Y/\sim),}
				\arrow["{i_1}"', from=2-1, to=2-2]
				\arrow["{i_2}", from=1-2, to=2-2]
				\arrow["s", from=1-1, to=1-2]
				\arrow["f"', from=1-1, to=2-1]
			\end{tikzcd}\]
			where, with a slight abuse of notation, we have used the same symbols to denote an arrow and its image through $\F$. We consider a pair of arrows in $\torf$ (i.e.\ maps) $a \colon \Fix(Z) \rightarrow A$ and $b \colon \Fix(Y) \rightarrow A$ such that $af=bs$, and we define \linebreak $\varphi \colon \Fix(Z \setminus f(X))+ \Fix(Y/\sim) \rightarrow A$ by putting
			$$\varphi(e)=
			\begin{cases}
				a(z) & \text{if } e=z \in \Fix(Z \setminus f(X)) \\
				b(y) & \text{if } e=[y], \text{with } y\in \Fix(Y) \\
				af(x) & \text{if } e=[y], \text{with } y \notin \Fix(Y),
			\end{cases}$$
			where the third part makes sense because $y=s(x)$ and $f(x) \in \Fix(Z)$. In order to show that $\varphi$ is well defined we consider $[y], [w] \in \Fix(Y /\sim)$ such that $[y]=[w]$ and we distinguish three cases:
			\begin{itemize}
				\item $y,w \in \Fix(Y)$ (with $y \neq w$, otherwise the assertion easily holds): then there exist $x,t \in X$ such that $y=s(x)$, $w=s(t)$, and $f(x)=f(t)$; hence, applying $p$, we deduce $x,t \in \Fix(X)$, and then $\varphi([y])=b(y)=bs(x)=af(x)=af(t)=bs(t)=b(w)=\varphi([w])$.
				\item $y \in \Fix(Y)$ and $w \notin \Fix(Y)$: then there exist $x,t \in X$ such that $y=s(x)$ (and so, since $y \in \Fix(Y)$, $x \in \Fix(X)$), $w=s(t)$, $f(t) \in \Fix(Z)$, and $f(x)=f(t)$. Since $x \in \Fix(X)$ we can use the commutativity of the square to get $bs(x)=af(x)$, and so $\varphi([y])=bs(x)=af(x)=af(t)=\varphi([w])$.
				\item $y,w \notin \Fix(Y)$: then there exist $x,t \in X$ such that $y=s(x)$, $w=s(t)$, $f(x),f(t) \in \Fix(Z)$, and $f(x)=f(t)$; therefore $\varphi([y])=af(x)=af(t)=\varphi([w])$.
			\end{itemize}
			Finally, $i_1,i_2$ are jointly epimorphic. To show it we deal with the only non-trivial case: we consider an element $[y] \in \Fix(Y/\sim)$ such that $y \notin \Fix(Y)$; therefore there exists $x \in X$ such that $y=s(x)$ and $f(x) \in \Fix(Z)$, and then $i_1(f(x))=[s(x)]=[y]$. Hence $\varphi$ is unique and $\F$ is protoadditive.
		\end{proof}
		
		We observe that, in general, the functor $\F$ does not preserve colimits. For example, if we consider an object $A$ of $\mset$ such that $A \neq \emptyset$ and $\Fix(A)= \emptyset$, then $\tbang{A} \colon A \rightarrow \terminal$ is an epimorphism but $\F(\tbang{A}) \colon \emptyset \rightarrow \terminal$ is not.
		
		We can now describe the central extensions for the Galois structure $\Gamma_{\torf}$ induced by $\torsione$ in $\mset^{op}$. We apply Theorem \ref{centrali e torsione} and we get that a regular epimorphism (i.e.\ an effective descent morphism in a Barr-exact context) of $\mset^{op}$ is a central extension if and only if $K[f]$ is an object of $\torf$. If we dually translate what has just been said we get that a regular monomorphism $f$ in $\mset$, seen as an arrow of $\mset^{op}$, is a central extension for $\Gamma_{\torf}$ if and only if $Q[f] \in \torf$ (where $Q[f]$ is the codomain of the $\zeros$-cokernel of $f$ in $\mset$). Finally, recalling that limits in $\mset$ are computed as in $\mathbf{Set}$, we get that a regular epimorphism $f$ of $\mset^{op}$ is a central extension if and only if $f \colon A \rightarrow B$, considered as an arrow of $\mset$, is a monomorphism and $B/f(A) \in \torf$ (or, equivalently, if $B=\Fix(B) \cup f(A)$).
		
		\subsection{Double negation in Heyting algebras}
		We recall that a \emph{Heyting algebra} is an algebraic structure $(H, \lor, \land, 1, 0, \Rightarrow)$ such that $(H, \lor, \land, 1, 0)$ is a bounded lattice and the binary operation $\Rightarrow$ satisfies
		\[ x \land y \leq z \text{ if and only if } x \leq y \Rightarrow z. \]
		Given a Heyting algebra $H$, let $\dn{H}$ denote the set of regular elements of $H$. An element $x \in H$ is said to be \emph{regular} if $\lnot \lnot x= x$ (recalling that $\lnot x \coloneqq x \Rightarrow 0$). It is a known fact that $(\dn{H}, \dn{\lor}, \land, 0, 1, \Rightarrow)$ is a Boolean algebra, where $$x \dn{\lor} y \coloneqq \lnot (\lnot x \land \lnot y).$$
		Therefore, this construction defines a functor
		\[\begin{tikzcd}[row sep=5pt]
			\heyting & \boole \\
			H & {\dn{H}} \\
			\\
			\\
			L & {\dn{L},}
			\arrow["f"', from=2-1, to=5-1]
			\arrow[maps to, from=2-1, to=2-2]
			\arrow[maps to, from=5-1, to=5-2]
			\arrow["{\F(f)}", from=2-2, to=5-2]
			\arrow["\F", from=1-1, to=1-2]
		\end{tikzcd}\]
		where $\F(f)$ is simply the restriction of $f$ to $\dn{H}$.
		Since it is true that $\lnot \lnot (\lnot \lnot x) = \lnot \lnot x$, we can define the map
		\[ \myfunc{\lnot \lnot}{H}{\dn{H}}{x}{\lnot \lnot x.} \]
		It is known that this map is a surjective morphism of Heyting algebras. Specifically, this morphism is the $H$-component of the unit of the adjunction $\F \dashv \i$, where $\i \colon \boole \rightarrow \heyting$ is the inclusion functor.
		
		A Heyting algebra $H$ is \emph{pseudo-deterministic} (we thank Mariano Messora for the suggestion of the name) if, for every $x \in H$, either $\lnot x= 1 \text{ or } \lnot x=0.$ Given a Heyting algebra $H$, we define $\T(H) \coloneqq \{ x \in H \,|\, \lnot x=0 \text{ or } \lnot x =1 \}$. We recall that, in every Heyting algebra, the equations $x \Rightarrow 1=1$, $\lnot (x \land y)=x \Rightarrow \lnot y$, $\lnot (x \lor y)= \lnot x \land \lnot y$, and $\lnot (x \Rightarrow y)= \lnot \lnot x \land \lnot y$ hold (a proof of these identities can be found in any book about Heyting algebras and intuitionistic logic). These identities can be used to prove that $\T(H)$ is a Heyting algebra whose operations are induced by those of $H$. In particular, we get
		$$\displaylines{\lnot (x \lor y)=
			\begin{cases}
				1 & \text{if } \lnot x = 1, \lnot y=1 \\
				0 & \text{otherwise}
			\end{cases}
			\qquad \qquad
			\lnot (x \land y)=
			\begin{cases}
				0 & \text{if } \lnot x = 0, \lnot y =0 \\
				1 & \text{otherwise}
			\end{cases}
			\cr
			\lnot (x \Rightarrow y)=
			\begin{cases}
				1 & \text{if } \lnot x = 0, \lnot y =1 \\
				0 & \text{otherwise.}
		\end{cases}}$$
		So, $\T(H)$ is a pseudo-deterministic Heyting algebra. We denote by $\pd$ the full subcategory of $\heyting$ whose objects are the pseudo-deterministic Heyting algebras. Moreover, the assignment described above establishes a functor
		\[\begin{tikzcd}[row sep=5pt]
			\heyting & \pd \\
			H & {\T(H)} \\
			\\
			\\
			L & {\T(L),}
			\arrow["f"', from=2-1, to=5-1]
			\arrow[maps to, from=2-1, to=2-2]
			\arrow[maps to, from=5-1, to=5-2]
			\arrow["{\T(f)}", from=2-2, to=5-2]
			\arrow["\T", from=1-1, to=1-2]
		\end{tikzcd}\]
		where $\T(f)$ is given by the restriction of $f$ to $\T(H)$. It is easy to observe that the inclusion of $\T(H)$ in $H$ is the $H$-component of the counit of the adjunction $\j \dashv \T$, where $\j \colon \pd \rightarrow \heyting$ is the inclusion functor.
		
		Since, in every Boolean algebra, $\lnot x =0$ implies $x=1$ and $\lnot x= 1$ implies $x=0$, we deduce that $$\boole \cap \pd=\{ \terminal, \initial \}.$$ From this point on, we will consider as class of zero objects $\zeros \coloneqq \boole \cap \pd=\{ \terminal, \initial \}$.
		
		\begin{prop}
			$(\pd,\boole)$ is a torsion theory in $\heyting$. Moreover, this torsion theory satisfies Conditions (S), (M) and (N).
		\end{prop}
		
		\begin{proof}
			We start by showing that, for any Heyting algebra $H$, $\F(H)=\terminal$ holds if and only if $H=\terminal$. It is clear that $H=\terminal$ implies $\F(H)=\terminal$. Conversely, if $\F(H)=\terminal$, this implies that $\lnot\lnot 1=\lnot\lnot 0$ in $H$, and as a result, $0=1$. Hence, once we prove that $(\pd, \boole)$ constitutes a torsion theory, we immediately conclude that it satisfies Conditions (M), since $\zeros$ is a terminal class, and (S), using the same argument as at the beginning of the proof of Proposition \ref{ideali in MV}. Let us consider a morphism of Heyting algebras $f \colon T \rightarrow F$, where $T$ is an object of $\pd$ and $F$ is an object of $\boole$. Let $x \in T$ be a fixed element. If $\lnot x = 0$, then $0 = f(\lnot x) = \lnot f(x)$. Since $F$ is a Boolean algebra, we deduce that $f(x) = 1$. Similarly, if $\lnot x = 1$, we get $f(x) = 0$. Thus, $f$ factors through an object of $\zeros$. We now fix a Heyting algebra $H \neq \terminal$ and show that the inner commutative square below is both a pullback and a pushout:
			\[\begin{tikzcd}
				K \\
				& {\T(H)} & \initial \\
				& H & {\F(H)} \\
				&&& M,
				\arrow["\chi", from=2-2, to=2-3]
				\arrow["\iota", from=2-3, to=3-3]
				\arrow["{\eta_H}"', from=3-2, to=3-3]
				\arrow["{t_H}"', from=2-2, to=3-2]
				\arrow["h"', curve={height=12pt}, from=1-1, to=3-2]
				\arrow["l", curve={height=-12pt}, from=1-1, to=2-3]
				\arrow["g", curve={height=-12pt}, from=2-3, to=4-4]
				\arrow["f"', curve={height=12pt}, from=3-2, to=4-4]
			\end{tikzcd}\]
			where $t_H$ is the inclusion, $\eta_H$ denotes the double negation map, and $\chi$ is defined as follows:
			$$\chi(x)=
			\begin{cases}
				0 & \text{if } \lnot x = 1 \\
				1 & \text{if } \lnot x = 0.
			\end{cases}$$
			Simple calculations allow us to verify that $\chi$ is a morphism of Heyting algebras. Let us consider a pair of morphisms of Heyting algebras, $h \colon K \rightarrow H$ and $l \colon K \rightarrow \initial$, satisfying the condition $\eta_H h=\iota l$. By observing that $\lnot \lnot h(k)= \iota l (k)$, for every $k \in K$, and recalling that in any Heyting algebra $\lnot \lnot \lnot x= \lnot x$, we can infer that $\lnot h(k)$ can only be equal to $0$ or $1$. Thus, $h$ factors through $\T(H)$, and the uniqueness of this factorization is ensured since $t_H$ is a monomorphism. Therefore, the given square is a pullback. We now fix a pair of morphisms of Heyting algebras, $f \colon H \rightarrow M$ and $g \colon \initial \rightarrow M$, such that $ft_H = g\chi$. We observe that the kernel pair of $\eta_H$ is given by $\Eq(\eta_H)=\{(x,y)\in H \, | \, \lnot\lnot x = \lnot\lnot y\}$. Moreover, since $\lnot\lnot (x \Rightarrow y) = \lnot\lnot x \Rightarrow \lnot\lnot y$, we deduce that, if $(x,y) \in \Eq(\eta_H)$, then $\lnot\lnot (x \Rightarrow y)=1$ and $\lnot\lnot (y \Rightarrow x)=1$. Hence $\lnot(x \Rightarrow y)=0$ and $\lnot(y \Rightarrow x)=0$, so we see that $x \Rightarrow y$ and $y \Rightarrow x$ are elements of $\T(H)$. Clearly, we have $\chi(x \Rightarrow y)=1$ and $\chi(y \Rightarrow x)=1$. By commutativity, we deduce that $f(x \Rightarrow y)=1$ (hence $f(x) \leq f(y)$) and $f(y \Rightarrow x)=1$ (hence $f(y) \leq f(x)$).
			Finally, since $\eta_H$ is surjective, and therefore it is the coequalizer of its kernel pair, there exists a unique morphism of Heyting algebras $\overline{f} \colon \F(H) \rightarrow M$, induced by the universal property of the coequalizer, such that $\overline{f}\eta_H=f$. It is also trivially true that $\overline{f}\chi=g$. Therefore, the square is a pushout.
			So, given a Heyting algebra $H$, the associated exact sequence, with a torsion object on the left and a torsion-free object on the right, is the following:
			\[\begin{tikzcd}
				{\T(H)} & H & {\F(H).}
				\arrow["{t_H}", from=1-1, to=1-2]
				\arrow["{\eta_H}", from=1-2, to=1-3]
			\end{tikzcd}\]
			To conclude, we show that this torsion theory satisfies Condition (N). Let us consider the diagram
			\[\begin{tikzcd}
				{\T(K[f])} & {K[f]} & H & L,
				\arrow["t", from=1-1, to=1-2]
				\arrow["k", from=1-2, to=1-3]
				\arrow["f", from=1-3, to=1-4]
			\end{tikzcd}\]
			where $k = \Zker(f)$ and $t=t_{K[f]}$. If $L = \terminal$, the statement is trivial; hence, let us assume $L \neq \terminal$. In order to prove that Condition (N) holds, we show that the following square is a pullback:
			\[\begin{tikzcd}
				Y \\
				& {\T(K[f])} & \initial \\
				& H & {H/(\Eq(\eta_H) \cap \Eq(f)),}
				\arrow["q"', two heads, from=3-2, to=3-3]
				\arrow["\chi", from=2-2, to=2-3]
				\arrow["\iota", from=2-3, to=3-3]
				\arrow["{k t}"', from=2-2, to=3-2]
				\arrow["a"', curve={height=12pt}, from=1-1, to=3-2]
				\arrow["b", curve={height=-12pt}, from=1-1, to=2-3]
			\end{tikzcd}\]
			where $\Eq(\eta_H) \cap \Eq(f)= \{ (x,y) \in H \times H \, | \, \lnot \lnot x= \lnot \lnot y \text{ and } f(x)=f(y) \}$ and $q$ is the quotient projection. Let us consider a pair of arrows $a \colon Y \rightarrow H$ and $b \colon Y \rightarrow \initial$ such that $qa=\iota b$. It can be observed that, for every $y \in Y$, if $b(y)=0$, then, by commutativity, $qa(y)=0$, which implies that $\lnot \lnot a(y)=0$ and $fa(y)=0$. Recalling that $K[f]=\{ x \in H \,|\, f(x)=0 \text{ or } f(x)=1 \}$, we deduce that $a(y) \in \T(K[f])$. By a similar reasoning, if $b(y)=1$, then $a(y) \in \T(K[f])$ as well. Thus, we can conclude that $a$ factors through $\T(K[f])$, and the uniqueness of this factorization is guaranteed by the fact that $k t$ is a composition of monomorphisms, and therefore a monomorphism itself.
		\end{proof}
		
		\begin{prop}
			The reflector $\F$ is a localization (i.e.\ it preserves all finite limits). Hence, in particular, $\F$ is protoadditive.
		\end{prop}
		
		\begin{proof}
			Clearly, $\F$ preserves the terminal object $\terminal$. So, it suffices to prove that $\F$ preserves every pullback. Consider the following pullback in $\heyting$:
			\[\begin{tikzcd}
				{H \times_LK} & K \\
				H & L.
				\arrow["f"', from=2-1, to=2-2]
				\arrow["g", from=1-2, to=2-2]
				\arrow["{\pi_H}"', from=1-1, to=2-1]
				\arrow["{\pi_K}", from=1-1, to=1-2]
				\arrow["\lrcorner"{anchor=center, pos=0}, draw=none, from=1-1, to=2-2]
			\end{tikzcd}\]
			Our goal is to show that $\dn{(H \times_L K)}=\dn{H} \times_{\dn{L}} \dn{K}$.
			Let $(h,k) \in \dn{(H \times_L K)}$. Then we have $(h,k)=\lnot \lnot (h,k)$ and $f(h)=g(k)$. This implies that $\lnot \lnot h= h$ (i.e.\ $h \in \dn{H}$) and $\lnot \lnot k \in K$ (i.e.\ $k \in \dn{K}$). Therefore, $(h,k) \in \dn{H} \times_{\dn{L}} \dn{K}$.
			Conversely, let $(h,k) \in \dn{H} \times_{\dn{L}} \dn{K}$. Then we have $\lnot \lnot h= h$, $\lnot \lnot k=k$, and $f(h)=g(k)$. It follows that $(h,k)=\lnot \lnot (h,k)$ and hence $(h,k) \in \dn{(H \times_L K)}$.
			Therefore $\dn{(H \times_L K)}=\dn{H} \times_{\dn{L}} \dn{K}$, as required.
		\end{proof}
		
		We are now ready to investigate the central extensions determined by the Galois structure induced by the adjunction $\F \dashv i$. As we have shown, an effective descent morphism $f \colon H \rightarrow L$ in $\heyting$ (i.e.\ a surjective map) is a central extension if and only if $K[f]$, which is defined as the set $\{ x \in H \,|\, f(x)=0 \text{ or } f(x)=1 \}$, is a Boolean algebra. In simpler terms, a surjective map is central if and only if its $\zeros$-kernel is an object of $\boole$. Furthermore, since the reflector $\F$ is a localization, it follows that an extension is central if and only if it is trivial. This last fact can also be seen as a consequence of Remark 4.6 in \cite{JKgalois}.
		
		We now turn our attention to the stable factorization system induced by the torsion theory $(\pd, \boole)$. Thanks to the results obtained in the previous sections, we get $\facte = \{ e \in \Arr(\heyting) \,|\, e$ \text{ is a $\zeros$-cokernel and } $K[e] \in \pd \}$ and $\factm = \{ m \in \Arr(\heyting) \,|\, K[m] \in \boole \}$. Given an arrow $f \colon H \rightarrow L$ in $\heyting$, we can construct the following factorization:
		\[\begin{tikzcd}
			H && L \\
			& {\overline{H},}
			\arrow["f", from=1-1, to=1-3]
			\arrow["e"', two heads, from=1-1, to=2-2]
			\arrow["{\overline{f}}"', from=2-2, to=1-3]
		\end{tikzcd}\]
		where $\overline{H} \coloneqq H/(\Eq(f) \cap \Eq(\eta_H))$, $e$ is the quotient projection, and $\overline{f}([x]) \coloneqq f(x)$ for every $[x] \in \overline{H}$. It is not difficult to see that $e$ is the $\zeros$-cokernel of the inclusion of $K[e] = \{ x \in H \,|\, f(x)=0, \lnot x=1 \} \cup \{ x \in H \,|\, f(x)=1, \lnot x =0 \}$, and that $K[e] \in \pd$, being a subalgebra of $\T(H)$. Moreover, we observe that $K[\overline{f}] \in \boole$. To see this, consider an element $[x] \in K[\overline{f}]$. We need to show that $[x]=\lnot \lnot [x]$. Suppose $\overline{f}([x])=0$; to conclude that $[x]=\lnot \lnot [x]$, we need to prove that $f(x)=f(\lnot \lnot x)$ and $\lnot \lnot x = \lnot \lnot \lnot \lnot x$. The second equality always holds, while the first is true because $f(x)=0$. The same happens if $\overline{f}([x])=1$. Thus, the proposed factorization is precisely the $(\facte, \factm)$-factorization.
		
		\subsection{Contraction of vertices in $\sset$}
		We denote by $\Delta$ the category whose objects are the totally ordered sets $[n] \coloneqq \{0,1,2,\dots,n\}$, where the order is induced by the usual one of $\mathbb{N}$, and whose morphisms are the order-preserving maps. We recall that a \emph{simplicial set} is a functor $X \colon \Delta^{op} \rightarrow \set$. The category of simplicial sets, denoted by $\sset$, has the simplicial sets as objects, and the natural transformations between them as morphisms. For every natural number $n$ and every simplicial set $X$, we define $X_n \coloneqq X([n])$.
		
		A simplicial set $X$ can be seen as family of sets $X_n$, for each non-negative integer $n$, with two sets of maps $d_i \colon X_n \rightarrow X_{n-1}$, for $ n>0$, and $s_i \colon X_n \rightarrow X_{n+1}$, for $0 \leq i \leq n$, such that the following conditions are satisfied for each $i, j$:
		
		\begin{equation*}
			d_id_j = d_{j-1}d_i, \hspace{1cm} i<j
		\end{equation*}
		
		\begin{equation*}
			s_is_j =s_{j} s_{i-1}, \hspace{1cm} i>j
		\end{equation*}
		
		\begin{equation*}
			d_is_j =
			\begin{cases}
				s_{j-1}d_i, & i<j \\
				id, & i=j, j+1 \\
				s_j d_{i-1}, & i>j+1.
			\end{cases}
		\end{equation*}

		This is the standard way to write the data of a simplicial set. The elements of $X_0$ are called the \emph{vertices}. Given that each $s_i$ is an injective map, to simplify the notation, we will assume that these maps are inclusion maps.
		
		Let $S$ denote the terminal object of $\sset$, and $V$ denote the initial object. It is clear that $S=\Hom_{\Delta}(\mathrm{-},[0])$, which implies that $S_n=\{* \}$, for every natural number $n$. Furthermore, $V_n=\emptyset$ for every natural number $n$. We observe that $\sset$ is a two-valued elementary topos. Indeed, if we consider a simplicial set $X \subseteq S$, then either $X_n=\{*\}$ for all $n\in\mathbb{N}$, or there exists a natural number $n$ such that $X_n=\emptyset$. However, in the latter case, it follows that $X_n=\emptyset$ for all $n\in\mathbb{N}$, and hence $X=V$. Therefore, we have $\Sub(S)=\{S,V\}$.
		
		Moreover, the Yoneda Lemma implies $$\sset(S,X)=\Hom_{\sset}(\Hom_{\Delta}(\mathrm{-},[0]), X)=X_0$$ for every simplicial set $X$. Therefore, the simplicial sets for which there exists a unique arrow $S\rightarrow X$ are precisely those such that $|X_0|=1$. We now define two full subcategories of $\sset$ whose objects are:
		$$\torf \coloneqq \{X\in\sset \,|\, X_n=X_0, X(f)=id_{X_0} \text{ for every } n \in \mathbb{N}, f \in \Arr(\Delta) \}$$
		and
		$$\tort \coloneqq \{X\in\sset \,|\, |X_0| \leq 1 \}.$$
		Given a simplicial set $X$, we define $\F(X)$ as the subobject of $X$ such that $\F(X)_n = X_0$ for every $n \in \mathbb{N}$. So, we have $\F(X) \in \torf$.
		
		From this point on, we will consider as class of zero objects $\zeros \coloneqq \torf \cap \tort = \{ S, V \}$.
		
		\begin{prop}
			$\torsione$ is a torsion theory in $\sset^{op}$.
		\end{prop}
		
		\begin{proof}
			We work, dually, in $\sset$. Let us consider an arrow $f \colon P \rightarrow T$ in $\sset$, where $P$ is an object of $\torf$ and $T$ is an object of $\tort$. If $T=V$, then clearly $f$ factors through an object of $\zeros$. If $T \neq V$, then $f_0(P_0)=T_0=\{*\}$. Moreover, by naturality, the following diagram commutes:
			\[\begin{tikzcd}
				{P_0} & {T_0=\{*\}} \\
				{P_0} & {T_1.}
				\arrow[hook, from=1-2, to=2-2]
				\arrow["{f_0}", from=1-1, to=1-2]
				\arrow[equal, from=1-1, to=2-1]
				\arrow["{f_1}"', from=2-1, to=2-2]
			\end{tikzcd}\]
			Thus, $f_1(P_0)=\{*\}$. By iterating this observation, we deduce that $f_n(P_0)=\{*\}$ for every natural number $n$. Therefore, $f$ factors through $S$. Furthermore, we define the simplicial set morphism $t_X \colon \F(X) \rightarrow X$, where $(t_X)_n$ is the inclusion of $X_0$ in $X_n$. Let us assume that $X \neq V$ and construct $\T(X)$ using the following pushout diagram:
			\[\begin{tikzcd}
				{\F(X)} & S \\
				X & {\T(X).}
				\arrow[from=1-1, to=1-2]
				\arrow["{t_X}"', from=1-1, to=2-1]
				\arrow["{\eta_X}"', from=2-1, to=2-2]
				\arrow["r", from=1-2, to=2-2]
				\arrow["\lrcorner"{anchor=center, pos=0, rotate=180}, draw=none, from=2-2, to=1-1]
			\end{tikzcd}\]
			Recalling that limits and colimits in $\sset$ are computed, level-wise, as in $\set$, we deduce that $\T(X)_0=X_0/X_0=\{*\}$ and, therefore, $\T(X) \in \tort$. We now show that the above square is also a pullback. To this end, let $\alpha \colon Y \rightarrow X$ be a morphism of simplicial sets such that $\eta_X \alpha= r \tbang{Y}$:
			\[\begin{tikzcd}
				Y \\
				& {\F(X)} & S \\
				& X & {\T(X).}
				\arrow[from=2-2, to=2-3]
				\arrow["{t_X}"', from=2-2, to=3-2]
				\arrow["{\eta_X}"', from=3-2, to=3-3]
				\arrow["r", from=2-3, to=3-3]
				\arrow["\alpha"', curve={height=12pt}, from=1-1, to=3-2]
				\arrow["{\tbang{Y}}", curve={height=-12pt}, from=1-1, to=2-3]
			\end{tikzcd}\]
			Since, for every $n\in\mathbb{N}$ and for every $y\in Y_n$, we have $(\eta_X)_n(\alpha_n(y))=r_n(*)=*$, it follows that $\alpha_n(y)\in X_0\subseteq X_n$. Hence, $\alpha$ restricts to an arrow to $\F(X)$ and the square is indeed a pullback. Therefore, we have shown that $t_X$ is the $\zeros$-kernel of $\eta_X$, and that $\eta_X$ is the $\zeros$-cokernel of $t_X$. Finally, if $X=V$, we clearly have $\F(X)=V$, and we define $\T(X)\coloneqq V$.
		\end{proof}
		
		We observe that the torsion theory $\torsione$ in $\sset^{op}$ satisfies Condition (S). Indeed, if $\F(X)=V$ (recall that $V$ is the terminal object in $\sset^{op}$), then $X_0=\emptyset$, and thus $X=V$. So, as in the previous examples, Condition (S) follows. Moreover, $\F$ preserves pullbacks in $\sset^{op}$. To see this, we work dually and we show that $\F$ preserves pushouts in $\sset$. Since pushouts in $\sset$ are computed level-wise as in $\set$, and $\F(X)=X_0$ for every simplicial set $X$, the claim follows easily. In particular, $\F$ is protoadditive, considered as functor on $\sset^{op}$. Observing that $\F$ obviously preserves the terminal object of $\sset^{op}$, we get that $\F$ preserves all finite limits, so it is a localization. Since, moreover, the class of zero objects is terminal, we can conclude that our torsion theory satisfies Condition (M). Finally, we prove that $\torsione$ satisfies Condition (N). Consider a morphism of simplicial sets $f \colon X \to Y$, where $X \neq V$, and define $Q[f]$ through the following pushout:
		\[\begin{tikzcd}
			X & \terminal \\
			Y & {Q[f].}
			\arrow["f"', from=1-1, to=2-1]
			\arrow[from=1-1, to=1-2]
			\arrow["q"', from=2-1, to=2-2]
			\arrow[from=1-2, to=2-2]
			\arrow["\lrcorner"{anchor=center, pos=0, rotate=180}, draw=none, from=2-2, to=1-1]
		\end{tikzcd}\]
		We need to prove that $\eta_{Q[f]}q$ is a $\zeros$-cokernel:
		\[\begin{tikzcd}
			X & Y & {Q[f]} & {\T(Q[f]).}
			\arrow["f", from=1-1, to=1-2]
			\arrow["q", from=1-2, to=1-3]
			\arrow["{\eta_{Q[f]}}", from=1-3, to=1-4]
		\end{tikzcd}\]
		We observe that $\T(Q[f])=(Y_n/f_n(X_n))/(Y_0/f_0(X_0))$, and applying a similar argument to the one seen for the case of $\mset$, we deduce that $Y_n/(Y_0 \cup f_n(X_n)) \cong (Y_n/f_n(X_n))/(Y_0/f_0(X_0))$. Therefore, we conclude that $\eta_{Q[f]}q$ is the $\zeros$-cokernel of the inclusion $\F(Y) \cup f(X) \hookrightarrow Y$ (where $(\F(Y) \cup f(X))_n \coloneqq Y_0 \cup f_n(X_n)$ for every natural number $n$). Moreover, if $X=V$, then $Q[f]=Y$, $q=id_Y$, and $\eta{Q[f]}=\eta_Y$. Thus, in this case $\eta_{Q[f]}q=\eta_Y$, which is a $\zeros$-cokernel by construction.\\
		
		We can now describe the stable factorization system induced by $\torsione$ on $\sset^{op}$. To simplify the argument, we will work in $\sset$. We have
		\[ \facte = \{ e \colon X \rightarrow Y \in \Arr(\sset) \, | \, e \text{ is a $\zeros$-kernel, and } Y/e(X) \in \tort \}, \] and
		\[ \factm = \{ m \colon X \rightarrow Y \in \Arr(\sset) \, | \, Y/m(X) \in \torf \}, \] where $(Y/e(X))_n \coloneqq Y_n/e_n(X_n)$ and $(Y/m(X))_n \coloneqq Y_n/m_n(X_n)$,
		for every natural number $n$. Therefore, as previously observed in the case of $\mset$, the $(\factm, \facte)$-factorization of an arrow $f \colon X \rightarrow Y$ is given by:
		\[\begin{tikzcd}
			X && Y \\
			& {f(X) \cup \F(Y),}
			\arrow["f", from=1-1, to=1-3]
			\arrow[tail, "e"', from=2-2, to=1-3]
			\arrow["m"', from=1-1, to=2-2]
		\end{tikzcd}\]
		where the morphism $m\in \factm$ is obtained by restricting the codomain of $f$, while $e\in \facte$ is, level-wise, the inclusion map.
		
		Finally, we can characterize the central extensions for $\Gamma_{\torf}$, which is the Galois structure induced by $\torsione$ in $\sset^{op}$. The functor $\F$ being a localization, the central and the trivial extensions coincide. Theorem \ref{centrali e torsione} implies that a regular epimorphism in $\sset^{op}$ is a central extension if and only if $K[f]$ is an object of $\torf$. In the dual perspective, we can say that a regular monomorphism $f$ in $\sset$, viewed as an arrow in $\sset^{op}$, is a central extension for $\Gamma_{\torf}$ if and only if $Q[f]$, which is the codomain of the $\zeros$-cokernel of $f$ in $\sset$, belongs to $\torf$. Moreover, since every monomorphism in $\sset$ is regular, we can conclude that a regular epimorphism $f$ in $\sset^{op}$ is a central extension if and only if $f \colon X \rightarrow Y$, considered as an arrow in $\sset$, is a monomorphism and $Y/f(X) \in \torf$, where $Y/f(X)$ is the simplicial set defined by $(Y/f(X))_n \coloneqq Y_n/f_n(X_n)$, for every natural number $n$. Equivalently, if $Y_n=Y_0 \cup f_n(X_n)$ for every natural number $n$.
		
		\subsection{$m$-divisibility in $\coslice$}
		The next example differs from the previous ones because the class of zero objects has more than two elements. In this example, we focus on the study of a torsion theory in the category $\coslice$. Recall that the objects of this category are arrows in the category $\Ab$ of abelian groups of the form $a \colon \Zm \rightarrow A$, where $\Zm$ is the quotient group $\mathbb{Z}/m\mathbb{Z}$ for $m>0$, while a morphism $f \colon a \rightarrow b$ (where $a \colon \Zm \rightarrow A$ and $b \colon \Zm \rightarrow B$) is an arrow $f \colon A \rightarrow B$ in $\Ab$ such that $fa = b$. It is not difficult to observe that the category $\coslice$ can be seen as the category whose objects are pairs $(A,a)$, where $A$ is an abelian group and $a \in A$ is an element such that $ma = 0$, and an arrow $f \colon (A,a) \rightarrow (B,b)$ is a group homomorphism $f \colon A \rightarrow B$ such that $f(a)=b$. To verify this, it suffices to note that the arrow $a \colon \Zm \rightarrow A$ is uniquely determined by the image of $1$ (since $\Zm$ is a cyclic group generated by $1$) and in $\Zm$ we have $m1 = m = 0$. In $\coslice$ the initial object is $(\Zm, 1)$ and the terminal object is $(\graffe{0}, 0)$. Applying the following proposition we get that $\coslice$ is a regular protomodular category.
		
		\begin{prop}[\cite{Bbbook}, Example 3.1.14, Example A.5.13]
			Let $\mathcal{C}$ be a Barr-exact protomodular category. Then, for every $X \in \mathcal{C}$, the coslice category $X/\mathcal{C}$ is Barr-exact and protomodular.
		\end{prop}
		
		Moreover, recall that in the coslice category $X/\mathcal{C}$, quotients and limits are computed as in the category $\mathcal{C}$. Therefore, in $\coslice$, the quotients of the initial object $(\Zm, 1)$ are objects of the form $(\Zh, 1)$, where $h$ divides $m$ (for simplicity of notation, we identify $(\mathbb{Z}_1, 1)$ with $(\graffe{0}, 0)$). As proved in \cite{Cappelletti}, in the regular context the class given by the quotients of the initial object is a terminal class of zero objects (and we denote this class with $\zeros$). Moreover, observe that in $\coslice$, with respect to this class of zero objects, one has $\Z(A,a)=(\spigolose{a},a)$.
		
		\begin{prop}\label{pullback in coslice}
			Let $f \colon (A,a) \rightarrow (B,b)$ be an arrow in $\coslice$. Then, $\Zker(f)=(\spigolose{\invf{0}, a}, a) \hookrightarrow (A,a)$.
		\end{prop}
		
		\begin{proof}
			Let us consider the commutative square:
			\[\begin{tikzcd}
				{(\spigolose{\invf{0},a},a)} & {(\spigolose{b},b)} \\
				{(A,a)} & {(B,b),}
				\arrow["\chi", from=1-1, to=1-2]
				\arrow[hook, from=1-1, to=2-1]
				\arrow[hook, from=1-2, to=2-2]
				\arrow["f"', from=2-1, to=2-2]
			\end{tikzcd}\]
			where $\chi(k+na)=nb$ (where $k \in \invf{0}$). Let us prove that the square is a pullback. Consider an object $(C,c)$ and a pair of arrows $g \colon (C,c) \rightarrow (A,a)$ and $h \colon (\spigolose{b},b) \rightarrow (B,b)$ such that $g f = h$. For every $x \in C$, we have $f g(x) = h(x) = n b = f(n a)$ for some $n \in \mathbb{N}$. Thus, $g(x) - n a \in \invf{0}$, meaning that the image of $g$ is contained in $\spigolose{\invf{0},a}$. Consequently, we can define a morphism $\overline{g} \colon (C,c) \rightarrow (\spigolose{\invf{0},a},a)$, concluding that the square is a pullback.
		\end{proof}
		
		In order to introduce our torsion theory in $\coslice$, we need to recall the following:
		
		\begin{defi}
			Let $m>0$ be a natural number; an abelian group $A$ is \emph{$m$-divisible} if for every $a \in A$ there exists $b \in A$ such that $a=mb$. An abelian group $A$ is \emph{divisible} if it is $m$-divisible for every $m>0$.
		\end{defi}
		
		\begin{example}
			Let $m$ be a positive integer. The following are examples of $m$-divisible abelian groups:
			\begin{itemize}
				\item The additive group $\mathbb Q$ and its quotient $\mathbb Q /\mathbb Z$ are $m$-divisible for every positive integer number $m$.
				\item The subgroup of the additive group $\mathbb Q$ defined by \[ \Zum = \left\{ x \in \mathbb{Q} \, \middle| \, \exists n \in \mathbb{N}, \, x_1, \dots, x_n \in \mathbb{Z}, \, x = \sum_{i=0}^n \frac{x_i}{m^i} \right\}. \]
				\item The subgroup of the multiplicative group $\mathbb{C}^*$ defined by \[S_m \coloneqq \graffe{z \in \mathbb C \,|\, \exists n \in \mathbb{N} \text{ s.t.\ } z^{m^n}=1},\]
				i.e., the group of the $m^n$th roots of unity as $n$ varies over the natural numbers.
			\end{itemize}
		\end{example}
		
		Let $\Divm$ denote the full subcategory of $\Ab$ whose objects are the $m$-divisible groups, and $\j \colon \Divm \rightarrow \Ab$ the inclusion functor. We show that $\j$ is a left adjoint. Given an abelian group $A$, we define
		\[\Dm(A) \coloneqq \graffe{a \in A \,|\,\exists \suc{a\ped{n}} \text{ s.t. } a\ped{0}=a, \forall n \in \mathbb{N} \,\, a\ped{n}=ma\ped{n+1}} \]
		It is easy to check that $\Dm(A)$ is a subgroup of $A$ and that, given a morphism of abelian groups $f \colon A \rightarrow B$ and an element $a \in \Dm(A)$, $f(a) \in \Dm(B)$. So $\Dm$ extends to a functor $\Dm \colon \Ab \rightarrow \Divm$.  For each $A \in \Ab$, the inclusion $\Dm(A) \hookrightarrow A$ is the $A$-component of the counit of the adjunction $j \dashv \Dm$. Indeed, for any $m$-divisible group $D$ and any morphism $f \colon D \rightarrow A$, $f$ factors through $\Dm(A)$ giving the following commutative triangle:
		
		\[\begin{tikzcd}
			& D \\
			{\Dm(A)} && A.
			\arrow["{\overline{f}}"', from=1-2, to=2-1]
			\arrow["f", from=1-2, to=2-3]
			\arrow[hook, from=2-1, to=2-3]
		\end{tikzcd}\]
		
		The reason for introducing this notion of $m$-divisibility is connected to an important property of our torsion theories. Let us consider an object $(A,a)$ in $\coslice$, and suppose that there exists a morphism $f \colon (A,a) \rightarrow (\Zm, 1)$. This implies that $\ord(a) = m$ (otherwise, there would exist a divisor $k$ of $m$ such that $k1=0$ in $\Zm$, which is absurd). We can then define the morphism $s \colon (\Zm, 1) \rightarrow (A,a)$, uniquely determined by $s(1) \coloneqq a$. Therefore, since $f$ is a split epimorphism, it follows that in $\Ab$, $A \cong \invf{0} \oplus \spigolose{a}$.
		
		We now have all the necessary ingredients to determine which objects in the coslice category make $\Homcos((A,a), (\Zm, 1))$ a singleton. This is important for identifying the candidates that could form the torsion part of a torsion theory (recalling that, in general, given a torsion object and a zero object, there can be at most one morphism from the first to the second). We showed that, if there exists a morphism $f \colon (A,a) \rightarrow (\Zm,1)$, then, up to isomorphism, $A = B \oplus \spigolose{a}$ (where $B=\invf{0}$), and $f$ is the projection onto the second component of the direct sum. Recalling that every morphism $g \colon (A,a) \rightarrow (\Zm, 1)$ in $\coslice$ satisfies $g(a)=1$ and that $A = B \oplus \spigolose{a}$, it follows that $g$ is uniquely determined by its restriction (viewed as a morphism of abelian groups) to $B$ (we are simply using the fact that a direct sum is a coproduct). Therefore, in this case, we have $\Homcos((A,a), (\Zm, 1)) \cong \Homab(B, \Zm)$. Thus, we are interested in characterizing those abelian groups $B$ such that $\Homab(B, \Zm)$ has exactly one element. In other words, we want to characterize those $B$'s such that $\Homab(B, \Zm) = \graffe{0}$.
		
		\begin{prop}
			Given an abelian group $A$, $\Homab(A, \Zm) = \graffe{0}$ if and only if $A$ is $h$-divisible for every $h$ which divides $m$.
		\end{prop}
		
		\begin{proof}
			$(\Longleftarrow)$ Let $f \colon A \rightarrow \Zm$ be a morphism of groups. Consider an element $a \in A$; since $A$ is $m$-divisible, there exists an element $b \in A$ such that $a=mb$. Thus, $f(a)=f(mb)=mf(b)=0$ i.e.\ $f(a)=0$ for every $a \in A$.\\
			$(\Longrightarrow)$ We first observe that, if $A$ is $p$-divisible for every prime $p$ dividing $m$, then it is divisible for every divisor $h$ of $m$: for simplicity, suppose $h = pq$, where $p$ and $q$ are primes, and assume that $A$ is both $p$-divisible and $q$-divisible. Given any element $a \in A$, there exists $b \in A$ such that $a = pb$; now, there exists a $c \in A$ such that $b = qc$. Therefore, $a = pb = pqc = hc$, i.e., $A$ is $h$-divisible. Now, suppose that $A$ is not $h$-divisible for some divisor $h$ of $m$. Then, $A$ is not $p$-divisible for some prime $p$ dividing $m$. This is equivalent to stating that $pA$ is a proper subset of $A$ (where $pA$ denotes the collection of elements of the form $pa$, as $a$ varies in $A$). Therefore, the quotient $A/pA$ is nontrivial. Observe that $A/pA$ is a $\mathbb{Z}_p$-vector space (the action of $\mathbb{Z}$ on $A$ induces a well-defined action of $\mathbb{Z}_p$ on $A/pA$). Thus, $A/pA$ can be viewed as a direct sum of copies of $\mathbb{Z}_p$, and hence there exists a surjective morphism $\pi \colon A/pA \rightarrow \mathbb{Z}_p$. Finally, note that, since $p$ divides $m$, there is a morphism $\mathbb{Z}_p \rightarrow \Zm$ (this morphism explicitly maps $1 \in \mathbb{Z}_p$ to $\frac{m}{p} \in \Zm$). Therefore, $\Homab(A, \Zm)$ is nontrivial.
		\end{proof}
		
		In light of what has been discussed, we define the full subcategory $\tort \subseteq \coslice$ whose objects are the $(A,a) \in \coslice$ such that $A=\spigolose{\Dm(A),a}$ (the idea is to consider objects generated by their base point and their $m$-divisible part). Moreover, we introduce the torsion-free part of our torsion theory as the full subcategory $\torf \subseteq \coslice$ whose objects are the $(A,a)$ such that $\Dm(A) = \graffe{0}$.
		
		\begin{prop}
			The pair $\torsione$ determines a torsion theory, with respect to the zero class $\zeros$ given by the quotients of $(\Zm, 1)$, in the category $\coslice$.
		\end{prop}
		
		\begin{proof}
			Let us start by proving that $\tort \cap \torf = \zeros$. Let $(\Zh, 1)$ be an element of $\zeros$. Note that $\Dm(\Zh) = \graffe{0}$ (since each element of $\Zh$ has order which divides $m$), and thus $\spigolose{\graffe{0}, a} = \spigolose{a} = \Zh$. If $(A,a)$ is an element of $\tort \cap \torf$, then $\Dm(A) = \graffe{0}$ and $A = \spigolose{\Dm(A), a} = \spigolose{a} \cong \mathbb{Z}\ped{\ord(a)}$; hence, $(A,a) \cong (\mathbb{Z}\ped{\ord(a)}, 1)$. We prove that for every $(A,a) \in \tort$ and for every $(B,b) \in \torf$, $\Homcos((A,a),(B,b)) \subseteq \zideal$. Consider an arrow $f \colon (A,a) \rightarrow (B,b)$ and recall that $f(\Dm(A)) \subseteq \Dm(B)=\graffe{0}$. Since $A=\spigolose{\Dm(A),a}$, every element $x$ of $A$ can be written as $x=d+na$ (where $d \in \Dm(A)$ and $n \in \mathbb{N}$). So, $f(x)=f(d+na)=f(na)=nb$. Thus, $f$ factors as in the commutative triangle
			\[\begin{tikzcd}
				{(A,a)} && {(B,b)} \\
				& {(\spigolose{b},b) \cong (\mathbb{Z}\ped{\ord(b)}, 1).}
				\arrow["f", from=1-1, to=1-3]
				\arrow[from=1-1, to=2-2]
				\arrow[from=2-2, to=1-3]
			\end{tikzcd}\]
			To proceed, we define $\T(A,a) \coloneqq (\spigolose{\Dm(A),a}, a)$ and $\F(A,a) \coloneqq (A/\Dm(A), [a])$. Let us begin by showing that $\T(A,a) = (\spigolose{\Dm(A),a},a) \in \tort$. We observe that $\Dm(\Dm(A)) = \Dm(A)$; we only need to see that $\Dm(A) \subseteq \Dm(\Dm(A))$ (the other inclusion is obvious). Consider an element $d \in \Dm(A)$; then there exists a sequence $\suc{d_n}$ such that $d_n \in A$ for every $n \in \mathbb{N}$, with $d_0 = d$ and $d_n = md_{n+1}$. Fix a $k \in \mathbb{N}$ and consider the sequence $\suc{c_n}$ where, for each $n \in \mathbb{N}$, $c_n \coloneqq d_{n+k}$. This sequence allows us to conclude that $d_n \in \Dm(A)$, and thus, $d \in \Dm(\Dm(A))$. Hence, $\Dm(\spigolose{\Dm(A),a}) \supseteq \Dm(\Dm(A))=\Dm(A)$; this implies $\spigolose{\Dm(\spigolose{\Dm(A),a}),a} \supseteq \spigolose{\Dm(A), a}$, i.e.\ $\T(A,a) \in \tort$. Now, we prove that $\F(A,a)=(A/\Dm(A), [a]) \in \torf$, i.e.\ $\Dm(A/\Dm(A))=\graffe{0}$. To do this, consider an element $[x] \in \Dm(A / \Dm(A))$. There exists a sequence of elements $\suc{[x_n]}$ such that $[x_0] = [x]$ and $[x_n] = m[x_{n+1}]$. Thus, for each $n \in \mathbb{N}$, there exists $d_{n+1} \in \Dm(A)$ such that $x_n = mx_{n+1} + d_{n+1}$. For simplicity, given $d \in \Dm(A)$, we denote by $\frac{d}{m}$ an element of $\Dm(A)$ such that $m\frac{d}{m} = d$. Moreover, for $n \in \mathbb{N}$, $\frac{d}{m^{n+1}}$ will denote an element such that $m\frac{d}{m^{n+1}} = \frac{d}{m^n}$. We then define $y_0 \coloneqq x$ and $y_n \coloneqq x_n + \sum_{i=1}^n \frac{d_i}{m^{n-i+1}}$. Observe that $my\ped{n+1}=mx\ped{n+1}+ m\sum\ped{i=1}^{n+1}\frac{d_i}{m^{n-i+2}}=mx\ped{n+1}+d\ped{n+1}+\sum_{i=1}^n\frac{d_i}{m^{n-i+1}}=x_n+\sum_{i=1}^n\frac{d_i}{m^{n-i+1}}=y_n$. Therefore, $x \in \Dm(A)$, and thus $[x] = 0$. We want then to show that the sequence
			\[ (\spigolose{\Dm(A),a},a) \xrightarrow{t\ped{(A,a)}} (A,a) \xrightarrow{\eta\ped{(A,a)}} (A/\Dm(A),[a]) \]
			is exact, where $t\ped{(A,a)}$ is the inclusion of $\spigolose{\Dm(A),a}$ into $A$, and $\eta\ped{(A,a)}$ is the projection of $A$ on $A/\Dm(A)$.
			First, observe that given a torsion object $(X,x)$, the arrow $\chi \colon (X,x) \rightarrow (X/\Dm(X), [x])$ is the maximum quotient of $(X,x)$ in $\zeros$. Indeed, since $X = \spigolose{\Dm(X),x}$, it follows that any element $y$ of $X$ can be written as $y = d + nx$ (with $d \in \Dm(A)$). Therefore, $\chi(d + nx) = n \chi (x) = n [x]$. This implies that $(X/\Dm(X), [x]) = (\spigolose{[x]}, [x]) \in \zeros$. Now, consider an object $(\Zh,1) \in \zeros$ and an arrow $f \colon (X,x) \rightarrow (\Zh, 1)$. We observe that, for every $d \in \Dm(X)$, $f(d) = f(md') = mf(d') = 0$ (where, since $d \in \Dm(A)$, there exists $d'$ such that $d = md'$, and since $h$ divides $m$, $mk = 0$ for every $k \in \Zh$). Thus, by the universal property of the quotient, there exists a unique $\overline{f} \colon \spigolose{[x]} \rightarrow \Zh$ such that $\overline{f} \chi = f$. Moreover, $\overline{f}([x]) = f(x) = 1$. Therefore, we have the following commutative triangle in $\coslice$:
			\[\begin{tikzcd}
				{(X,x)} && {(\spigolose{[x]},[x])} \\
				& {(\Zh, 1).}
				\arrow["\chi", from=1-1, to=1-3]
				\arrow["f"', from=1-1, to=2-2]
				\arrow["{\overline{f}}", from=1-3, to=2-2]
			\end{tikzcd}\]
			Thus, $\chi$ is a maximum quotient in $\zeros$ (and, in particular, $\chi=\eta \ped{(X,x)}$). To conclude, consider the commutative square
			\[\begin{tikzcd}
				{(\spigolose{\Dm(A),a},a)} & {(\spigolose{[a]},[a])} \\
				{(A,a)} & {(A/\Dm(A), [a]),}
				\arrow["\xi", from=1-1, to=1-2]
				\arrow["{t\ped{(A,a)}}"', from=1-1, to=2-1]
				\arrow["\varepsilon", from=1-2, to=2-2]
				\arrow["{\eta\ped{(A,a)}}"', from=2-1, to=2-2]
			\end{tikzcd}\]
			where $\varepsilon$ and $t \ped{(A,a)}$ are the inclusions. For convenience, let us write $t \coloneqq t\ped{(A,a)}$ and $\eta \coloneqq \eta\ped{(A,a)}$. If we prove that the above square is both a pullback and a pushout, we are done. Indeed, since $\varepsilon$ is a monomorphism, it follows that $t = \Zker(\eta)$, and applying Proposition \ref{coker e max quotient} we deduce that $\eta = \Zcoker(t)$. Let us begin by showing that the square is a pullback. Consider an object $(C,c)$ and a pair of arrows $f \colon (C,c) \rightarrow (A,a)$ and $g \colon (C,c) \rightarrow (\spigolose{[a]}, [a])$ such that $\varepsilon g = \eta f$. Take an element $e \in C$ and observe that $\eta f(e) = [f(e)] = g(e) = n[a]$ (for some $n \in \mathbb{N}$). Therefore, $f(e) = d + n a$ with $d \in \Dm(A)$, and thus $f$ factors through $(\spigolose{\Dm(A),a},a)$. Hence the square is a pullback. Moreover, thanks to \cite[Proposition 14]{Bourn protomod}, we can conclude that the square is also a pushout. This ensures that the sequence
			\[ (\spigolose{\Dm(A),a},a) \xrightarrow{t\ped{(A,a)}} (A,a) \xrightarrow{\eta\ped{(A,a)}} (A/\Dm(A),[a]) \]
			is exact.
		\end{proof}
		
		The next step is to verify that the torsion theory we have just introduced satisfies the conditions we considered in the first part of the paper.
		
		\begin{prop}
			$\torsione$ in $\coslice$ satisfies Conditions (M) (and so (M')), (S), and (N).
		\end{prop}
		
		\begin{proof}
			Let us begin by showing that the torsion theory $\torsione$ satisfies Condition (M). Consider $(A,a) \in \tort$ and an exact sequence of the form $(A,a) \xrightarrow{k} (B,a) \xrightarrow{q} (C,c)$ (for simplicity, we assume that $k$ is the inclusion). Recall that $\xi \colon (A,a) \rightarrow (\spigolose{[a]}, [a])$ is a maximum quotient in $\zeros$ (where the square brackets denote the classes modulo the quotient by $\Dm(A)$). Therefore, $q$ is defined by the following pushout:
			\[\begin{tikzcd}
				{(A,a)} & {(\spigolose{[a]},a)} \\
				{(B,a)} & {(B/\Dm(A), [a]),}
				\arrow["\xi", from=1-1, to=1-2]
				\arrow["k"', from=1-1, to=2-1]
				\arrow["i", from=1-2, to=2-2]
				\arrow["q"', from=2-1, to=2-2]
				\arrow["\lrcorner"{anchor=center, pos=0, rotate=180}, draw=none, from=2-2, to=1-1]
			\end{tikzcd}\]
			where $i$ is the inclusion. To verify that the above diagram is a pushout, consider a morphism $f \colon (B,a) \rightarrow (E,e)$ and a morphism $g \colon (\spigolose{[a]}, [a]) \rightarrow (E,e)$ such that $fk = g\xi$. It follows that for every $d \in \Dm(A)$, we have $f(d) = fk(d) = g\xi(d) = g([0]) = 0$. Hence, $f$ factors through the quotient, and there exists a morphism $\overline{f} \colon (B/\Dm(A), [a]) \rightarrow (E,e)$. Moreover, $\overline{f}i = g$: take $n[a] \in \spigolose{[a]}$, then $\overline{f}i(n[a]) = \overline{f}([na]) = f(na) = nf(a) = ne = g(n[a])$. The fact that $i$ is a monomorphism ensures that Condition (M) is satisfied.
			Let us now show that $\torsione$ satisfies Condition (S). Consider a torsion object $(A,a)$ and examine the following pullback diagram:
			\[\begin{tikzcd}
				{(P,(a,1))} & {(\Zh, 1)} \\
				{(A,a)} & {(\spigolose{[a]},[a]),}
				\arrow["\chi", from=1-1, to=1-2]
				\arrow[from=1-1, to=2-1]
				\arrow[from=1-2, to=2-2]
				\arrow["\eta"', from=2-1, to=2-2]
			\end{tikzcd}\]
			where $\eta$ is given by the quotient of $A$ over $\Dm(A)$ (i.e.\ it is the $(A,a)$-component of the unit of the reflection given by $\F$). We know that $P = \graffe{(x,k) \in A \times \Zh \, | \, [x] = k[a]}$. Let us prove that $\Dm(P) = \Dm(A) \times \graffe{0}$. On the one hand, given $(x,k) \in \Dm(P)$, there exists a sequence $\suc{(x_n,k_n)}$ in $P$ such that $(x_0,k_0) = (x,k)$ and for every $n \in \mathbb{N}$, $(x_n,k_n) = m(x_{n+1}, k_{n+1})$. Hence, since $mk_n = 0$ for every $n \in \mathbb{N}$ (observe that $h$ divides $m$), the sequence becomes $\suc{(x_n,0)}$. Considering then the sequence $\suc{x_n}$, we deduce that $(x,k) = (x,0) \in \Dm(A) \times\graffe{0}$. On the other hand, given an element $(x,0) \in \Dm(A) \times\graffe{0}$, we know that there exists a sequence $\suc{x_n}$ in $A$ showing that $x$ is in $\Dm(A)$. Therefore, considering the sequence $\suc{(x_n,0)}$, we observe that this sequence proves that $(x,0)$ belongs to $\Dm(P)$ (noting that for every $n \in \mathbb{N}$, we have $[x_n] = 0$, since $x_n \in \Dm(A)$). Let us now show that $P = \spigolose{\Dm(A) \times\graffe{0}, (a,1)}$. Given $(x,k) \in P$, we know that $[x] = k[a]$ and also that $x \in A = \spigolose{\Dm(A), a}$; hence, $x = d + l a$ for some $d \in \Dm(A)$ and $l \in \mathbb{N}$. It follows that $l[a] = k[a]$, and thus we can write $(x,k) = (d + l a, k) = (d,0) + (l a, k)$ where both summands belong to $P$. Now, since $l[a] = k[a]$, we deduce that $l - k = g h'$, with $g$ some natural number and $h' = \ord([a])$. Therefore, $l = k + g h'$ and $[g h' a] = g(h' [a]) = g 0 = 0$, i.e., $g h' a \in \Dm(A)$. Then, observe that $(x,k) = (d,0) + (g h' a, 0) + k (a,1)$ (where all three summands belong to $P$). This shows that $P = \spigolose{\Dm(P), (a,1)}$, i.e., $(P, (a,1)) \in \tort$.
			Furthermore, since $\chi^{-1}(0) = \graffe{(x,0) \in P} = \Dm(A) \times\graffe{0}$, it follows that $\chi$ is the quotient of $P$ by $\Dm(P)$, and thus Condition (S) holds. Let us finally show that Condition (N) holds. Given an arrow $f \colon (A,a) \rightarrow (B,b)$, we define $K \coloneqq \spigolose{\invf{0}, a}$ and consider the sequence $(\spigolose{\Dm(K),a},a) \xrightarrow{t} (K,a) \xrightarrow{k} (A,a) \xrightarrow{f} (B,b)$, where $k = \Zker(f)$ (see Proposition \ref{pullback in coslice}) and $t = t\ped{(K,a)}$. We want to show that $kt$ is a $\zeros$-kernel. Define $q \colon (A,a) \rightarrow (A/\Dm(K), [a])$ as the projection onto the quotient. By Proposition \ref{pullback in coslice}, we know that $\Zker(q)$ is the inclusion of $(\spigolose{q^{-1}(0), a},a)$ into $(A,a)$. Since $q^{-1}(0) = \Dm(K)$, it follows that this inclusion is $kt$. Therefore, $kt$ is a $\zeros$-kernel, and thus Condition (N) holds.
		\end{proof}
		
		\begin{prop}
			The functor $\F \colon \coslice \rightarrow \torf$ associated with the torsion theory $\torsione$ is protoadditive.
		\end{prop}
		
		\begin{proof}
			Recall that $\F(A,a) = (A/\Dm(A), [a])$. Consider the pullback diagram
			\[\begin{tikzcd}
				{(E,(a,c))} & {(C,c)} \\
				{(A,a)} & {(B,b)}
				\arrow["{\pi_C}", from=1-1, to=1-2]
				\arrow["{\pi_A}"', from=1-1, to=2-1]
				\arrow["\lrcorner"{anchor=center, pos=0.125}, draw=none, from=1-1, to=2-2]
				\arrow["g", from=1-2, to=2-2]
				\arrow["f"', shift right, from=2-1, to=2-2]
				\arrow["s"', shift right, from=2-2, to=2-1]
			\end{tikzcd}\]
			where $fs=id\ped{(B,b)}$. Now, construct the pullback of $\F(f)$ along $\F(g)$
			\[\begin{tikzcd}
				{(E/\Dm(E),[a,c])} \\
				& {(H,([a],[c]))} & {(C/\Dm(C),[c])} \\
				& {(A/\Dm(A),[a])} & {(B/\Dm(B),[b])}
				\arrow["{\varphi}", from=1-1, to=2-2]
				\arrow["{\F(\pi_C)}", curve={height=-12pt}, from=1-1, to=2-3]
				\arrow["{\F(\pi_A)}"', curve={height=12pt}, from=1-1, to=3-2]
				\arrow["{p_C}", from=2-2, to=2-3]
				\arrow["{p_A}"', from=2-2, to=3-2]
				\arrow["\lrcorner"{anchor=center, pos=0.125}, draw=none, from=2-2, to=3-3]
				\arrow["{\F(g)}", from=2-3, to=3-3]
				\arrow["{\F(f)}"', shift right, from=3-2, to=3-3]
				\arrow["{\F(s)}"', shift right, from=3-3, to=3-2]
			\end{tikzcd}\]
			and consider the induced arrow $\varphi \colon (E/\Dm(E),[a,c]) \to (H,([a],[c]))$, which is given by $\varphi([x,y])=([x],[y])$. We first show that $\varphi$ is surjective. Consider $([x], [y]) \in H$ such that $\F(f)([x]) = \F(g)([y])$. Then, $[f(x)] = [g(y)]$, which implies $f(x) - g(y) = d \in \Dm(B)$. This leads to $s(d) \in \Dm(A)$. Notice that $(x - s(d), y) \in E$: $f(x-s(d))=f(x)-fs(f(x)-g(y))=f(x)-f(x)+g(y)=g(y)$. Finally, we have $\varphi[x - s(d), y] = ([x - s(d)], [y]) = ([x], [y])$ (the last equality holds since $s(d) \in \Dm(A)$). Therefore, $\varphi$ is surjective. Now, let us prove that $\varphi$ is injective. Let $(x, y) \in E$ be such that $[x] = 0$ and $[y] = 0$. This implies $x \in \Dm(A)$ and $y \in \Dm(C)$. Moreover, since $(x, y) \in E$, we know that $f(x) = g(y)$. We want to show that $[x, y] = 0$. Since $x \in \Dm(A)$ and $y \in \Dm(C)$, there exist sequences $\suc{x_n} \subseteq A$ and $\suc{y_n} \subseteq C$ such that $x = x_0$, $y = y_0$, and for every $n \in \mathbb{N}$, $x_n = mx_{n+1}$ and $y_n = my_{n+1}$. For each $n \in \mathbb{N}$, we define $d_n \coloneqq f(x_n) - g(y_n)$. It is immediate to observe that for every $n \in \mathbb{N}$, $d_n = md_{n+1}$. We also note that $(x_n - s(d_n), y_n) \in E$ for every $n \in \mathbb{N}$. Indeed, $f(x_n-s(d_n))=f(x_n)-fs(f(x_n)-g(y_n))=g(y_n)$. Recalling that $d_0 = f(x_0) - g(y_0) = f(x) - g(y) = 0$, we observe that the sequence $\suc{(x_n - s(d_n), y_n)}$ proves that $(x, y) \in \Dm(E)$. Therefore, $[x, y] = 0$ and so $\varphi$ is injective.
		\end{proof}
		Also in this case, we conclude by studying the factorization system and the Galois structure induced by the torsion theory. Consider an arrow $f \colon (A,a) \to (B,b)$ in $\coslice$, and observe that the factorization
		\[\begin{tikzcd}
			{(A,a)} && {(B,b)} \\
			& {(A/\Dm(\spigolose{\invf{0},a}),[a]),}
			\arrow["f", from=1-1, to=1-3]
			\arrow[two heads, "e"', from=1-1, to=2-2]
			\arrow["m"', from=2-2, to=1-3]
		\end{tikzcd}\]
		where $e$ is the quotient map and $m$ is induced by the quotient, is precisely our $(\facte, \factm)$-factorization. Indeed, it suffices to recall that $\Zker(f) = (\spigolose{\invf{0},a}, a)$ and that the torsion part of $\Zker(f)$ is $(\spigolose{\Dm(\spigolose{\invf{0},a}),a},a)$. To simplify the notation, we observe that $\Dm(\spigolose{\invf{0},a}) = \Dm(\invf{0})$. On the one hand, $\Dm(\invf{0}) \subseteq \Dm(\spigolose{\invf{0},a})$. On the other hand, if we consider $h + na \in \Dm(\spigolose{\invf{0},a})$ (with $f(h) = 0$ and $n \in \mathbb{Z}$), we know that there exist $h_1 \in \invf{0}$ and $n_1 \in \mathbb{N}$ such that $h + na = m(h_1 + n_1 a) = mh_1 + n_1(ma) = mh_1$ (since $ma = 0$). Therefore, $h + na \in \Dm(\invf{0})$. So, we can rewrite the factorization of $f$ as follows:
		\[\begin{tikzcd}
			{(A,a)} && {(B,b)} \\
			& {(A/\Dm(\invf{0}),[a]).}
			\arrow["f", from=1-1, to=1-3]
			\arrow[two heads, "e"', from=1-1, to=2-2]
			\arrow["m"', from=2-2, to=1-3]
		\end{tikzcd}\]
		
		Finally, we know that a surjective morphism in the coslice category is a central extension for the structure $\Gamma_{\torf}$ if and only if its $\zeros$-kernel belongs to $\torf$. In other words, this is equivalent to requiring that $\Dm(\spigolose{\invf{0},a}) = \graffe{0}$. However, as we have just seen, this is in turn equivalent to asking that $\Dm(\invf{0}) = \graffe{0}$.
		
		\section*{Acknowledgements}
		Both authors are members of the Gruppo Nazionale per le Strutture Algebriche, Geometriche e le loro Applicazioni (GNSAGA) dell'Istituto Nazionale di Alta Matematica ``Francesco Severi''.
		
		This work was supported by Shota Rustaveli National Science Foundation of Georgia (SRNSFG), through grant FR-24-9660.

	\end{document}